\documentclass[11pt,reqno,a4paper]{amsart}

\usepackage[T1]{fontenc}
\usepackage[utf8]{inputenc}
\usepackage{lmodern}
\usepackage{microtype}
\usepackage{geometry}
\usepackage{amsmath,amssymb,amsthm,mathtools,mathrsfs}
\usepackage{enumitem}
\usepackage[numbers,sort&compress]{natbib}
\usepackage[colorlinks=true,linkcolor=blue,citecolor=blue,urlcolor=blue]{hyperref}
\usepackage{aliascnt}
\usepackage[nameinlink,noabbrev]{cleveref}

\allowdisplaybreaks
\setlist[itemize]{topsep=4pt,itemsep=2pt,parsep=1pt}
\setlist[enumerate]{topsep=4pt,itemsep=2pt,parsep=1pt}
\numberwithin{equation}{section}

\newtheorem{theorem}{Theorem}[section]
\newaliascnt{proposition}{theorem}
\newtheorem{proposition}[proposition]{Proposition}
\aliascntresetthe{proposition}
\newaliascnt{lemma}{theorem}
\newtheorem{lemma}[lemma]{Lemma}
\aliascntresetthe{lemma}
\newaliascnt{corollary}{theorem}
\newtheorem{corollary}[corollary]{Corollary}
\aliascntresetthe{corollary}
\newaliascnt{assumption}{theorem}
\newtheorem{assumption}[assumption]{Assumption}
\aliascntresetthe{assumption}
\theoremstyle{definition}
\newaliascnt{definition}{theorem}
\newtheorem{definition}[definition]{Definition}
\aliascntresetthe{definition}
\theoremstyle{remark}
\newaliascnt{remark}{theorem}
\newtheorem{remark}[remark]{Remark}
\aliascntresetthe{remark}

\Crefname{theorem}{Theorem}{Theorems}
\Crefname{proposition}{Proposition}{Propositions}
\Crefname{lemma}{Lemma}{Lemmas}
\Crefname{corollary}{Corollary}{Corollaries}
\Crefname{assumption}{Assumption}{Assumptions}
\Crefname{definition}{Definition}{Definitions}
\Crefname{remark}{Remark}{Remarks}

\newcommand{\R}{\mathbb R}
\newcommand{\C}{\mathbb C}
\newcommand{\Z}{\mathbb Z}
\newcommand{\N}{\mathbb N}
\newcommand{\T}{\mathbb T}
\newcommand{\E}{\mathbb E}
\newcommand{\Prob}{\mathbb P}
\newcommand{\dd}{\mathrm d}
\newcommand{\e}{\mathrm e}
\newcommand{\ii}{\mathrm i}
\newcommand{\one}{\mathbf 1}
\newcommand{\Id}{\operatorname{Id}}
\newcommand{\Ran}{\operatorname{Ran}}
\newcommand{\supp}{\operatorname{supp}}
\newcommand{\Rea}{\operatorname{Re}}
\newcommand{\spanop}{\operatorname{span}}
\newcommand{\cL}{\mathcal L}
\newcommand{\cK}{\mathcal K}

\newcommand{\cE}{\mathcal E}
\newcommand{\cD}{\mathcal D}
\newcommand{\cA}{\mathcal A}
\newcommand{\cS}{\mathcal S}

\newcommand{\cV}{\mathcal V}

\newcommand{\cR}{\mathcal R}
\newcommand{\norm}[1]{\left\lVert #1\right\rVert}
\newcommand{\abs}[1]{\left\lvert #1\right\rvert}

\newcommand{\la}{\langle}
\newcommand{\ra}{\rangle}
\newcommand{\law}{\mathrm{Law}}

\title[Polynomial mixing for cubic SNLS]{Polynomial mixing for the 3D damped cubic nonlinear Schr\"odinger equation with degenerate noise}

\author{Rongchang Liu}
\address{School of Mathematics, Sichuan University, Chengdu 610064, China}
\email{rcliu@scu.edu.cn}

\author{Kening Lu}
\address{School of Mathematics, Sichuan University, Chengdu 610064, China}
\email{keninglu@scu.edu.cn}

\author{Lin Shi}
\address{ School of Mathematical Sciences, 
University of Electronic Science and Technology of China,
Chengdu, Sichuan 611731, China}
\email{shilinlavender@163.com}

\date{}


\thanks{This work was supported by the Fundamental Research Funds for
the Central Universities and the National Natural Science Foundation of China
(Grant Nos.~12090010, 12090013 and 12471154).}

\begin{document}

\begin{abstract}
We prove  polynomial mixing for the defocusing damped cubic stochastic nonlinear Schr\"odinger equation on the three-dimensional torus under saturating smooth finite rank Brownian forcing.  The mixing rate is measured in the $p$-Wasserstein metric induced by the $H^1$ distance for every $1\le p<\infty$. We also obtain sharp geometric characterizations of saturation.

The proof is based on a polynomial mixing criterion built on a stable--compact decomposition of the exact solution differences with
polynomial moment control of the logarithmic path amplification. Dense Malliavin range allows the compact defect to be compensated by
finite-dimensional Cameron--Martin shifts, producing a block
multiplier with negative mean logarithm.  A logarithmic transportation
gauge, combined with a renewal--reset coupling scheme, then yields
mixing at every prescribed polynomial order in the weaker $L^2$ distance. A stationary regularity gain to $H^{2-}$ then enables us to upgrade the convergence to $H^1$. 
\end{abstract}

\maketitle
\tableofcontents
\section{Introduction}

The cubic nonlinear Schr\"odinger equation (NLS) occupies a distinguished
place in the statistical theory of weak wave turbulence.  In the celebrated
Kolmogorov--Zakharov cascade picture, nonlinear interactions transfer
conserved quantities across scales, while under the weakly nonlinear kinetic scaling, the
wave spectrum evolves according to the wave kinetic equation, whose
nonequilibrium constant-flux stationary solutions are the 
Kolmogorov--Zakharov spectra
\cite{Kolmogorov1941,Frisch1995,ZakharovLvovFalkovich1992,
NewellRumpf2011}.  For the unforced cubic NLS, the wave kinetic equation has recently been
rigorously derived \cite{DengHani2023,DengHani2026},  providing a rigorous kinetic framework for the mathematical study of
the Kolmogorov--Zakharov spectra and cascade laws.

In experiments and numerical studies, stationary cascades are typically
realized in a forced--dissipated setting, where forcing is concentrated in
a narrow band around a low frequency injection scale and separated in scale from
the dissipative sinks, allowing nonlinear interactions to sustain the nonequilibrium stationary flux underlying
the Kolmogorov--Zakharov cascade picture. Damped--driven stochastic
NLS equations provide a natural microscopic model for this regime, and
the corresponding kinetic description has recently been rigorously derived in
\cite{GrandeHani2026}, see also \cite{KuksinMaiocchi2015,DymovKuksin2021} for earlier developments.

A complementary approach to the same statistical problem is to start
from the underlying dynamics itself and ask whether the evolution
selects a distinguished invariant probability measure with strong
mixing properties.  Such a measure would provide the natural dynamical
object from which the nonequilibrium flux laws predicted by the
Kolmogorov--Zakharov theory might ultimately be recovered under suitable
scaling limits \cite{Bedrossian2024}.  This dynamical viewpoint, rooted
in Kolmogorov's ergodic program and developed further through the work
of Sinai and Ruelle, regards invariant measures, mixing, Lyapunov
exponents, and entropy as intrinsic statistical objects of the dynamics
\cite{Sinai1989,Ruelle1978,EckmannRuelle1985}.  For damped--driven
stochastic NLS, a fundamental first question is whether randomness
injected into only a few low modes can be propagated by nonlinear
interactions so as to determine a unique mixing stationary statistical state.

We address this problem for the
defocusing damped cubic stochastic nonlinear Schr\"odinger equation (SNLS) on
the three-dimensional torus 
$\T^3=(\R/2\pi\Z)^3$,
\begin{align*}
 \dd u_t=-\bigl(\gamma u_t-\ii\Delta u_t
 +\ii\abs{u_t}^2u_t\bigr)\,\dd t+\sum_{j=1}^m b_j\dd W_t^j
\end{align*}
where $\gamma>0$, $W=(W^1,\ldots,W^m)$ is a standard real Brownian
motion, and $b_j\in C^\infty(\T^3;\C)$.

Our main result shows that, for any value of the damping intensity $\gamma>0$, cubic saturation suffices to imply that the Markov semigroup
admits a unique invariant probability measure $\mu$ with full support in
$H^1$ and concentrated in $H^{2-}$, and converges to $\mu$ at every prescribed
polynomial rate in the $p$-Wasserstein distance induced by the $H^1$
norm, for every $1\le p<\infty$.  We also obtain sharp geometric
characterizations of cubic saturation. In particular, we construct a saturating family of three smooth complex-valued noise profiles. The details are deferred to \Cref{sec:setting}.

A common theme behind several existing approaches to mixing for
dissipative SPDEs is that the obstruction to contraction is effectively
finite-dimensional.  For hypoelliptic  SPDEs of parabolic type, smoothing and
dissipation reduce the uncontrolled dynamics to finitely many effective
directions, while nonlinear propagation of degenerate noise controls
the remaining obstruction
\cite{HairerMattingly2006,HairerMattingly2008,HairerMattingly2011}.
Subsequent controllability and compactness-based approaches
\cite{KuksinNersesyanShirikyan2020,Nersesyan2022,
ChenXiangZhangZhao2025,ChenXiangZhang2026} demonstrate in several settings that this
obstruction need not be tied to a prescribed Fourier splitting.
These developments motivate treating the obstruction to contraction as a compact defect accessible to the stochastic directions. This viewpoint
led in our previous work \cite{LiuLuWave2026} to an abstract
stable--compact spectral gap criterion with an application to the
damped cubic wave equation.

In this work, building on these developments and the stable--compact viewpoint in
\cite{LiuLuWave2026}, we develop an abstract polynomial mixing criterion for Markov systems driven by Gaussian noise in a substantially weaker quantitative regime, thereby
placing the mixing mechanisms for hypoelliptic SPDEs of parabolic,
hyperbolic, and dispersive type within a common structural perspective.  A key principle shared by these approaches and captured by the present criterion is that the analytic mechanism producing compactness can be separated from the probabilistic mechanism compensating it. Indeed, the approaches mentioned above can all be viewed as exhibiting a
common structure of contraction modulo a compact defect. On each Lyapunov core, compactness reduces the obstruction to finitely
many effective directions.  In the fixed accuracy regime of the present
criterion, qualitative density of the stochastic directions is then
sufficient for their compensation.

The quantitative mixing regime is then governed by the strength of the
local coupling that remains after this compensation.  In
\cite{LiuLuWave2026}, positive power control of the path amplification,
together with a separate high energy contraction, yields contraction
in a power transportation metric and hence a spectral gap.  A
comparably strong local contraction arises in the parabolic and
bounded noise settings of
\cite{HairerMattingly2008,KuksinNersesyanShirikyan2020,
ChenXiangZhangZhao2025,ChenXiangZhang2026}, where the corresponding
role is played instead by parabolic smoothing and Foia\c{s}--Prodi
splitting, or by bounded noise controllability. The present work lies in a genuinely weaker quantitative regime: for the unbounded Gaussian forcing considered here, the available
energy--Strichartz estimates yield finite moments of the
logarithm of the path amplification.  We show that a negative mean
logarithmic distance multiplier nevertheless yields local contraction
in a carefully designed logarithmic transportation gauge, which a renewal--reset argument
globalizes to mixing at every prescribed polynomial order. Thus the
mixing scale is determined not by the analytic origin of the compact
obstruction, but by the quantitative strength of the compensated local
dynamics.

For the SNLS, the relevant compactness is encoded in the Bogoliubov operator structure of the exact difference equation. Retaining the diagonal real potential in the unitary part of the evolution, a Sylvester argument shows that the conjugate interaction contributes a compact defect, yielding an exact stable--compact structure at finite separation.  This is related in spirit to the asymptotic
compactness of the linearized dynamics used recently for one-dimensional
bounded noise NLS
\cite{ChenXiangZhangZhao2025,ChenXiangZhang2026}, but here the
compactness is produced directly by the operator structure of the
Bogoliubov equation and holds for exact solution differences.  The linearized flow
has the analogous structure, while the exact difference formulation here avoids the Taylor remainder in the state variable.

The stochastic counterpart is the propagation of the finite-rank
forcing through the cubic nonlinearity at the energy regularity.
After factoring out the invertible tangent flow, two adapted
Hilbert-valued covariations and a frozen short-time expansion generate
the cubic descendants in the reduced Malliavin range.  This construction
is related to the invertible tangent flow framework of
\cite{BaudoinTeichmann2005} and to the quadratic variation propagation
mechanisms of
\cite{MattinglyPardoux2006,BakhtinMattingly2007}, but is carried out
directly at the three-dimensional energy regularity.  The
compact obstruction viewpoint lowers the required hypoellipticity:
cubic saturation need only yield qualitative density of the endpoint
Malliavin range, rather than a quantitative lower bound on the
Malliavin covariance.  This also replaces the prescribed
determining mode nondegeneracy used in the classical coupling approach
to stochastic NLS
\cite{DebusscheOdasso2005,GuoLiu2025}.  The same saturation algebra
generates the short-time deterministic controls needed for uniform
irreducibility, in the spirit of
\cite{AgrachevSarychev2005,AgrachevSarychev2006,Sarychev2012,
GlattHoltzHerzogMattingly2018}.

A final ingredient is a stationary regularity gain arising from the dispersive structure of the equation.  Every invariant measure
is concentrated on $H^{2-}$, almost one derivative above the energy space. Related stationary regularization results were obtained in
\cite{ChenXiangZhang2026,GlattHoltzMartinezRichards2026}. Here, the mechanism is the four-wave resonance geometry: a stationary
four-point It\^o identity and the associated damped resonance multiplier
control the nonlinear Sobolev flux.  This structure is closely related
to resonant normal form and modified energy decompositions for
deterministic NLS
\cite{CollianderKeelStaffilaniTakaokaTao2008,CollianderKwonOh2012}
and to stationary correlation expansions in forced wave systems
\cite{RosenhausSmolkin2023}, but here the resonance
resolvent is generated directly by stationarity and is used to control
stationary Sobolev flux rather than a finite-time nonlinear remainder.
This stationary regularity result is of independent interest and, in
the mixing argument, provides precisely the high-frequency control
needed to upgrade the $L^2$ coupling to the
$p$-Wasserstein metric in $H^1$.

We conclude this discussion with a few further connections to the
stochastic NLS literature.  Existence of invariant measures for damped
stochastic NLS with additive forcing has been established in a variety
of settings
\cite{Kim2006,EkrenKukavicaZiane2017,BrzezniakFerrarioZanella2024}. On $\R^d$, $d\le3$, Brze\'zniak, Ferrario, and Zanella
\cite{BrzezniakFerrarioZanella2023} proved uniqueness under
sufficiently large damping for Hilbert--Schmidt additive noise, and
Nguyen and Seong \cite{NguyenSeong2025} subsequently obtained mixing
at every polynomial order under additional regularity of the noise.
These results require neither saturation nor nondegeneracy of the
forcing, since contraction is supplied instead by sufficiently strong
damping.

\section{Settings and main results}\label{sec:setting}
In this section, we give necessary functional settings and state the main result on polynomial mixing of SNLS. The abstract criterion on polynomial mixing with its proof is deferred to the next section.  The defocusing nonlinear Schr\"odinger equation we consider is 
\begin{align}\label{eq:intro-spde}
 \dd u_t=-\bigl(\gamma u_t-\ii\Delta u_t+\ii\abs{u_t}^2u_t\bigr)\,\dd t+B\,\dd W_t, \quad \text{ on }\mathbb T^3,
\end{align}
where $\gamma>0$ is the constant damping intensity and $W=(W^1,\ldots,W^m)$ is a standard $\R^m$-valued Brownian motion over the probability space $(\Omega,\mathcal F, (\mathcal F_t),\mathbb P)$, and the linear operator $Be_j=b_j,\,1\leq j\leq m$ with each $b_j$ smooth complex valued so that  
\begin{align*}
B\,\dd W_t=\sum_{j=1}^m b_j\,\dd W_t^j.
\end{align*}
For $s\in\R$, with $\Lambda=(I-\Delta)^{1/2}$, put
\begin{align*}
 H^s=H^s(\T^3;\C),\qquad
 \norm u_{H^s}=\norm{\Lambda^su}_{L^2}.
\end{align*}
All Sobolev spaces are regarded over the real scalar field unless otherwise stated.  The real inner product of $L^2$ is
\begin{align*}
 (u,v)_\R=\Rea\int_{\T^3}u\bar v\,\dd x.
\end{align*}
The equation is globally well posed in $H^1$, with trajectories in
$C([0,T];H^1)\cap L^{7/2}(0,T;W^{7/8,7/2})$ almost surely; see, for
example, \cite[Theorem~1.5(ii)]{CheungMosincat2019}.  The additional
linear damping is a bounded dissipative perturbation and does not
affect the standard energy-space theory. The phase space is $X=H^1$, while $H=L^2$ is the coupling space.

The Hamiltonian and its damping dissipation are
\begin{align}\label{eq:hamiltonian-dissipation}
 \cE(u)=\frac12\norm{\nabla u}_{L^2}^2+\frac14\norm u_{L^4}^4, \quad \mathsf D(u)=\norm{\nabla u}_{L^2}^2+\norm u_{L^4}^4.
\end{align}
In particular, $\mathsf D\ge2\cE$.  Since $\T^3$ has finite volume and  $H^1\hookrightarrow L^4$, 
\begin{align}\label{eq:energy-coercivity}
 \norm u_{H^1}^2\le 2\cE(u)+C\cE(u)^{1/2},\qquad
 \cE(u)\le C\bigl(\norm u_{H^1}^2+\norm u_{H^1}^4\bigr).
\end{align}
Thus the energy sublevels are bounded in
$H^1$ and compact in $L^2$.

\subsection{Cubic saturation}
Put $\mathsf N(z)=|z|^2z$ and $\mathsf Jz=\ii z$, and regard
$\mathsf N$ as a smooth map over the real scalar field.  Its third
derivative is the constant real trilinear map (the bar denotes the complex conjugation)
\begin{align}\label{eq:cubic-third-derivative}
 D^3\mathsf N[h,k,\ell]=2\bigl(hk\bar\ell+h\ell\bar k+k\ell\bar h\bigr).
\end{align}
In particular, for real $\phi,b,c$,
\begin{align*}
 \mathsf JD^3\mathsf N[\phi,b,c]=6\ii\phi bc,\qquad \mathsf JD^3\mathsf N[\ii\phi,b,c]=-2\phi bc.
\end{align*}
The following definition is the cubic saturation that will be used in the Malliavin propagation and approximate controllability. 

\begin{definition}[Cubic saturation]
\label{def:cubic-saturation}
Let
\begin{align*}
 \mathscr W_0=\Ran B=\spanop_\R\{b_1,\ldots,b_m\}
\end{align*}
and recursively define
\begin{align}\label{eq:cubic-recursion}
 \mathscr W_{n+1}=\spanop_\R\Bigl(
 \mathscr W_n\cup\{\mathsf JD^3\mathsf N[\phi,h,k]:
 \phi\in\mathscr W_n,\ h,k\in\mathscr W_0\}\Bigr).
\end{align}
The forcing is called \emph{cubic saturating} if
\begin{align*}
 \overline{\mathscr W_\infty}^{\,H^1}=H^1,
 \qquad \text{ where } \mathscr W_\infty=\bigcup_{n\ge0}\mathscr W_n.
\end{align*}
\end{definition}

\begin{remark}
Geometric characterizations of cubic saturation are given in
\Cref{subsec:saturation-geometry}.  In the complex phase-complete
Fourier class, saturation is characterized by generation of the full
difference lattice, with four wave vectors being minimal, corresponding
in this class to eight real forcing directions.  For instance, as in
\Cref{cor:fourier-minimal}, one may take
\begin{align*}
 \Ran_\R B
 =\spanop_\R\{1,\ii,\e^{\ii x_1},\ii\e^{\ii x_1},
 \e^{\ii x_2},\ii\e^{\ii x_2},
 \e^{\ii x_3},\ii\e^{\ii x_3}\}.
\end{align*}
Beyond the phase-complete Fourier class, more general smooth
complex-valued forcing may saturate with only three profiles; for
example, \eqref{eq:three-complex-profiles} gives, for any $R>r>0$,
\begin{align*}
 b_1(x)=1,\qquad
 b_2(x)=(R+r\cos x_3)\e^{\ii x_1},\qquad
 b_3(x)=(R+r\sin x_3)\e^{\ii x_2}.
\end{align*}
The main result remains valid on each invariant Fourier sector, with
saturation understood relative to that sector.  When the associated
difference lattice has rank $r=1$ or $2$, the restricted dynamics are
naturally identified with an $r$-dimensional damped cubic NLS with a
constant positive-definite quadratic dispersion, for which the same
polynomial mixing conclusion holds.
\end{remark}

\subsection{Main theorem}
Given $1\le p<\infty$ and a Polish space $X$ with metric $d$, let $\mathcal P_p(X)$ denote the set of probability measures on $X$ with finite $p$-th moment.  The induced $p$-Wasserstein metric on $\mathcal P_p(X)$ is
\begin{align*}
 \mathcal W_{p,d}(\mu_1,\mu_2)
 =\left(\inf_{\Gamma\in\mathsf C(\mu_1,\mu_2)}
   \int_{X\times X}d(x,y)^p\,\Gamma(\dd x,\dd y)\right)^{1/p},
\end{align*}
where $\mathsf C(\mu_1,\mu_2)$ denotes the set of couplings of $\mu_1$ and $\mu_2$; see, for example, \cite{Chen2004,Villani2008}.  We write $\mathcal W_{p,H^s}$ when $d(u,v)=\norm{u-v}_{H^s}$.  We also set
\begin{align*}
  H^{2-}:=\bigcap_{n\ge2}H^{2-1/n}=\bigcap_{s<2}H^s.
\end{align*}

We now state the main result.
\begin{theorem}
\label{thm:main}
Assume that the noise profiles are smooth and
cubic saturating in the sense of
\Cref{def:cubic-saturation}.  Then the Markov semigroup $P_t$ of \eqref{eq:intro-spde} has a unique
invariant probability measure $\mu$ with  $\supp\mu=H^1$, and 
\begin{align*}
  \int_{H^1}(1+\cE(u))^r\,\mu(\dd u)<\infty, \quad \text{ for every } 0<r<\infty. 
\end{align*}
Furthermore, $\mu(H^{2-})=1$ and for every $1\le p<\infty$ and $q>0$ there are
$M=M(p,q)\in\N$ and $C=C(p,q)<\infty$ such that
\begin{align}\label{eq:main-mixing}
 \mathcal W_{p,H^1}\bigl(P_t(u,\cdot),\mu\bigr)
 \le C\{1+\cE(u)^M\}(1+t)^{-q}.
\end{align}
\end{theorem}

\section{A polynomial  mixing criterion}
\label{sec:abstract-criterion}

In this section we formulate and prove an abstract polynomial mixing
criterion in a form adapted to SPDE applications.  Its four assumptions are
a polynomial Lyapunov structure, a stable--compact structure with polynomial log-control, dense
Malliavin range, and uniform irreducibility. 

Let $X$ be a separable Banach space continuously embedded in a separable Hilbert space $H$, and let $(P_t)_{t\ge0}$ be a Feller Markov semigroup on $X$.  For every $T>0$, let $(E_T,\mathcal H_T,\boldsymbol{\gamma}_T)$ be an abstract Wiener space with Cameron--Martin embedding 
\begin{align*}
\iota_T:\mathcal H_T\hookrightarrow E_T
\end{align*}
and write $\mathcal H_T^{0}=\iota_T^*(E_T^*)\subset\mathcal H_T$, where the Riesz isomorphism is used to identify $\mathcal H_T^*$ with $\mathcal H_T$. 
Since $\iota_T$ is injective, $\mathcal H_T^{0}$ is dense in $\mathcal H_T$.  
We use $\Prob$ and $\E$ for probability and expectation on whatever
probability space carries the random variables under consideration.
When $W$ denotes the canonical Wiener variable on $E_T$ and no other
randomness is present, $\E f(W)$ means integration of $f(\omega),\omega\in E_T$ against $\boldsymbol{\gamma}_T$.

Let 
\begin{align*}
 \boldsymbol\Phi_T:X\times E_T\longrightarrow C([0,T];X),\qquad
 \Phi_T(x,\omega)=\boldsymbol\Phi_T(x,\omega)(T),
\end{align*}
be a Borel pathwise representation of $P_T$, compatible almost surely
with restriction and concatenation of independent Wiener blocks; the
exceptional sets may depend on the starting state.  Conditional on the
past state, each fresh block is independent and avoids its corresponding
exceptional set almost surely; hence these identities hold throughout
every countable block construction below.  Write the Malliavin derivative 
\begin{align*}
 \cA_T(x,\omega)=\cD\Phi_T(x,\omega)\in\cL(\mathcal H_T,H),
\end{align*}
and assume that $(x,\omega)\mapsto\cA_T(x,\omega)h$ is Borel for every $h\in\mathcal H_T^{0}$.  If $Y$ is a Banach space, a smooth cylindrical map $U:E_T\to Y$ means
\begin{align*}
 U(\omega)=\varphi\bigl(\ell_1(\omega),\ldots,\ell_N(\omega)\bigr)
\end{align*}
for some $\ell_1,\ldots,\ell_N\in E_T^*$ and a smooth map $\varphi:\R^N\to Y$.

Let $\cV:X\to[1,\infty)$ be measurable and locally bounded, and put
\begin{align*}
 \mathbb V_R=\{x\in X:\cV(x)\le R\},\qquad V_p=1+\cV^p.
\end{align*}
For two path inputs set
\begin{align*}
 \mathbf d_T\bigl((x,\omega),(y,\widetilde\omega)\bigr)
 =\sup_{0\le s\le T}
 \norm{\boldsymbol\Phi_T(x,\omega)(s)-\boldsymbol\Phi_T(y,\widetilde\omega)(s)}_H.
\end{align*}
Throughout this section an exact-marginal coupling means a Borel coupling kernel in the starting pair which can be concatenated at block stopping times with fresh Wiener increments.

The first assumption provides recurrence and moment control together
with compactness of the Lyapunov sublevels in the coupling topology.
\begin{assumption}\label{ass:poly-lyapunov}
The sets $\mathbb V_R$ are compact in the $H$-topology.  For every $p\ge1$, there are $c_p,C_p>0$ such that
\begin{align}\label{eq:poly-lyapunov}
 P_tV_p(x)\le C_p\e^{-c_pt}V_p(x)+C_p,\qquad t\ge0.
\end{align}
\end{assumption}

The second assumption combines the stable--compact endpoint structure
with polynomial control of the logarithmic path amplification.  This
is weaker than the corresponding regularity available for the wave
equation in \cite{LiuLuWave2026}, where the path amplification itself
admits the positive power control needed for a power metric contraction leading to a spectral gap.
\begin{assumption}\label{ass:stable-compact-regularity}
The following properties hold.

\begin{enumerate}[label=\textup{(\roman*)}]
\item For every sufficiently large $T$, there are strongly measurable $S_{T,x,y,\omega},K_{T,x,y,\omega}\in\cL(H)$ such that
\begin{align}\label{eq:abstract-stable-compact}
 \Phi_T(y,\omega)-\Phi_T(x,\omega)
 =\bigl(S_{T,x,y,\omega}+K_{T,x,y,\omega}\bigr)(y-x)
\end{align}
for $x,y\in X$ and $\boldsymbol{\gamma}_T$-almost every $\omega$, where
\begin{align*}
 \norm{S_{T,x,y,\omega}}\le\rho(T),\qquad
 \rho(T)\longrightarrow0,\qquad K_{T,x,y,\omega}\in\cK(H).
\end{align*}

\item For every $T>0$ there is a Borel functional $\mathcal G_T:X\times E_T\to[1,\infty]$ such that, for every $p\ge1$,
\begin{align}\label{eq:path-size-moments}
 \int_{E_T}\mathcal G_T(x,\omega)^p\,\boldsymbol{\gamma}_T(\dd\omega)
 \le C_{T,p}V_{M_{T,p}}(x),\qquad x\in X,
\end{align}
for some $C_{T,p},M_{T,p}<\infty$.  Moreover, there is a polynomial $Q_T$ with nonnegative coefficients such that with $\mathbf r_T = \norm{x-y}_H+\norm h_{\mathcal H_T}$, 
\begin{align}\label{eq:weak-path-growth}
 \mathbf d_T\bigl((x,\omega),(y,\omega+\iota_Th)\bigr)
 \le \mathbf r_T
 \exp\!\left\{Q_T\bigl(\mathcal G_T(x,\omega)+\mathcal G_T(y,\omega+\iota_Th)\bigr)\right\}
\end{align}
whenever the two growth functionals on the right are finite.

\item For every $R<\infty$, sufficiently large $T$, finite-dimensional $F\subset\mathcal H_T^{0}$, at each $(x_0,y_0)\in\mathbb V_R^2$ and for $\boldsymbol{\gamma}_T$-almost every $\omega_0\in E_T$, the map
\begin{align}\label{eq:finite-cm-map}
 (x,y,\omega,h)\longmapsto
 \bigl(K_{T,x,y,\omega},\Phi_T(y,\omega+\iota_Th)\bigr)
\end{align}
is defined and continuous on a product neighborhood in  $\mathbb V_R^2\times E_T\times F$, is continuously Fr\'echet differentiable in $h$, and has derivative
\begin{align*}
 \bigl(0,\cA_T(y,\omega+\iota_Th)|_F\bigr),
\end{align*}
which is jointly continuous there in operator norm. Here $\mathbb V_R$ is equipped with $H$ topology. 
\end{enumerate}
\end{assumption}

The third assumption provides the dense Malliavin directions needed
to compensate the compact endpoint defect by finitely many
Cameron--Martin shifts.  The exceptional Wiener set here is allowed to depend on the initial state, which  is sufficient for the finite factor construction below and is the
form naturally obtained in the SNLS application.  By contrast, a
uniform-in-state exceptional set is available in
\cite{LianLiuLu2026,LiuLuWave2026}.

\begin{assumption}\label{ass:dense-malliavin}
For every sufficiently large $T$ and every $x\in X$,
\begin{align*}
 \overline{\Ran\cA_T(x,\omega)}^{\,H}=H
\end{align*}
for $\boldsymbol{\gamma}_T$-almost every $\omega$.  No exceptional set uniform in $x$ is required.
\end{assumption}

The final assumption supplies a uniform irreducibility mechanism on
Lyapunov sublevels.  It provides the reset step in the renewal coupling.
\begin{assumption}
\label{ass:common-accessibility}
There exist $x_*\in H$ and $R_{*}<\infty$ such that, for every $R<\infty$ and $r>0$, there are $T=T(R,r)>0$ and $p=p(R,r)>0$ for which
\begin{align}\label{eq:common-accessibility}
 \inf_{x\in\mathbb V_R}
 P_T\bigl(x,\mathbb V_{R_*}\cap B_H(x_*,r)\bigr)\ge p.
\end{align}
\end{assumption}

\begin{remark}
In parabolic applications, a convenient stronger form of
\Cref{ass:common-accessibility} is the following: there exists
$x_*\in X$ such that, for every $R<\infty$ and $\varepsilon>0$, one
can find $T=T(R,\varepsilon)>0$ satisfying
\begin{align*}
 \inf_{x\in\mathbb V_R}
 P_T\bigl(x,B_X(x_*,\varepsilon)\bigr)>0.
\end{align*}
\end{remark}

Under these four assumptions we obtain the following mixing criterion.
\begin{theorem}[Polynomially controlled stable--compact mixing criterion]
\label{thm:stable-compact-polynomial-mixing}
Suppose that \Cref{ass:poly-lyapunov,ass:stable-compact-regularity,ass:dense-malliavin,ass:common-accessibility} hold.  Then, for every $q>0$ and $0<\delta\le1$, there are $M=M(q)<\infty$ and $C=C(q,\delta)<\infty$ such that
\begin{align}\label{eq:abstract-mixing-rate}
 \mathcal W_{1\wedge\norm{\cdot-\cdot}_H^\delta}
 \bigl(P_t(x,\cdot),P_t(y,\cdot)\bigr)
 \le C\{V_M(x)+V_M(y)\}(1+t)^{-q}
\end{align}
for all $x,y\in X$ and $t\ge0$.
\end{theorem}

The proof reduces the compact endpoint defect on a Lyapunov core to
finitely many directions and compensates them by Cameron--Martin shifts.
After exact marginal repair, this yields a block multiplier with
negative mean logarithm and finite moments of its positive logarithm.
Lyapunov returns and uniform irreducibility produce a renewal coupling,
while a logarithmic transportation gauge converts this negative drift
into arbitrary-order polynomial mixing.

\subsection{Finite-factor compensation and the negative-log block}

The next lemma turns the compact endpoint defect into a finite-dimensional
Malliavin correction, uniformly on a compact Lyapunov core.  The correction
may moreover be chosen cylindrical, with uniform first order control and
arbitrarily small Gaussian excess outside the good driver set.
\begin{lemma}
\label{lem:finite-factor-compensation}
Assume \Cref{ass:poly-lyapunov,ass:stable-compact-regularity,ass:dense-malliavin}.  Fix $R\ge1$, a sufficiently large block $T$, and $\epsilon,\eta>0$.  Then $\mathbb V_R^2$ admits a finite Borel partition $D_j$ for which there are compact sets $G_j\subset E_T$ with $\boldsymbol{\gamma}_T(G_j)>1-\epsilon$, finite-dimensional spaces $F_j\subset\mathcal H_T^{0}$, and bounded smooth cylindrical fields $\mathscr R_j:E_T\longrightarrow\cL(H,F_j)$
such that
\begin{align}\label{eq:finite-factor-compensation}
 \norm{K_{T,x,y,\omega}-\cA_T(y,\omega)\mathscr R_j(\omega)}_{\cL(H)}\le\eta,
 \qquad (x,y)\in D_j,\quad\omega\in G_j
\end{align}
and on each $D_j\times G_j$, 
\begin{align}\label{eq:finite-factor-compensation-onejet}
 \lim_{s\downarrow0}\ \sup_{\substack{(x,y)\in D_j,\ \omega\in G_j\\
 h\in F_j,\ 0<\norm h_{\mathcal H_T}\le s}}
 \frac{\norm{\Phi_T(y,\omega+\iota_Th)-\Phi_T(y,\omega)-\cA_T(y,\omega)h}_H}{\norm h_{\mathcal H_T}}=0.
\end{align}
Moreover, there is $L=L(R,T,\epsilon,\eta)<\infty$ such that, 
\begin{align}\label{eq:finite-factor-cylindrical-bounds}
 \max_j\sup_{\omega\in E_T}\norm{\mathscr R_j(\omega)}_{\cL(H,F_j)}\le L
\end{align}
and  for every
$\delta>0$, the fields $\mathscr R_j$ may be chosen with the same
$D_j,G_j,F_j$,  to vanish outside cylindrical sets
$\mathsf G_j\supset G_j$  with 
\begin{align*}
\max_j\boldsymbol{\gamma}_T(\mathsf G_j\setminus G_j)<\delta, \text{ and } \max_j\sup_{\omega\in E_T}
\norm{\cD\mathscr R_j(\omega)|_{F_j}}_{\cL(F_j,\cL(H,F_j))}<\infty.
\end{align*}
\end{lemma}

\begin{proof}
The proof is divided into three steps. 

\emph{Step 1: pointwise finite-dimensional compensation with high probability.}
Fix a center $z_0=(x_0,y_0)\in\mathbb V_R^2$. By compactness of $K_{T,z_0,\omega}=K_{T,x_0, y_0,\omega}$, dense range of $\cA_T(y_0,\omega)$ from \Cref{ass:dense-malliavin}, and density of $\mathcal H_T^{0}$ in $\mathcal H_T$, for almost every $\omega$ there is a finite-rank operator  $R_\omega:H\longrightarrow \mathcal H_T^{0}$ with 
\begin{align}\label{eq:PC-pointwise-compensation}
\norm{K_{T,z_0,\omega}
-\cA_T(y_0,\omega)R_\omega}_{\cL(H)}<\eta/8.
\end{align}
Indeed, one can approximate $K_{T,z_0,\omega}$ in operator norm by a finite-rank operator and then approximate each vector in its finite-dimensional range by an element of $\Ran\cA_T(y_0,\omega)$ with preimage in $\mathcal H_T^{0}$.

Moreover, since $H$ and $\mathcal H_T$ are separable, there is a countable operator-norm dense family
\begin{align*}
 \{Q_n:n\ge1\}\subset\mathcal K(H,\mathcal H_T)
\end{align*}
consisting of finite-rank operators with $\Ran Q_n\subset\mathcal H_T^{0}$.  By \eqref{eq:PC-pointwise-compensation} and the boundedness of $\cA_T(y_0,\omega)$, there is $Q_n$ (with $n$ may depend on $\omega$) satisfies
\begin{align*}
 \norm{K_{T,z_0,\omega}-\cA_T(y_0,\omega)Q_n}_{\cL(H)}<\eta/8.
\end{align*}
It follows from \Cref{ass:stable-compact-regularity}  that 
\begin{align*}
 n^*(\omega):=\min\left\{n\ge1:
 \norm{K_{T,z_0,\omega}-\cA_T(y_0,\omega)Q_n}_{\cL(H)}<\eta/8
 \right\},
\end{align*}
is measurable.  Therefore, $A_n:=\{\omega:n^*(\omega)=n\}$
are measurable and pairwise disjoint, and their union has full $\boldsymbol{\gamma}_T$-measure. Choose $N$ so large that
\begin{align*}
 \boldsymbol{\gamma}_T\left(\bigcup_{n=1}^N A_n\right)>1-\epsilon/8,
\end{align*}
and set
\begin{align*}
 F_{z_0}=\sum_{n=1}^N\Ran Q_n\subset\mathcal H_T^{0},
\end{align*}
which is  finite dimensional.  By \Cref{ass:stable-compact-regularity}\textup{(iii)}, applied at the fixed pair $z_0=(x_0,y_0)$ and to $F_{z_0}$, there is a full-measure set $\Omega_{z_0}\subset E_T$ on which the local continuity and Cameron--Martin $C^1$ conclusions of that assumption hold at $(z_0,\omega,0)$.  By inner regularity choose compact sets
\begin{align*}
 G_{z_0,n}\subset A_n\cap\Omega_{z_0},\qquad1\le n\le N,
\end{align*}
so that, with $G_{z_0}=\bigcup_{n=1}^NG_{z_0,n}$,
\begin{align*}
 \boldsymbol{\gamma}_T(G_{z_0})>1-\epsilon/4.
\end{align*}
Thus $G_{z_0}$ is compact and the selected compensator $Q_{n^*}$ takes only the fixed values $Q_1,\ldots,Q_N$ there, all with range contained in $F_{z_0}$.

\emph{Step 2: smooth cylindrical localization.}
We now replace the finite-valued selector $Q_{n^*}$ on $G_{z_0}$ by a smooth cylindrical field without changing its values there.  Fix $\delta>0$.  By inner regularity choose a compact set $G^*_{z_0}\subset E_T\setminus G_{z_0}$ such that
\begin{align*}
 \boldsymbol{\gamma}_T\bigl(E_T\setminus(G_{z_0}\cup G^*_{z_0})\bigr)<\delta.
\end{align*}
The compact sets $G_{z_0,1},\ldots,G_{z_0,N},G^*_{z_0}$ are pairwise disjoint.  Hahn--Banach separation and compactness therefore provide finitely many functionals $\ell_1,\ldots,\ell_M\in E_T^*$ such that the map
\begin{align*}
 \Pi(\omega)=\bigl(\ell_1(\omega),\ldots,\ell_M(\omega)\bigr)
\end{align*}
separates these compact sets: their images under $\Pi$ are pairwise disjoint compact subsets of $\R^M$.  Choose $\chi_n\in C_c^\infty(\R^M;[0,1])$ with $\chi_n=1$ on $\Pi(G_{z_0,n})$ and with support disjoint from the images of $G_{z_0,m}$, $m\ne n$, and of $G^*_{z_0}$.  Put
\begin{align*}
 \mathscr R_{z_0}(\omega)=\sum_{n=1}^N\chi_n(\Pi(\omega))Q_n,
 \qquad
 \mathsf G_{z_0}=\bigcup_{n=1}^N\Pi^{-1}(\supp\chi_n).
\end{align*}
Then $\mathscr R_{z_0}:E_T\to\cL(H,F_{z_0})$ is bounded and smooth cylindrical, equals $Q_n$ on $G_{z_0,n}$, vanishes outside $\mathsf G_{z_0}$, and
\begin{align}\label{eq:excess}
 \boldsymbol{\gamma}_T(\mathsf G_{z_0}\setminus G_{z_0})<\delta.
\end{align}
Since $0\le\chi_n\le1$,
\begin{align}\label{eq:local-compensator-amplitude}
 \sup_{\omega\in E_T}\norm{\mathscr R_{z_0}(\omega)}_{\cL(H,F_{z_0})}
 \le\sum_{n=1}^N\norm{Q_n}_{\cL(H,\mathcal H_T)}
 =:L_{z_0},
\end{align}
independent of $\delta$. Moreover by the chain rule for Malliavin derivative, for $k\in F_{z_0}$,
\begin{align*}
 \cD\mathscr R_{z_0}(\omega)k
 =\sum_{n=1}^N
 \nabla\chi_n(\Pi(\omega))\cdot\Pi(\iota_Tk)\,Q_n,
\end{align*}
and hence
\begin{align}\label{eq:local-compensator-derivative}
 \sup_{\omega\in E_T}
 \norm{\cD\mathscr R_{z_0}(\omega)|_{{F_{z_0}}}}
 _{\cL(F_{z_0},\cL(H,F_{z_0}))}
 \le\norm{\Pi\circ\iota_T|_{F_{z_0}}}
 \sum_{n=1}^N\norm{\nabla\chi_n}_{L^\infty}
 \norm{Q_n}_{\cL(H,\mathcal H_T)}<\infty,
\end{align}
which  may depend on
$\delta$ through the cylindrical localization.

\emph{Step 3: Uniformization and finite covering.}
For every $\omega\in G_{z_0}$, \Cref{ass:stable-compact-regularity}\textup{(iii)}, applied to the finite-dimensional space $F_{z_0}$, gives a product neighborhood of $(z_0,\omega,0)$ on which the map in \eqref{eq:finite-cm-map} and its $F_{z_0}$-derivative are jointly continuous. Since $G_{z_0}$ is compact, finitely many such neighborhoods cover it. Intersecting the corresponding relative state neighborhoods and taking the minimum Cameron--Martin radius, we obtain a relative neighborhood $U_{z_0}$ of $z_0$ and $r_{z_0}>0$ on which these conclusions hold uniformly over $G_{z_0}$.

On $G_{z_0,n}$ one has $\mathscr R_{z_0}(\omega)=Q_n$, and hence, by the definition of $A_n$,
\begin{align*}
\norm{K_{T,z_0,\omega}
-\cA_T(y_0,\omega)\mathscr R_{z_0}(\omega)}_{\cL(H)}
<\eta/8.
\end{align*}
Since $\mathscr R_{z_0}=Q_n$ on $G_{z_0,n}$ and the bound
\eqref{eq:local-compensator-amplitude} is independent of $\delta$,
the continuity in the state variables of $K_T$ and
$\cA_T|_{F_{z_0}}$ allows $U_{z_0}$ to be decreased independently
of $\delta$ so that
\begin{align}\label{eq:uniform-compensation}
\norm{K_{T,z,\omega}
-\cA_T(y,\omega)\mathscr R_{z_0}(\omega)}_{\cL(H)}
\le\eta,
\qquad z=(x,y)\in U_{z_0},\quad\omega\in G_{z_0}.
\end{align}
After decreasing $U_{z_0}$ once more, its closure in $\mathbb V_R^2$ is compact in the $H$-topology. Uniform continuity of the $F_{z_0}$-derivative on the finitely many resulting compact product patches and the fundamental theorem of calculus give a modulus $\varpi_{z_0}(s)\downarrow0$ such that
\begin{align}\label{eq:first-order}
\sup_{\substack{(x,y)\in U_{z_0},\ \omega\in G_{z_0}\\
 h\in F_{z_0},\ 0<\norm h_{\mathcal H_T}\le s}}
\frac{\norm{\Phi_T(y,\omega+\iota_Th)-\Phi_T(y,\omega)-\cA_T(y,\omega)h}_H}
{\norm h_{\mathcal H_T}}
\le\varpi_{z_0}(s),
\end{align}
for $0<s\le r_{z_0}$. 

Choose a smaller relative neighborhood $V_{z_0}$ with $z_0\in V_{z_0}$ and $\overline{V_{z_0}}\subset U_{z_0}$. Compactness of $\mathbb V_R^2$ yields finitely many $V_{z_j}$ covering the core. Assign each pair to its least index and let $D_j$ be the resulting Borel partition. With
\begin{align*}
G_j=G_{z_j},\qquad
F_j=F_{z_j},\qquad
\mathscr R_j=\mathscr R_{z_j},
\end{align*}
\eqref{eq:finite-factor-compensation} follows from
\eqref{eq:uniform-compensation}, while
\eqref{eq:first-order} gives
\eqref{eq:finite-factor-compensation-onejet}.  Since only finitely many
centers are used,
\begin{align*}
 L=\max_jL_{z_j}<\infty,
\end{align*}
and \eqref{eq:local-compensator-amplitude} gives
\eqref{eq:finite-factor-cylindrical-bounds}.  Notice that this $L$ is
independent of $\delta$.  For every fixed $\delta>0$,
\eqref{eq:excess} and \eqref{eq:local-compensator-derivative} give the desired last assertion. 
\end{proof}

The finite factor correction can now be converted into an exact-marginal
coupling with negative mean logarithmic contraction on the Lyapunov core.
The tradeoff for repairing the shifted companion to the true second marginal
is only of order of the initial distance.
\begin{proposition}
\label{prop:negative-log-core}
Assume \Cref{ass:poly-lyapunov,ass:stable-compact-regularity,ass:dense-malliavin}.  For every $R\ge1$ and $\mathsf A>0$, there are $T_c,r_c,q_*>0$ such that every $x,y\in\mathbb V_R$ with $0<r=\norm{x-y}_H\le r_c$ admits a Borel exact-marginal coupling $X,Y$, a companion process $Z$, and an event $\mathcal M$ such that
\begin{align}\label{eq:path-repair}
 Y_s=Z_s\quad(0\le s\le T_c)\quad\text{on }\mathcal M,\qquad \Prob(\mathcal M^c)\le C_{R,\mathsf A}r.
\end{align}
Moreover, with
\begin{align*}
 \mathfrak m_0=q_*\vee\frac{\norm{X_{T_c}-Z_{T_c}}_H}{r},
 \qquad
 \mathfrak m_0^{\sup}=1\vee\sup_{0\le s\le T_c}\frac{\norm{X_s-Z_s}_H}{r},
\end{align*}
one has
\begin{align}\label{eq:core-log-bounds}
 \E\log\mathfrak m_0\le-\mathsf A,\qquad
 \E(\log^+\mathfrak m_0^{\sup})^m\le C_{R,\mathsf A,m},\qquad m\in\N.
\end{align}
\end{proposition}
\begin{proof}
The proof is divided into three steps. 

\emph{Step 1: finite-factor compensation and logarithmic exceptional costs.}
Set $q_1=\e^{-4\mathsf A-5}$.  Choose $T_c$ so large that $\rho(T_c)\le q_1/4$ and put $\eta=q_1/4$.  For $x\ne y$ in $\mathbb V_R$, write $r=\norm{x-y}_H$ and
\begin{align*}
 A^0_{x,y}(\omega)=1\vee\frac{\mathbf d_{T_c}((x,\omega),(y,\omega))}{r}.
\end{align*}
By \eqref{eq:weak-path-growth}--\eqref{eq:path-size-moments}, for every $m<\infty$,
\begin{align*}
 \sup_{\substack{x,y\in\mathbb V_R,\ x\ne y}}\E(\log A^0_{x,y})^m<\infty.
\end{align*}
In particular, we may choose $\epsilon\in(0,1)$ small such that 
\begin{align}\label{eq:synchronous-exception-cost}
(1-\epsilon)\log q_1+\frac12\le-\mathsf A, \quad \text{ and } \sup_{\substack{x,y\in\mathbb V_R,\ x\ne y}}\int_E\log A^0_{x,y}\,\boldsymbol{\gamma}_{T_c}(\dd\omega)\le\frac14
\end{align}
for every Borel set $E\subset E_{T_c}$ with $\boldsymbol{\gamma}_{T_c}(E)\le\epsilon$.

Apply \Cref{lem:finite-factor-compensation} with this $\epsilon$ and
$\eta$, and let $D_j,G_j,F_j$ be the resulting finite-factor data and
$L$ the uniform bound in
\eqref{eq:finite-factor-cylindrical-bounds}.  For any admissible choice
of the cylindrical fields and $(x,y)\in D_j$, put $v=y-x$,
$r=\norm v_H$, and
\begin{align*}
 \widehat U_{x,y}(\omega)=-\mathscr R_j(\omega)\frac vr,\qquad
 \Theta_{x,y}(\omega)=\omega+r\iota_{T_c}\widehat U_{x,y}(\omega).
\end{align*}
Thus $\norm{\widehat U_{x,y}(\omega)}_{\mathcal H_{T_c}}\le L$.
Whenever
\begin{align}\label{eq:scaled-active-derivative}
 r\sup_{\omega\in E_{T_c}}
 \norm{\left.\mathcal D\mathscr R_j(\omega)\right|_{F_j}}
 _{\cL(F_j,\cL(H,F_j))}\le\frac12,
\end{align}
the finite-dimensional Gaussian transformation estimate, see for example \cite{LiuLuWave2026,Bogachev1998}, gives 
\begin{align*}
 \left\|\frac{\dd(\Theta_{x,y})_\#\boldsymbol{\gamma}_{T_c}}
 {\dd\boldsymbol{\gamma}_{T_c}}\right\|_{L^{2}(\boldsymbol{\gamma}_{T_c})}
 \le C_{R,\mathsf A},
\end{align*}
where the constant depends only on the finite-dimensional spaces and
the amplitude bound $L$, and not on the unscaled active derivative.
In particular, every exceptional set which is null for the original
Wiener driver remains null for the shifted driver.
By H\"older's inequality and \eqref{eq:path-size-moments}, every finite
moment of $\mathcal G_{T_c}(y,\Theta_{x,y}(\omega))$ is therefore
uniformly bounded.  Since \eqref{eq:weak-path-growth} gives
\begin{align*}
 \log^+\!\left(1\vee
 \frac{\mathbf d_{T_c}((x,\omega),(y,\Theta_{x,y}(\omega)))}r\right)
 \le\log(1+L)+Q_{T_c}\bigl(\mathcal G_{T_c}(x,\omega)
 +\mathcal G_{T_c}(y,\Theta_{x,y}(\omega))\bigr),
\end{align*}
we obtain, for every $m<\infty$,
\begin{align}\label{eq:shifted-log-bound}
 \sup_j\sup_{(x,y)\in D_j}
 \E\left[\log^+\!\left(1\vee
 \frac{\mathbf d_{T_c}((x,W),(y,\Theta_{x,y}(W)))}
 {\norm{x-y}_H}\right)\right]^m
 \le C_{R,\mathsf A,m},
\end{align}
whenever \eqref{eq:scaled-active-derivative} holds.  Importantly, this
bound is uniform over the cylindrical localizations allowed by
\Cref{lem:finite-factor-compensation}.

Write
\begin{align*}
 \Xi^j_{x,y}(\omega)=\log^+\!\left(1\vee
 \frac{\mathbf d_{T_c}((x,\omega),(y,\Theta_{x,y}(\omega)))}r\right).
\end{align*}
Then there is $C_{\rm sh}<\infty$, independent of the cylindrical
localization, such that $\E(\Xi^j_{x,y})^2\le C_{\rm sh}$
for every admissible localization and every $(x,y)\in D_j$ satisfying
\eqref{eq:scaled-active-derivative}. Choose $\delta>0$ with
$C_{\rm sh}^{1/2}\delta^{1/2}\le1/4$ and fix the corresponding fields
$\mathscr R_j$ and cylindrical support sets $\mathsf G_j$ from
\Cref{lem:finite-factor-compensation}, so that
\begin{align*}
 \supp\mathscr R_j\subset\mathsf G_j,\qquad
 G_j\subset\mathsf G_j,\qquad
 \max_j\boldsymbol{\gamma}_{T_c}(\mathsf G_j\setminus G_j)<\delta.
\end{align*}
Then
\begin{align}\label{eq:shifted-exception-cost}
 \sup_{(x,y)\in D_j}
 \int_{\mathsf G_j\setminus G_j}\Xi^j_{x,y}(\omega)\,
 \boldsymbol{\gamma}_{T_c}(\dd\omega)\le\frac14.
\end{align}
The fields are now fixed, and the derivative bound in
\Cref{lem:finite-factor-compensation} gives
\begin{align*}
 D_*:=\max_j\sup_{\omega\in E_{T_c}}
 \norm{\left.\mathcal D\mathscr R_j(\omega)\right|_{F_j}}
 _{\cL(F_j,\cL(H,F_j))}<\infty.
\end{align*}
Since there are only finitely many $j$,
\eqref{eq:finite-factor-compensation-onejet} yields a common modulus
$\varpi(s)\downarrow0$ such that
\begin{align}\label{eq:uniform-core-onejet}
 \sup_j\sup_{\substack{(x,y)\in D_j,\ \omega\in G_j\\
 h\in F_j,\ 0<\norm h_{\mathcal H_{T_c}}\le s}}
 \frac{\norm{\Phi_{T_c}(y,\omega+\iota_{T_c}h)-\Phi_{T_c}(y,\omega)
 -\cA_{T_c}(y,\omega)h}_H}{\norm h_{\mathcal H_{T_c}}}
 \le\varpi(s).
\end{align}
Choose $r_c\in(0,1]$ so small that
\begin{align}\label{eq:core-radius-choice}
 r_cD_*\le\frac12,\qquad
 L\varpi(Lr_c)\le q_1-\rho(T_c)-\eta.
\end{align}
Thus \eqref{eq:scaled-active-derivative} and all the preceding shifted
estimates hold whenever $0<r\le r_c$.

\emph{Step 2: compensated contraction and the logarithmic bounds.}
Fix $(x,y)\in D_j$ with $0<r=\norm{x-y}_H\le r_c$ and set
\begin{align*}
 h_{x,y}(\omega)=r\widehat U_{x,y}(\omega)
 =-\mathscr R_j(\omega)(y-x).
\end{align*}
Then
\begin{align*}
 \norm{h_{x,y}(\omega)}_{\mathcal H_{T_c}}\le Lr.
\end{align*}
Define
\begin{align*}
 X_s=\boldsymbol\Phi_{T_c}(x,\omega)(s),\qquad
 Z_s=\boldsymbol\Phi_{T_c}
 (y,\omega+\iota_{T_c}h_{x,y}(\omega))(s).
\end{align*}
On $G_j$,
\begin{align*}
 X_{T_c}-Z_{T_c}
 ={}&-\bigl\{\Phi_{T_c}(y,\omega)-\Phi_{T_c}(x,\omega)
 -\cA_{T_c}(y,\omega)\mathscr R_j(\omega)(y-x)\bigr\}\\
 &-\bigl\{\Phi_{T_c}(y,\omega+\iota_{T_c}h_{x,y}(\omega))
 -\Phi_{T_c}(y,\omega)-\cA_{T_c}(y,\omega)h_{x,y}(\omega)\bigr\}.
\end{align*}
By \eqref{eq:abstract-stable-compact},
\eqref{eq:finite-factor-compensation},
\eqref{eq:uniform-core-onejet}, and \eqref{eq:core-radius-choice},
\begin{align*}
 \norm{X_{T_c}-Z_{T_c}}_H
 \le\{\rho(T_c)+\eta+L\varpi(Lr_c)\}r
 \le q_1r.
\end{align*}

Set $q_*=q_1/2$.  Then $\mathfrak m_0\le q_1$ on $G_j$.  On
$\mathsf G_j^c$ one has $\mathscr R_j=0$, so $Z$ is the synchronous
companion and
\begin{align*}
 \log\mathfrak m_0\le\log A^0_{x,y}.
\end{align*}
Since
$\boldsymbol{\gamma}_{T_c}(\mathsf G_j^c)\le\boldsymbol{\gamma}_{T_c}(G_j^c)<\epsilon$,
\eqref{eq:synchronous-exception-cost} gives
\begin{align*}
 \int_{\mathsf G_j^c}\log\mathfrak m_0\,
 \boldsymbol{\gamma}_{T_c}(\dd\omega)\le\frac14.
\end{align*}
On $\mathsf G_j\setminus G_j$ one has
$\log\mathfrak m_0\le\Xi^j_{x,y}$, so
\eqref{eq:shifted-exception-cost} gives another contribution at most
$1/4$.  Consequently,
\begin{align*}
 \E\log\mathfrak m_0
 \le\boldsymbol{\gamma}_{T_c}(G_j)\log q_1+\frac12
 \le(1-\epsilon)\log q_1+\frac12
 \le-\mathsf A.
\end{align*}
Moreover,
\begin{align*}
 \log\mathfrak m_0^{\sup}
 =\log^+\!\left(1\vee
 \frac{\mathbf d_{T_c}((x,W),(y,\Theta_{x,y}(W)))}r\right),
\end{align*}
and \eqref{eq:shifted-log-bound} gives
\begin{align*}
 \E(\log^+\mathfrak m_0^{\sup})^m
 \le C_{R,\mathsf A,m},\qquad m\in\N.
\end{align*}

\emph{Step 3: exact marginal repair.}
Let
\begin{align*}
 \Theta(\omega)=\omega+\iota_{T_c}h_{x,y}(\omega),\qquad
 \nu=\Theta_\#\boldsymbol{\gamma}_{T_c}.
\end{align*}
For $k\in F_j$,
\begin{align*}
 \mathcal Dh_{x,y}(\omega)k
 =-\bigl(\mathcal D\mathscr R_j(\omega)k\bigr)(y-x),
\end{align*}
and hence, writing $d_*=\max_j\dim F_j$,
\begin{align*}
 \sup_\omega\norm{h_{x,y}(\omega)}_{\mathcal H_{T_c}}\le Lr,\qquad
 \sup_\omega\norm{\left.\mathcal Dh_{x,y}(\omega)\right|_{F_j}}_{\rm HS}
 \le\sqrt{d_*}D_*r,
\end{align*}
while
\begin{align*}
 \sup_\omega\norm{\left.\mathcal Dh_{x,y}(\omega)\right|_{F_j}}_{\rm op}
 \le D_*r\le\frac12.
\end{align*}
The finite-dimensional Gaussian transformation estimate therefore gives
\begin{align*}
 \operatorname{Ent}(\nu\mid\boldsymbol{\gamma}_{T_c})\le C_{R,\mathsf A}r^2.
\end{align*}

Let $\mathbf Q_y$ denote the genuine path law issued from $y$ on
$[0,T_c]$.  Since $\nu\ll\boldsymbol{\gamma}_{T_c}$, the Borel
representation agrees $\nu$-almost surely with the genuine solution
issued from $y$.  Therefore
\begin{align*}
 \law(Z)=\boldsymbol\Phi_{T_c}(y,\cdot)_\#\nu,\qquad
 \mathbf Q_y=\boldsymbol\Phi_{T_c}(y,\cdot)_\#\boldsymbol{\gamma}_{T_c}.
\end{align*}
The monotonicity of relative entropy under measurable pushforwards gives
\begin{align*}
 \operatorname{Ent}(\law(Z)\mid\mathbf Q_y)
 \le\operatorname{Ent}(\nu\mid\boldsymbol{\gamma}_{T_c})
 \le C_{R,\mathsf A}r^2.
\end{align*}
Hence Pinsker's inequality and maximal coupling as in the standard generalized-coupling construction
\cite{ButkovskyKulikScheutzow2020}, followed by the gluing
lemma with the joint law of $(X,Z)$, yield a triple $(X,Z,Y)$ such that
$X$ and $Y$ have the genuine path laws issued from $x$ and $y$,
respectively, and
\begin{align*}
 \Prob\{Y\ne Z\}\le C_{R,\mathsf A}r.
\end{align*}
Thus, with $\mathcal M=\{Y=Z\}$,
\begin{align*}
 Y_s=Z_s,\qquad 0\le s\le T_c,
\end{align*}
on $\mathcal M$, and
\begin{align*}
 \Prob(\mathcal M^c)\le C_{R,\mathsf A}r.
\end{align*}
This proves \eqref{eq:path-repair}.

All ingredients of the construction depend Borel measurably on
$(x,y)$.  Standard measurable maximal coupling and gluing on Polish
spaces therefore make the resulting exact-marginal coupling kernel
Borel in $(x,y)$.
\end{proof}

\subsection{Returns and resets}

The following lemma provides moment control for returning to the Lyapunov core  $\{\cV(x)+\cV(y)\le R\}$. 
\begin{lemma}
\label{lem:pair-return}
Assume \Cref{ass:poly-lyapunov,ass:stable-compact-regularity}.  Let $X$ and $Y$ be the synchronously driven trajectories issued from
$x$ and $y$ respectively. For every $m\in\N$, there are $\mathsf h_m\in\N$, $R_m,M_m<\infty$, and $\lambda_m,C_m>0$ such that with
\begin{align*}
 \tau=\inf\{n\ge0:\cV(X_{n\mathsf h_m})+\cV(Y_{n\mathsf h_m})\le R_m\},
\end{align*}
the path multiplier
\begin{align*}
 B_\tau=1\vee\sup_{0\le s\le \mathsf h_m\tau}
 \frac{\norm{X_s-Y_s}_H}{\norm{x-y}_H},\, \text{ with } B_\tau=1 \text{ when }  x=y,
\end{align*}
satisfies
\begin{align}\label{eq:pair-return}
 \E_{x,y}\e^{\lambda_m\tau}
 +\E_{x,y}(\log B_\tau)^m
 \le C_m\{V_{M_m}(x)+V_{M_m}(y)\}.
\end{align}
\end{lemma}

\begin{proof}

Fix $m\in\N$.  Apply \eqref{eq:weak-path-growth} on a unit block with
$h=0$.  Since $Q_1$ is polynomial, \eqref{eq:path-size-moments} gives
$M_*=M_*(m)$ and $C_*=C_*(m)$ such that
\begin{align}\label{eq:one-block-log}
 \E\left[\log\left(1\vee
 \frac{\mathbf d_1((x,W),(y,W))}{\norm{x-y}_H}\right)\right]^m
 \le C_*\{V_{M_*}(x)+V_{M_*}(y)\},
\end{align}
with the ratio interpreted as one when $x=y$. Set
\begin{align*}
 \mathbf V(x,y)=V_{M_*}(x)+V_{M_*}(y).
\end{align*}
Let $Q_{\mathsf h}$ be the transition kernel of the pair obtained by driving the
two solutions from $x$ and $y$ with the same Wiener path over a block
of length $\mathsf h$.  Since its two coordinate marginals are $P_{\mathsf h}(x,\cdot)$
and $P_{\mathsf h}(y,\cdot)$, \eqref{eq:poly-lyapunov} gives
\begin{align*}
 Q_{\mathsf h}\mathbf V(x,y)
 \le a_{\mathsf h}\mathbf V(x,y)+b,\qquad
 a_{\mathsf h}=C_{M_*}\e^{-c_{M_*}\mathsf h},\qquad b=2C_{M_*}.
\end{align*}
Choose an integer $\mathsf h=\mathsf h_m$ so large that $a_{\mathsf h}\le1/4$, and then choose
$R_m$ so large that
\begin{align*}
 b\le\frac14\mathbf V(x,y)
 \qquad\text{whenever}\qquad
 \cV(x)+\cV(y)>R_m.
\end{align*}
Thus, with $\alpha=1/2$,
\begin{align}\label{eq:pair-strict-drift}
 Q_{\mathsf h}\mathbf V(x,y)\le\alpha\mathbf V(x,y)
 \qquad\text{for }\cV(x)+\cV(y)>R_m.
\end{align}

Write
\begin{align*}
 \tau=\inf\{n\ge0:\cV(X_{n\mathsf h})+\cV(Y_{n\mathsf h})\le R_m\}.
\end{align*}
Since $\{\tau>n\}\in\mathcal F_{n\mathsf h}$, \eqref{eq:pair-strict-drift}
and the Markov property give
\begin{align*}
 \E_{x,y}\bigl[
 \mathbf V(X_{(n+1)\mathsf h},Y_{(n+1)\mathsf h})\one_{\{\tau>n+1\}}
 \mid\mathcal F_{n\mathsf h}\bigr]\le
 \one_{\{\tau>n\}}
 Q_{\mathsf h}\mathbf V(X_{n\mathsf h},Y_{n\mathsf h})
 \le\alpha\mathbf V(X_{n\mathsf h},Y_{n\mathsf h})\one_{\{\tau>n\}}.
\end{align*}
Iteration yields
\begin{align}\label{eq:stopped-return-drift}
 \E_{x,y}\bigl[
 \mathbf V(X_{n\mathsf h},Y_{n\mathsf h})\one_{\{\tau>n\}}\bigr]
 \le\alpha^n\mathbf V(x,y).
\end{align}
Since $\mathbf V\ge1$,
\begin{align*}
 \Prob_{x,y}\{\tau>n\}\le\alpha^n\mathbf V(x,y).
\end{align*}
Hence, for any $\lambda_m>0$ such that
$\e^{\lambda_m}\alpha<1$,
\begin{align}\label{eq:return-time-exponential}
 \E_{x,y}\e^{\lambda_m\tau}=1+(\e^{\lambda_m}-1)
   \sum_{n\ge0}\e^{\lambda_m n}\Prob_{x,y}\{\tau>n\}\le C\mathbf V(x,y).
\end{align}

We next control the amplification accumulated before the return.  For
$k<\tau$ and $0\le j<\mathsf h$, set
\begin{align*}
 A_{k,j}=1\vee\sup_{0\le s\le1}
 \frac{\norm{X_{k\mathsf h+j+s}-Y_{k\mathsf h+j+s}}_H}
      {\norm{X_{k\mathsf h+j}-Y_{k\mathsf h+j}}_H},
\end{align*}
with $A_{k,j}=1$ when the denominator vanishes.  By the Markov property
and \eqref{eq:one-block-log},
\begin{align*}
 \E\bigl[(\log A_{k,j})^m\mid\mathcal F_{k\mathsf h+j}\bigr]
 \le C_*\{V_{M_*}(X_{k\mathsf h+j})+V_{M_*}(Y_{k\mathsf h+j})\}.
\end{align*}
Moreover, for $0\le j<\mathsf h$, \eqref{eq:poly-lyapunov} gives
\begin{align*}
 \E\bigl[
 V_{M_*}(X_{k\mathsf h+j})+V_{M_*}(Y_{k\mathsf h+j})
 \mid\mathcal F_{k\mathsf h}\bigr]
 \le C_{\mathsf h}\mathbf V(X_{k\mathsf h},Y_{k\mathsf h}).
\end{align*}
Since $\{\tau>k\}\in\mathcal F_{k\mathsf h}\subset\mathcal F_{k\mathsf h+j}$, the two
conditional estimates imply
\begin{align*}
 \E_{x,y}\bigl[
 \one_{\{\tau>k\}}(\log A_{k,j})^m\bigr]
 &\le C_*
 \E_{x,y}\bigl[
 \one_{\{\tau>k\}}
 \{V_{M_*}(X_{k\mathsf h+j})+V_{M_*}(Y_{k\mathsf h+j})\}\bigr]\\
 &\le C_{\mathsf h}
 \E_{x,y}\bigl[
 \one_{\{\tau>k\}}\mathbf V(X_{k\mathsf h},Y_{k\mathsf h})\bigr]\\
 &\le C_{\mathsf h}\alpha^k\mathbf V(x,y),
\end{align*}
where the last inequality is \eqref{eq:stopped-return-drift}.
Choose $\eta>0$ so that $\e^\eta\alpha<1$.  Summing in $k$ and $j$
therefore gives
\begin{align}\label{eq:return-log-occupation}
 \E_{x,y}\sum_{k<\tau}\e^{\eta k}
 \sum_{j=0}^{\mathsf h-1}(\log A_{k,j})^m
 \le C\mathbf V(x,y).
\end{align}

Put $D_t=\norm{X_t-Y_t}_H$.  On every unit interval for which
$D_{k\mathsf h+j}>0$,
\begin{align*}
 \sup_{0\le s\le1}D_{k\mathsf h+j+s}\le A_{k,j}D_{k\mathsf h+j},
\end{align*}
while if $D_{k\mathsf h+j}=0$, pathwise uniqueness and the almost-sure
compatibility of the representation with concatenation imply that the
two paths remain identical under the common future Wiener blocks.
Iterating the preceding inequality
therefore yields
\begin{align*}
 B_\tau
 \le\prod_{k<\tau}\prod_{j=0}^{\mathsf h-1}A_{k,j},
 \qquad
 \log B_\tau
 \le\sum_{k<\tau}\sum_{j=0}^{\mathsf h-1}\log A_{k,j}.
\end{align*}
For $m>1$,  H\"older's inequality gives
\begin{align*}
 \left(\sum_{k<\tau}\sum_{j=0}^{\mathsf h-1}\log A_{k,j}\right)^m
 \le C_{\eta,m,\mathsf h}
 \sum_{k<\tau}\e^{\eta k}
 \sum_{j=0}^{\mathsf h-1}(\log A_{k,j})^m,
\end{align*}
and for $m=1$ the same estimate is immediate.  Together with
\eqref{eq:return-log-occupation}, this proves
\begin{align*}
 \E_{x,y}(\log B_\tau)^m
 \le C_m\mathbf V(x,y).
\end{align*}
Combining this estimate with \eqref{eq:return-time-exponential} proves
\eqref{eq:pair-return} with $M_m=M_*$.

\end{proof}

The next lemma gives a common bounded reset through uniform accessibility. 

\begin{lemma}
\label{lem:one-shot-reset}
Assume \Cref{ass:poly-lyapunov,ass:common-accessibility}.  There is $R_0<\infty$ such that, for every $R\ge1$ and $r>0$, there are $l<\infty$, $p>0$, and an exact-marginal coupling of the two path laws from every $x,y\in\mathbb V_R$ satisfying
\begin{align}\label{eq:one-shot-reset}
 \Prob\{X_l,Y_l\in\mathbb V_{R_0},\
 \norm{X_l-Y_l}_H\le r\}\ge p.
\end{align}
\end{lemma}

\begin{proof}
Fix $R\ge1$ and $r>0$, and apply \eqref{eq:common-accessibility} with radius $r/2$.  Thus, for some $l<\infty$ and $p_0>0$, with
\begin{align*}
 A=\mathbb V_{R_*}\cap B_H(x_*,r/2),
\end{align*}
we have
\begin{align*}
 \inf_{z\in\mathbb V_R}P_l(z,A)\ge p_0.
\end{align*}
Couple the two $l$-block path laws independently.  Then for every $x,y\in\mathbb V_R$,
\begin{align*}
 \Prob\{X_l\in A,\ Y_l\in A\}
 =P_l(x,A)P_l(y,A)\ge p_0^2.
\end{align*}
On this event,
\begin{align*}
 X_l,Y_l\in\mathbb V_{R_*},
 \qquad
 \norm{X_l-Y_l}_H\le r.
\end{align*}
Hence \eqref{eq:one-shot-reset} holds with $R_0=R_*$ and $p=p_0^2$.
\end{proof}

The following lemma upgrades the preceding reset to a stopping time
construction with an exponential moment.  Let $R_0$ be as in
\Cref{lem:one-shot-reset}.

\begin{lemma}
\label{lem:common-reset}
Assume \Cref{ass:poly-lyapunov,ass:stable-compact-regularity,ass:common-accessibility}.  For every
$r>0$ and $x,y\in X$, there is an exact-marginal coupling and a stopping
time $\sigma_r$ such that
\begin{align*}
 X_{\sigma_r},Y_{\sigma_r}\in\mathbb V_{R_0},\qquad
 \norm{X_{\sigma_r}-Y_{\sigma_r}}_H\le r.
\end{align*}
Moreover, there are $\lambda>0$ and $C<\infty$ such that
\begin{align}\label{eq:reset-clock}
 \E_{x,y}\e^{\lambda\sigma_r}
 \le C\{V_1(x)+V_1(y)\}.
\end{align}
\end{lemma}

\begin{proof}

Set
\begin{align*}
 \mathbf V_1(x,y)=V_1(x)+V_1(y).
\end{align*}
For any coupling of the two $\mathsf h$-blocks, \eqref{eq:poly-lyapunov} gives
\begin{align*}
 \E\mathbf V_1(X_{\mathsf h},Y_{\mathsf h})
 \le a_{\mathsf h}\mathbf V_1(x,y)+b,\qquad
 a_{\mathsf h}=C_1\e^{-c_1\mathsf h},\qquad b=2C_1.
\end{align*}
Choose $\mathsf h\in\N$ so that $a_{\mathsf h}\le1/4$, and then $R<\infty$ so large that
$b\le\mathbf V_1(x,y)/4$ whenever $\cV(x)+\cV(y)>R$.  Thus
\begin{align*}
 \E\mathbf V_1(X_{\mathsf h},Y_{\mathsf h})
 \le\frac12\mathbf V_1(x,y),
 \qquad \cV(x)+\cV(y)>R.
\end{align*}
Using common $\mathsf h$-blocks, let
\begin{align*}
 \tau_R=\inf\{n\ge0:\cV(X_{n\mathsf h})+\cV(Y_{n\mathsf h})\le R\}.
\end{align*}
The stopped calculation in \Cref{lem:pair-return} yields, for some
$\lambda_0>0$,
\begin{align}\label{eq:reset-return-clock}
 \E_{x,y}\e^{\lambda_0\mathsf h\tau_R}
 \le C\mathbf V_1(x,y).
\end{align}
Put
\begin{align*}
 \mathcal C_R=\{(x,y):\cV(x)+\cV(y)\le R\}.
\end{align*}
Since $\mathcal C_R\subset\mathbb V_R^2$, \Cref{lem:one-shot-reset}
gives $l<\infty$ and $p>0$ such that, from every pair in
$\mathcal C_R$, an $l$-block coupling succeeds with probability at
least $p$, where success means
\begin{align*}
 X_l,Y_l\in\mathbb V_{R_0},\qquad
 \norm{X_l-Y_l}_H\le r.
\end{align*}

Suppose first that $(x,y)\in\mathcal C_R$.  Make the first $l$-block
attempt.  If it succeeds, let $S_1$ be this event and set
$\sigma_r=l$.  On its complement, use fresh common $\mathsf h$-blocks
until the pair returns to $\mathcal C_R$ and denote the return index by
\begin{align*}
 \tau_1=\inf\{n\ge0:
 \cV(X_{l+n\mathsf h})+\cV(Y_{l+n\mathsf h})\le R\}.
\end{align*}
At time $l+\mathsf h\tau_1$ make a second fresh $l$-block attempt.  Let $S_2$
be the event that the first attempt failed and the second succeeds;
then
\begin{align*}
 \sigma_r=2l+\mathsf h\tau_1
 \qquad\text{on }S_2.
\end{align*}
Iterating, let $S_n$ be the event that the first $n-1$ attempts fail
and the $n$-th succeeds.  If $\tau_j$ denotes the return index to
$\mathcal C_R$ after the $j$-th failed attempt, define
\begin{align}\label{eq:reset-time-decomposition}
 \sigma_r
 =nl+\mathsf h\sum_{j=1}^{n-1}\tau_j
 \qquad\text{on }S_n.
\end{align}

The events $S_n$ are pairwise disjoint.  Moreover, after every failed
attempt the pair is returned to $\mathcal C_R$, so the Markov property
and \Cref{lem:one-shot-reset} give
\begin{align*}
 \Prob\left(\bigcap_{j=1}^n S_j^c\right)\le(1-p)^n, \text{ implying } \Prob\left(\bigcup_{n\ge1}S_n\right)=1
\end{align*}
and therefore \eqref{eq:reset-time-decomposition} defines an almost surely finite
stopping time $\sigma_r$.  By construction,
\begin{align*}
 X_{\sigma_r},Y_{\sigma_r}\in\mathbb V_{R_0},\qquad
 \norm{X_{\sigma_r}-Y_{\sigma_r}}_H\le r.
\end{align*}

It remains to prove the exponential moment.  Exact marginality of each
$l$-block and \eqref{eq:poly-lyapunov} give
\begin{align*}
 \sup_{(x,y)\in\mathcal C_R}
 \E_{x,y}\mathbf V_1(X_l,Y_l)<\infty.
\end{align*}
Markov property and \eqref{eq:reset-return-clock}  therefore gives
\begin{align*}
 K:=\sup_{(x,y)\in\mathcal C_R}
 \E_{x,y}\e^{\lambda_0(l+\mathsf h\tau_1)}<\infty.
\end{align*}
For $0<\lambda<\lambda_0$, put $\theta=\lambda/\lambda_0$.  By
H\"older's inequality,
\begin{align*}
 \beta(\lambda)
 :=\sup_{(x,y)\in\mathcal C_R}
 \E_{x,y}\bigl[
 \e^{\lambda(l+\mathsf h\tau_1)}\one_{S_1^c}\bigr]\le K^\theta(1-p)^{1-\theta}<1
\end{align*}
for $\lambda$ sufficiently small.  Applying the Markov property after
each return to $\mathcal C_R$ then gives
\begin{align*}
 \E_{x,y}\bigl[
 \e^{\lambda\sigma_r}\one_{S_n}\bigr]
 \le\e^{\lambda l}\beta(\lambda)^{n-1},
 \qquad (x,y)\in\mathcal C_R.
\end{align*}
Summing in $n$,
\begin{align}\label{eq:reset-core-clock}
 \sup_{(x,y)\in\mathcal C_R}
 \E_{x,y}\e^{\lambda\sigma_r}
 \le\frac{\e^{\lambda l}}{1-\beta(\lambda)}<\infty.
\end{align}

For arbitrary $x,y\in X$, first use common $\mathsf h$-blocks until the
entrance time $T_0=\mathsf h\tau_R$ into $\mathcal C_R$, and then perform the
preceding construction with fresh blocks.  Denote its additional
duration by $\widehat\sigma_r$ and set
\begin{align*}
 \sigma_r=T_0+\widehat\sigma_r.
\end{align*}
By \eqref{eq:reset-core-clock}, \eqref{eq:reset-return-clock} and  the Markov property,
\begin{align*}
 \E_{x,y}\e^{\lambda\sigma_r}
 \le C\E_{x,y}\e^{\lambda \mathsf h\tau_R}
 \le C\mathbf V_1(x,y)
 =C\{V_1(x)+V_1(y)\},
\end{align*}
which proves \eqref{eq:reset-clock}.  Since all blocks have the correct
coordinate marginals and are concatenated at stopping times with fresh
randomness, the resulting coupling is exact-marginal.
\end{proof}

\subsection{The logarithmic gauge and coupling cycles}

Fix $a>1$ and set $L_a=2(a+1)$.  Define $\ell_a(0)=0$ and
\begin{align*}
 \ell_a(r)=
 \begin{cases}
  \{L_a+\log(1/r)\}^{-a},&0<r\le1,\\
  L_a^{-a}+aL_a^{-a-1}(1-\e^{-(r-1)}),&r\ge1.
 \end{cases}
\end{align*}
The function $\ell_a$ is increasing, concave, bounded by one, and subadditive.  Moreover,
\begin{align}\label{eq:power-below-log}
 1\wedge r^\delta\le C_{a,\delta}\ell_a(r),
 \qquad r\ge0,\quad0<\delta\le1.
\end{align}

The next lemma converts a negative mean logarithmic multiplier into a
strict drift for the logarithmic gauge.  It also shows that algebraically
small exceptional errors are negligible at sufficiently small distances.

\begin{lemma}
\label{lem:negative-log-drift}
Let $\mathfrak m\ge q_*>0$ almost surely and suppose that
\begin{align*}
 \E\log\mathfrak m\le-\kappa,\qquad
 \E(\log^+\mathfrak m)^m\le C_m
\end{align*}
for some $m>a+2$.  Then there are $r_0,c>0$ such that
\begin{align}\label{eq:log-gauge-drift}
 \E\ell_a(r\mathfrak m)
 \le\ell_a(r)-c\ell_a(r)^{1+1/a},
 \qquad0<r\le r_0.
\end{align}
Moreover, if $e(r)\ge0$ satisfies $e(r)\le C_\theta r^\theta$ for
some $\theta>0$, then after decreasing $r_0$ and $c$ if necessary,
\begin{align}\label{eq:log-gauge-drift-error}
 \E\ell_a(r\mathfrak m)+e(r)
 \le\ell_a(r)-c\ell_a(r)^{1+1/a},
 \qquad0<r\le r_0.
\end{align}
\end{lemma}

\begin{proof}
Put $L=L_a+\log(1/r)$ and $ Y=\log\mathfrak m.$
Take $r_0$ so small that $L\ge2L_a$ for $0<r\le r_0$.  On
$\{Y\le L/2\}$ one then has $r\mathfrak m\le1$.  Therefore
\begin{align*}
 \ell_a(r\mathfrak m)=(L-Y)^{-a}\leq L^{-a}+aYL^{-a-1}+ CY^2L^{-a-2}.
\end{align*}
On the complementary event $\{Y> L/2\}$, using $\ell_a\le1$ and Markov's inequality gives
\begin{align*}
 \Prob\{Y>L/2\}\le C_mL^{-m}.
\end{align*}
Therefore 
\begin{align*}
 \E\ell_a(r\mathfrak m) \le L^{-a}-a\kappa L^{-a-1}+CL^{-a-2}+C_mL^{-m}.
\end{align*}
Since $m>a+2$, the last two terms are $o(L^{-a-1})$.  Hence, after
decreasing $r_0$,
\begin{align*}
 \E\ell_a(r\mathfrak m)
 \le L^{-a}-\frac{a\kappa}{2}L^{-a-1}.
\end{align*}
For $r\le1$,
\begin{align*}
 \ell_a(r)=L^{-a},\qquad
 L^{-a-1}=\ell_a(r)^{1+1/a},
\end{align*}
which proves \eqref{eq:log-gauge-drift}. Finally,
\begin{align*}
 r^\theta
 =\e^{\theta L_a}\e^{-\theta L}
 =o(L^{-a-1})
 =o\bigl(\ell_a(r)^{1+1/a}\bigr)
 \qquad(r\downarrow0).
\end{align*}
Thus any $e(r)\le C_\theta r^\theta$ is absorbed
after decreasing $r_0$ and $c$, proving
\eqref{eq:log-gauge-drift-error}.
\end{proof}

The same splitting without a sign assumption on $\E Y$ gives the envelope estimate
\begin{align}\label{eq:log-envelope}
 \E\sup_{0\le s\le\tau}\ell_a(r\e^{Y_s})
 \le C\ell_a(r)
\end{align}
whenever  $\sup_{s\le\tau}Y_s\le Y$ and $\E(Y_+)^m<\infty$ for some $m>a$.

We now concatenate the negative-log core block with Lyapunov returns and
resets to obtain a renewable coupling scheme.
\begin{proposition}
\label{prop:log-coupling-cycles}
Assume the hypotheses of \Cref{thm:stable-compact-polynomial-mixing} and fix $a>1$.  There are a core radius $R$, a distance $r_c>0$, stopping times $0\le\sigma_0<\sigma_1<\cdots$, and an exact-marginal coupling such that
\begin{align*}
 X_{\sigma_n},Y_{\sigma_n}\in\mathbb V_R,\qquad
 \norm{X_{\sigma_n}-Y_{\sigma_n}}_H\le r_c.
\end{align*}
Writing
\begin{align*}
 D_n=\ell_a(\norm{X_{\sigma_n}-Y_{\sigma_n}}_H),\qquad
 H_n=\sup_{\sigma_n\le t<\sigma_{n+1}}
 \ell_a(\norm{X_t-Y_t}_H),
\end{align*}
there are $M<\infty$ and $c,\lambda,C>0$ such that
\begin{align}
 \E[D_{n+1}\mid\mathcal F_{\sigma_n}]
 &\le D_n-cD_n^{1+1/a},\label{eq:cycle-drift}\\
 \E[H_n\mid\mathcal F_{\sigma_n}]
 &\le CD_n,\label{eq:cycle-envelope}\\
 \E[\e^{\lambda(\sigma_{n+1}-\sigma_n)}
 \mid\mathcal F_{\sigma_n}]
 &\le C.\label{eq:cycle-clock}
\end{align}
Every cycle has a deterministic positive minimum duration, and
\begin{align}\label{eq:initial-clock}
 \E_{x,y}\e^{\lambda\sigma_0}
 \le C\{V_M(x)+V_M(y)\}.
\end{align}
\end{proposition}

\begin{proof}

Fix an integer $m>a+3$.  Apply \Cref{lem:pair-return} at this order and enlarge its return level, if necessary, so that the resulting return core is contained in $\mathbb V_R^2$ and $\mathbb V_{R_0}\subset\mathbb V_R$, 
where $R_0$ is the fixed reset radius from \Cref{lem:one-shot-reset}.

\emph{Step 1: one matched core--return block.}
We first derive a uniform consequence of \Cref{lem:pair-return} for random exact-marginal outputs.  Let a block of arbitrary deterministic length $T$ start from $x,y\in\mathbb V_R$, and let $(\xi,\zeta)$ be any exact-marginal pair of its endpoints.  Thus
\begin{align*}
 \law(\xi)=P_T(x,\cdot),\qquad
 \law(\zeta)=P_T(y,\cdot).
\end{align*}
Starting from $(\xi,\zeta)$, use fresh common Wiener blocks and let $\tau$ and $B$ be the return index and path multiplier supplied by \Cref{lem:pair-return}.  Conditionally on $(\xi,\zeta)$,
\begin{align*}
 \E\bigl[\e^{\lambda_m\tau}+(\log B)^m\mid\xi,\zeta\bigr]
 \le C_m\{V_{M_m}(\xi)+V_{M_m}(\zeta)\}.
\end{align*}
Exact marginality and \eqref{eq:poly-lyapunov} therefore give
\begin{align*}
 \E\{V_{M_m}(\xi)+V_{M_m}(\zeta)\}
 =P_TV_{M_m}(x)+P_TV_{M_m}(y)\le C_{R,m},
\end{align*}
uniformly in $T$.  Consequently there are $M_{\rm ret},C_m<\infty$, independent of the block length, such that
\begin{align}\label{eq:return-multiplier}
 \E\log B\le M_{\rm ret},\qquad
 \E(\log B)^m\le C_m,
\end{align}
and the actual return time $\mathsf h_m\tau$ has a uniform exponential moment.

Apply \Cref{prop:negative-log-core} on $\mathbb V_R$ with
\begin{align*}
 \mathsf A=M_{\rm ret}+3,
\end{align*}
and denote the resulting constants by $T_c,r_{\rm core},q_*$.  Starting from $(x,y)\in\mathbb V_R^2$ with
\begin{align*}
 0<r=\norm{x-y}_H\le r_{\rm core},
\end{align*}
perform the core block and then, from its two genuine outputs, use fresh common Wiener blocks until the return supplied by \Cref{lem:pair-return}.  On the repair event $\mathcal M$, the genuine second path agrees with the companion path throughout the core block, so the distance at the return time is at most $r\mathfrak m_0B$. 
Define the virtual multiplier
\begin{align*}
 \mathfrak m=
 \begin{cases}
  \mathfrak m_0B,&\text{on }\mathcal M,\\
  q_*,&\text{on }\mathcal M^c.
 \end{cases}
\end{align*}
Since $\mathfrak m_0\ge q_*$ and $B\ge1$,
\begin{align*}
 \log\mathfrak m\le\log\mathfrak m_0+\log B.
\end{align*}
Thus \eqref{eq:core-log-bounds} and \eqref{eq:return-multiplier} give
\begin{align}\label{eq:cycle-log-multiplier}
 \E\log\mathfrak m\le-3,\qquad
 \E(\log^+\mathfrak m)^m\le C_m.
\end{align}
Apply \Cref{lem:negative-log-drift} to this family and let $r_{\log},c_0>0$ be the resulting constants.  Choose
\begin{align*}
 0<r_c\le r_{\rm core}\wedge r_{\log}.
\end{align*}
Then, uniformly for $0<r\le r_c$,
\begin{align}\label{eq:ideal-cycle-drift}
 \E\ell_a(r\mathfrak m)
 \le\ell_a(r)-c_0\ell_a(r)^{1+1/a}.
\end{align}

\emph{Step 2: construction of the cycle times and the drift.}
For a matched core--return block starting at distance $0<r<r_c$, let
\begin{align*}
 \mathcal E_r=\mathcal M^c\cup\{r\mathfrak m>r_c\}.
\end{align*}
By \eqref{eq:path-repair}, \eqref{eq:cycle-log-multiplier}, and
Markov's inequality,
\begin{align*}
 \Prob(\mathcal E_r)
 \le Cr+C_m\{\log(r_c/r)\}^{-m}.
\end{align*}
Since $m>a+3$, choose $r^{\#}\in(0,r_c)$ so small that
\begin{align}\label{eq:small-reset-error}
 Cr+C_m\{\log(r_c/r)\}^{-m}
 \le\frac{c_0}{2}\ell_a(r)^{1+1/a},
 \qquad0<r\le r^{\#}.
\end{align}

We now construct the cycle times.  Apply \Cref{lem:common-reset} to
the original pair $(x,y)$ with target distance $q_*r^{\#}$ and denote
the resulting stopping time by $\sigma_0$.  Since
$\mathbb V_{R_0}\subset\mathbb V_R$ and $q_*<1$,
\begin{align*}
 X_{\sigma_0},Y_{\sigma_0}\in\mathbb V_R,\qquad
 \norm{X_{\sigma_0}-Y_{\sigma_0}}_H
 \le q_*r^{\#}<r_c.
\end{align*}

Suppose recursively that $\sigma_n$ has been defined and
\begin{align*}
 X_{\sigma_n},Y_{\sigma_n}\in\mathbb V_R,\qquad
 r:=\norm{X_{\sigma_n}-Y_{\sigma_n}}_H\le r_c.
\end{align*}
If $r=0$, use identical future Wiener blocks.  After the deterministic
time $T_c$, continue with common $\mathsf h_m$-blocks until the return
supplied by \Cref{lem:pair-return}, and define $\sigma_{n+1}$ to be
that return time.  The two coordinates remain identical throughout.

Suppose now that $r>0$.  Starting at $\sigma_n$, perform the
negative-log core block of Step~1.  From its two genuine endpoints,
use fresh common $\mathsf h_m$-blocks until the return supplied by
\Cref{lem:pair-return}.  Let $\tau_n$ and $B_n$ be the corresponding
return index and path multiplier, and set
\begin{align*}
 \widehat\sigma_{n+1}
 =\sigma_n+T_c+\mathsf h_m\tau_n.
\end{align*}
Let $\mathcal M_n$, $\mathfrak m_{0,n}$, and $\mathfrak m_n$ denote
the corresponding repair event, core multiplier, and virtual
multiplier.  If
\begin{align*}
 \mathcal M_n\quad\text{occurs and}\quad
 r\mathfrak m_n\le r_c,
\end{align*}
set
\begin{align*}
 \sigma_{n+1}=\widehat\sigma_{n+1}.
\end{align*}
Otherwise apply \Cref{lem:common-reset} at
$\widehat\sigma_{n+1}$ with target distance $q_*r^{\#}$.  If
$\mathcal R_n$ denotes its additional duration, set
\begin{align*}
 \sigma_{n+1}=\widehat\sigma_{n+1}+\mathcal R_n.
\end{align*}
In either case,
\begin{align*}
 X_{\sigma_{n+1}},Y_{\sigma_{n+1}}\in\mathbb V_R,\qquad
 \norm{X_{\sigma_{n+1}}-Y_{\sigma_{n+1}}}_H\le r_c.
\end{align*}
Thus the construction may be iterated indefinitely.  All blocks are
concatenated at stopping times with fresh Wiener increments, so every
$\sigma_n$ is a stopping time.  Moreover,
\begin{align*}
 \sigma_{n+1}-\sigma_n\ge T_c.
\end{align*}

We next prove the drift.  Condition on $\mathcal F_{\sigma_n}$ and,
for notational simplicity, suppress the cycle index:
\begin{align*}
 r=\norm{X_{\sigma_n}-Y_{\sigma_n}}_H,\qquad
 \mathcal M=\mathcal M_n,\qquad
 \mathfrak m=\mathfrak m_n.
\end{align*}
Suppose first that $r\ge r^{\#}$.  If repair fails, then
$\mathfrak m=q_*$ and the reset endpoint satisfies
\begin{align*}
 \norm{X_{\sigma_{n+1}}-Y_{\sigma_{n+1}}}_H
 \le q_*r^{\#}\le q_*r=r\mathfrak m.
\end{align*}
If repair succeeds but $r\mathfrak m>r_c$, the reset is used only to
restore the small distance condition for the next cycle, and again
\begin{align*}
 \norm{X_{\sigma_{n+1}}-Y_{\sigma_{n+1}}}_H
 \le q_*r^{\#}<r_c<r\mathfrak m.
\end{align*}
When no reset is used, the endpoint distance is already at most
$r\mathfrak m$.  Hence in all cases
\begin{align*}
 \norm{X_{\sigma_{n+1}}-Y_{\sigma_{n+1}}}_H
 \le r\mathfrak m,
\end{align*}
and \eqref{eq:ideal-cycle-drift} gives
\begin{align*}
 \E[D_{n+1}\mid\mathcal F_{\sigma_n}]
 \le D_n-c_0D_n^{1+1/a}.
\end{align*}

Suppose instead that $0<r<r^{\#}$.  On
\begin{align*}
 \mathcal E_r=\mathcal M^c\cup\{r\mathfrak m>r_c\},
\end{align*}
the reset endpoint need not be bounded by $r\mathfrak m$, so we use
only $\ell_a\le1$.  Therefore
\begin{align*}
 \E[D_{n+1}\mid\mathcal F_{\sigma_n}]
 \le\E\ell_a(r\mathfrak m)+\Prob(\mathcal E_r)\le\ell_a(r)-\frac{c_0}{2}\ell_a(r)^{1+1/a},
\end{align*}
by \eqref{eq:ideal-cycle-drift} and
\eqref{eq:small-reset-error} which proves
\eqref{eq:cycle-drift}
after decreasing the constant. 

\emph{Step 3: envelope and clocks.} Condition again on $\mathcal F_{\sigma_n}$ and write
\begin{align*}
 r=\norm{X_{\sigma_n}-Y_{\sigma_n}}_H.
\end{align*}
If $r=0$, then the two paths are identical throughout the cycle and
$H_n=0$.

Suppose first that $0<r<r^{\#}$.  On $\mathcal E_r^c$ the repair
succeeds and no reset is used.  By \Cref{prop:negative-log-core} and \Cref{lem:pair-return}, the core block and the following
common return then satisfy  
\begin{align*}
 \sup_{\sigma_n\le t<\sigma_{n+1}}
 \norm{X_t-Y_t}_H
 \le r\mathfrak m_0^{\sup}B.
\end{align*}
By \eqref{eq:core-log-bounds} and
\eqref{eq:return-multiplier},
\begin{align*}
 \E\bigl[\log^+(\mathfrak m_0^{\sup}B)\bigr]^m\le C_m.
\end{align*}
Since $m>a$, the envelope estimate \eqref{eq:log-envelope} gives
\begin{align*}
 \E\bigl[H_n\one_{\mathcal E_r^c}
 \mid\mathcal F_{\sigma_n}\bigr]
 \le C\ell_a(r).
\end{align*}
On $\mathcal E_r$ we use only $H_n\le1$.  By
\eqref{eq:small-reset-error},
\begin{align*}
 \E\bigl[H_n\one_{\mathcal E_r}\mid\mathcal F_{\sigma_n}\bigr]
 \le\Prob(\mathcal E_r)
 \le\frac{c_0}{2}\ell_a(r)^{1+1/a}
 \le C\ell_a(r).
\end{align*}
Thus
\begin{align*}
 \E[H_n\mid\mathcal F_{\sigma_n}]
 \le CD_n,\qquad0<r<r^{\#}.
\end{align*}
For $r\ge r^{\#}$, simply use $H_n\le1$ and
\begin{align*}
 D_n=\ell_a(r)\ge\ell_a(r^{\#})>0
\end{align*}
to obtain the same bound with a larger constant.  This proves
\eqref{eq:cycle-envelope}.

It remains to control the cycle duration.  The core block has the
deterministic length $T_c$, and the following common return has a
uniform conditional exponential moment by Step~1.  If a reset is
used, it starts from a pair in $\mathbb V_R^2$, so
\Cref{lem:common-reset} gives a uniform conditional exponential
moment for its additional duration.  Choosing $\lambda>0$ smaller
than the corresponding exponential rates and conditioning
successively at the end of the core block and at the common return
therefore gives
\begin{align*}
 \E\bigl[\e^{\lambda(\sigma_{n+1}-\sigma_n)}
 \mid\mathcal F_{\sigma_n}\bigr]\le C.
\end{align*}
This also covers the $r=0$ branch and proves
\eqref{eq:cycle-clock}.

Finally, $\sigma_0$ was constructed by
\Cref{lem:common-reset} with the fixed target distance
$q_*r^{\#}$.  After decreasing $\lambda$ to the common value above,
\eqref{eq:reset-clock} gives
\begin{align*}
 \E_{x,y}\e^{\lambda\sigma_0}
 \le C\{V_1(x)+V_1(y)\}
 \le C\{V_M(x)+V_M(y)\},
\end{align*}
which is \eqref{eq:initial-clock}.  Since every constituent block has
the correct coordinate marginals and the blocks are concatenated at
stopping times with fresh randomness, the resulting continuous time
coupling has the exact two Markov marginals throughout.
\end{proof}

\subsection{Deterministic time conversion and proof of the criterion}

\begin{proof}[Proof of \Cref{thm:stable-compact-polynomial-mixing}]
Fix $q>0$ and $0<\delta\le1$, and choose $a=q+1>1$. 
Apply \Cref{prop:log-coupling-cycles} with this $a$.  We first convert the cycle drift into decay in the cycle index.  Put
\begin{align*}
 u_n=\E D_n.
\end{align*}
Taking expectations in \eqref{eq:cycle-drift} and using Jensen's inequality gives
\begin{align*}
 u_{n+1}\le u_n-cu_n^{1+1/a}
\end{align*}
which implies 
\begin{align*}
 \E D_n\le C(1+n)^{-a}.
\end{align*}

We next pass to deterministic time.  Let $T_*>0$ be the deterministic lower bound for the cycle durations supplied by \Cref{prop:log-coupling-cycles}.  Iterating \eqref{eq:cycle-clock} and using \eqref{eq:initial-clock}, there is $C_{\#}\ge1$ such that
\begin{align*}
 \E_{x,y}\e^{\lambda\sigma_n}
 \le C\{V_M(x)+V_M(y)\}C_{\#}^n.
\end{align*}
Since $\ell_a\le1$, before the first cycle time,
\begin{align*}
 \E_{x,y}\bigl[
 \ell_a(\norm{X_t-Y_t}_H)\one_{\{t<\sigma_0\}}\bigr]
 \le\Prob_{x,y}\{\sigma_0>t\}
 \le C\{V_M(x)+V_M(y)\}\e^{-\lambda t}.
\end{align*}
On $\{t\ge\sigma_0\}$, let $N_t$ be the unique integer such that
\begin{align*}
 \sigma_{N_t}\le t<\sigma_{N_t+1}.
\end{align*}
Choose $c_*>0$ so small that
\begin{align*}
 C_{\#}^{c_*t}\e^{-\lambda t}\le\e^{-\lambda t/2},
 \qquad t\ge0.
\end{align*}
For $n<c_*t$,
\begin{align*}
 \Prob\{N_t=n\}
 \le\Prob\{\sigma_{n+1}>t\}
 \le C\{V_M(x)+V_M(y)\}\e^{-\lambda t/2}.
\end{align*}
Summing these early cycle contributions and using $\ell_a\le1$ gives
\begin{align*}
 \sum_{n<c_*t}
 \E\bigl[\ell_a(\norm{X_t-Y_t}_H)\one_{\{N_t=n\}}\bigr]
 \le C\{V_M(x)+V_M(y)\}\e^{-ct}.
\end{align*}
Moreover, $\sigma_n\ge nT_*$, so $N_t\le t/T_*$.  For the remaining indices,
\begin{align*}
 \E\bigl[\ell_a(\norm{X_t-Y_t}_H)\one_{\{N_t=n\}}\bigr]
 \le\E H_n\le C\E D_n\le C(1+n)^{-a}.
\end{align*}
Therefore
\begin{align*}
 \E_{x,y}\ell_a(\norm{X_t-Y_t}_H)
 &\le C\{V_M(x)+V_M(y)\}\e^{-ct}
 +C\sum_{c_*t\le n\le t/T_*}(1+n)^{-a}\notag\\
 &\le C\{V_M(x)+V_M(y)\}(1+t)^{-(a-1)}.
\end{align*}
Since $a-1=q$, \eqref{eq:power-below-log} now gives
\begin{align*}
 \E_{x,y}\bigl[1\wedge\norm{X_t-Y_t}_H^\delta\bigr]
 \le C_{a,\delta}\E_{x,y}\ell_a(\norm{X_t-Y_t}_H)
 \le C\{V_M(x)+V_M(y)\}(1+t)^{-q}.
\end{align*}
The coupling has the exact two Markov marginals at every time, so the coupling inequality yields \eqref{eq:abstract-mixing-rate}.
\end{proof}

\section{Polynomial mixing for the cubic stochastic NLS}

This section verifies
\Cref{ass:stable-compact-regularity,ass:dense-malliavin,ass:common-accessibility}
of \Cref{thm:stable-compact-polynomial-mixing} for
\eqref{eq:intro-spde} with $X=H^1$ and $H=L^2$.  The Lyapunov
hypothesis \Cref{ass:poly-lyapunov} is verified by
\Cref{prop:lyapunov}.  We will also provide
geometric characterizations of saturation and stationary regularity gain before proving
\Cref{thm:main}.  Throughout this section, set
\begin{align*}
 \cV(u)=1+\cE(u),\qquad
 \mathbb V_R=\{u\in H^1:\cV(u)\le R\},
\end{align*}
where $\cE$ is the NLS Hamiltonian defined in
\eqref{eq:hamiltonian-dissipation}.

For the SNLS application, the abstract Wiener data of
\Cref{sec:abstract-criterion} are specialized as follows.  For
$T>0$, let
\begin{align*}
 E_T&=C_0([0,T];\R^m)
 =\{\omega\in C([0,T];\R^m):\omega(0)=0\},\\
 H_T&=\{h\in H^1([0,T];\R^m):h(0)=0\},
 \qquad
 \norm h_{H_T}^2=\int_0^T|\dot h(t)|^2\,\dd t.
\end{align*}
Let $\boldsymbol{\gamma}_T$ be Wiener measure on $E_T$ and
$\iota_T:H_T\hookrightarrow E_T$ the natural Cameron--Martin
embedding.  Identifying $H_T^*$ with $H_T$ by the Riesz isomorphism,
set
\begin{align*}
 H_T^0=\iota_T^*(E_T^*)\subset H_T.
\end{align*}
We use
\begin{align*}
 \boldsymbol\Phi_T:H^1\times E_T\longrightarrow C([0,T];H^1),
 \qquad
 \Phi_T(x,\omega)=\boldsymbol\Phi_T(x,\omega)(T),
\end{align*}
for the Borel trajectory representation constructed in
\Cref{app:analytic}.

\subsection{The stable--compact structure}
This subsection verifies  the exact stable--compact difference decomposition. The remaining regular dependence requirements in
\Cref{ass:stable-compact-regularity} are verified in
\Cref{prop:snls-regular-dependence} in Appendix~\ref{app:analytic}.

The operator ideal argument is related to the analysis of Bogoliubov propagators in \cite[Theorem~3.6]{BruneauDerezinski2007}. Here a Sylvester identity yields compactness, and approximation of the time-dependent coefficients gives the form needed for exact SNLS differences.

Although $L^2(\T^3;\C)$ is regarded throughout as a real Hilbert
space for the nonlinear and Malliavin calculus, in this subsection
we temporarily retain its canonical complex Hilbert structure and
write $\mathbb H=L^2(\T^3;\C)$.
A real-linear equation involving both $\xi$ and $\bar\xi$ is realized
as the restriction of a complex-linear doubled system on
$\mathbb H\oplus\mathbb H$ to the closed real subspace $\mathbb H_{\rm phys}
=\{(\xi,\bar\xi):\xi\in\mathbb H\}.$

The following lemma shows that the off-diagonal conjugate coupling in
a Bogoliubov block system contributes only a compact defect to the
finite time propagator, which is useful for the stable--compact decomposition of SNLS. 
\begin{lemma}
\label{lem:opposite-sign-sylvester}
Consider $\mathbb H$ as a complex Hilbert space and let
$\mathsf H=\mathsf H^*\ge0$ have compact resolvent, and 
$C\in\cL(\mathbb H)$.  On $\mathbb H\oplus\mathbb H$ set
\begin{align*}
 D=\begin{pmatrix}-\ii\mathsf H&0\\0&\ii\mathsf H\end{pmatrix},\qquad
 R=\begin{pmatrix}0&-\ii C\\ \ii C^*&0\end{pmatrix},\qquad
 G=D+R.
\end{align*}
Then $\e^{tG}-\e^{tD}$ is compact for every $t\in\R$.
\end{lemma}

\begin{proof}It suffices to consider $t\ge0$. 
The operator $D$ is skew-adjoint and generates a unitary $C_0$-group.
Since $R$ is bounded, $G=D+R$ generates a $C_0$-group and
$\mathrm{Dom}(G)=\mathrm{Dom}(D)$.
Set $\mathsf A=I+\mathsf H\ge I$.  For $Q\in\cL(\mathbb H)$ define
\begin{align*}
 \mathcal S_{\mathsf A}(Q)
 =\int_0^\infty\e^{-r\mathsf A}Q\e^{-r\mathsf A}\,\dd r.
\end{align*}
This integral converges in operator norm and
$\mathcal S_{\mathsf A}(Q)$ is compact.  
For $x,y\in\mathrm{Dom}(\mathsf A)$, integration of
\begin{align*}
 -\frac{\dd}{\dd r}
 \la\e^{-r\mathsf A}Q\e^{-r\mathsf A}x,y\ra
\end{align*}
gives
\begin{align*}
 \la\mathcal S_{\mathsf A}(Q)x,\mathsf Ay\ra
 +\la\mathcal S_{\mathsf A}(Q)\mathsf Ax,y\ra
 =\la Qx,y\ra.
\end{align*}
Since $\mathsf A$ is self-adjoint, one obtains  $\mathcal S_{\mathsf A}(Q)x\in\mathrm{Dom}(\mathsf A)$ 
and the Lyapunov--Sylvester equation 
\begin{align}\label{eq:sylvester-identity}
 \mathsf A\mathcal S_{\mathsf A}(Q)
 +\mathcal S_{\mathsf A}(Q)\mathsf A=Q
 \qquad\text{on }\mathrm{Dom}(\mathsf H).
\end{align}
Now put
\begin{align*}
 X=\begin{pmatrix}
 0&\mathcal S_{\mathsf A}(C)\\
 \mathcal S_{\mathsf A}(C^*)&0
 \end{pmatrix}.
\end{align*}
Then $X$ is compact and $X\,\mathrm{Dom}(D)\subset\mathrm{Dom}(D)$.  Since $\mathsf H=\mathsf A-I$, \eqref{eq:sylvester-identity} yields
on $\mathrm{Dom}(D)$,
\begin{align*}
 [D,X]=R+C_0,\quad \text{ where }  C_0=2\ii
 \begin{pmatrix}
 0&\mathcal S_{\mathsf A}(C)\\
 -\mathcal S_{\mathsf A}(C^*)&0
 \end{pmatrix}.
\end{align*}

Let $\mathbf G(t)=\e^{tG},\, \, 
\mathbf D(t)=\e^{tD}.$
Duhamel's formula gives
\begin{align}\label{eq:G-D}
 \mathbf G(t)-\mathbf D(t)
 =\int_0^t\mathbf G(t-r)R\mathbf D(r)\,\dd r.
\end{align}
For vectors in $\mathrm{Dom}(D)$, differentiate
\begin{align*}
 F(r)=\mathbf G(t-r)X\mathbf D(r).
\end{align*}
Since $X$ preserves $\mathrm{Dom}(D)=\mathrm{Dom}(G)$, and 
\begin{align*}
 F'(r)
 =-\mathbf G(t-r)\bigl([D,X]+RX\bigr)\mathbf D(r),
\end{align*}
using $R=[D,X]-C_0$, \eqref{eq:G-D} and integrating in $r$ therefore gives
\begin{align}\label{eq:sylvester-duhamel}
 \mathbf G(t)-\mathbf D(t)
 =\mathbf G(t)X-X\mathbf D(t)
 -\int_0^t
 \mathbf G(t-r)(C_0+RX)\mathbf D(r)\,\dd r.
\end{align}
The boundary terms are compact and the integrand 
is operator-norm continuous with values in
$\cK(\mathbb H\oplus\mathbb H)$.  Hence the integral is compact. Density extends 
\eqref{eq:sylvester-duhamel} to the whole space and proves
the result.
\end{proof}

The next lemma gives a time dependent version of the previous decomposition. In what follows, the operator $\mathbf M_u$ represents the multiplication by $u$. 
\begin{lemma}
\label{lem:time-dependent-bog}
Let $\mathsf d\in L^1(0,T;L^\infty(\T^3;\R))$ with
$\mathsf d\ge0$, and let
$\mathsf c\in L^1(0,T;L^\infty(\T^3;\C))$.   Consider
\begin{align*}
 \partial_tz=-\ii(-\Delta+\mathbf M_{\mathsf d(t)})z
 -\ii \mathbf M_{\mathsf c(t)}\bar z .
\end{align*}
Let $\mathcal U(t,s)$ be its real $L^2$ propagator and let
$U_{\mathsf d}(t,s)$ be the complex unitary propagator of
\begin{align*}
 \partial_tz=-\ii(-\Delta+\mathbf M_{\mathsf d(t)})z.
\end{align*}
Then
\begin{align}\label{eq:time-dependent-decomposition}
 \mathcal U(t,s)=U_{\mathsf d}(t,s)+\mathcal K(t,s)
\end{align}
as real linear operators, where $\mathcal K(t,s)$ is compact on $L^2$.
\end{lemma}

\begin{proof}
We first prove the claim for interval step coefficients and then pass
to general $L^1_tL^\infty_x$ coefficients by operator-norm stability
of the propagators.

\emph{Step 1: interval step coefficients.}
Suppose first that $\mathsf d$ and $\mathsf c$ are interval step functions.  Fix $0\le s<t\le T$ and choose
\begin{align*}
 s=t_0<t_1<\cdots<t_N=t
\end{align*}
so that $\mathsf d(r)=\mathsf d_j$ and $\mathsf c(r)=\mathsf c_j$ on $(t_{j-1},t_j)$.  Set
\begin{align*}
 \mathsf H_j=-\Delta+\mathbf M_{\mathsf d_j},\qquad C_j=\mathbf M_{\mathsf c_j},\qquad \tau_j=t_j-t_{j-1}.
\end{align*}
Since $\mathsf d_j$ is real-valued, bounded, and nonnegative, $\mathsf H_j\geq 0$ is self-adjoint on $H^2(\T^3)$ and has compact resolvent, while $C_j\in\cL(L^2)$. On $\mathbb H\oplus\mathbb H$ put
\begin{align*}
 D_j=\begin{pmatrix}-\ii\mathsf H_j&0\\0&\ii\mathsf H_j\end{pmatrix},\qquad
 R_j=\begin{pmatrix}0&-\ii C_j\\ \ii C_j^*&0\end{pmatrix},\qquad
 G_j=D_j+R_j.
\end{align*}
By \Cref{lem:opposite-sign-sylvester},
\begin{align}\label{eq:step-doubled-compact}
 \e^{\tau_jG_j}=\e^{\tau_jD_j}+\widetilde K_j,\qquad
 \widetilde K_j\in\cK(\mathbb H\oplus\mathbb H).
\end{align}
Both groups preserve the physical real subspace
\begin{align*}
 \mathbb H_{\rm phys}=\{(z,\bar z):z\in\mathbb H\}.
\end{align*}
Indeed, with $U_j=\e^{-\ii\tau_j\mathsf H_j}$,
\begin{align*}
 \e^{\tau_jD_j}(z,\bar z)=\bigl(U_jz,\overline{U_jz}\bigr),
\end{align*}
while $\e^{\tau_jG_j}(z,\bar z)
=(\mathcal U_jz,\overline{\mathcal U_jz})$ by construction of the
doubled equation.  Hence $\widetilde K_j$ also preserves
$\mathbb H_{\rm phys}$.  Define
\begin{align*}
 K_jz=\pi_1\widetilde K_j(z,\bar z),
\end{align*}
where $\pi_1$ denotes projection onto the first component.  Taking the
first component in \eqref{eq:step-doubled-compact} gives
\begin{align}\label{eq:step-real-compact}
 \mathcal U_j=U_j+K_j.
\end{align}
Since $z\mapsto(z,\bar z)$ and $\pi_1$ are bounded real-linear maps
and $\widetilde K_j$ is compact, one has
$K_j\in\cK(L^2)$.

By concatenation,
\begin{align*}
 \mathcal U(t,s)=\mathcal U_N\cdots\mathcal U_1,\qquad
 U_{\mathsf d}(t,s)=U_N\cdots U_1.
\end{align*}
Using \eqref{eq:step-real-compact},
\begin{align*}
 \mathcal U_N\cdots\mathcal U_1=(U_N+K_N)\cdots(U_1+K_1).
\end{align*}
The unique term containing no compact factor is $U_N\cdots U_1$; every other term contains at least one $K_j$ and is therefore compact.  Hence
\begin{align*}
 \mathcal U(t,s)-U_{\mathsf d}(t,s)\in\cK(L^2),
\end{align*}
which proves \eqref{eq:time-dependent-decomposition} for interval step coefficients.

\emph{Step 2: approximation of general coefficients.}
By the standard density of simple functions in Bochner spaces and
the regularity of Lebesgue measure, we may choose  interval step functions $\mathsf d_n,\mathsf c_n$ such that
\begin{align*}
 \mathsf d_n\ge0,\qquad
 \mathsf d_n\longrightarrow\mathsf d,\qquad
 \mathsf c_n\longrightarrow\mathsf c
 \quad\text{in }L^1_tL^\infty_x.
\end{align*}
Set
\begin{align*}
 \eta_n(r)=\norm{\mathsf d_n(r)-\mathsf d(r)}_{L^\infty}
 +\norm{\mathsf c_n(r)-\mathsf c(r)}_{L^\infty},
\end{align*}
and
\begin{align*}
 M_n=\norm{\mathsf d_n}_{L^1_tL^\infty_x}
 +\norm{\mathsf c_n}_{L^1_tL^\infty_x},\qquad
 M=\norm{\mathsf d}_{L^1_tL^\infty_x}
 +\norm{\mathsf c}_{L^1_tL^\infty_x}.
\end{align*}
Then Duhamel's formula  gives
\begin{align*}
 \sup_{s\le t\le T}
 \norm{\mathcal U_n(t,s)-\mathcal U(t,s)}
 \le C\e^{C(M_n+M)}\int_0^T\eta_n(r)\,\dd r
 \longrightarrow0.
\end{align*}
Similarly,
\begin{align*}
 \sup_{s\le t\le T}
 \norm{U_{\mathsf d_n}(t,s)-U_{\mathsf d}(t,s)}
 \le\int_0^T
 \norm{\mathsf d_n(r)-\mathsf d(r)}_{L^\infty}\,\dd r
 \longrightarrow0.
\end{align*}
For every $n$, Step~1 gives
$\mathcal U_n(t,s)-U_{\mathsf d_n}(t,s)\in\cK(L^2)$.
The compact operators are closed in the operator norm, hence
$\mathcal U(t,s)-U_{\mathsf d}(t,s)$ is compact.
\end{proof}

The following proposition gives the exact stable--compact difference structure. 
\begin{proposition}
\label{prop:exact-stable-compact-difference}
For every $T>0$ there are strongly measurable real linear operator
fields $S_{T,x,y,\omega},K_{T,x,y,\omega}\in\cL(L^2)$ such that, for
every $x,y\in H^1$ and $\boldsymbol{\gamma}_T$-almost every
$\omega\in E_T$,
\begin{align}\label{eq:snls-exact-stable-compact}
 \Phi_T(y,\omega)-\Phi_T(x,\omega)
 =\bigl(S_{T,x,y,\omega}+K_{T,x,y,\omega}\bigr)(y-x),
\end{align}
where
\begin{align*}
 \norm{S_{T,x,y,\omega}}\le\e^{-\gamma T},\qquad
 K_{T,x,y,\omega}\in\cK(L^2).
\end{align*}
\end{proposition}

\begin{proof}
For $(x,y,\omega)$ such that both $(x,\omega)$ and $(y,\omega)$ belong
to the regular driver domain $\mathfrak D_T$ defined in
\eqref{eq:regular-driver-domain}, put
\begin{align*}
 u=\boldsymbol\Phi_T(x,\omega),\qquad
 v=\boldsymbol\Phi_T(y,\omega),\qquad e=v-u.
\end{align*}
The noise cancels from the equation for $e$.  The decomposition
\begin{align}\label{eq:nonlinear-decomposition}
 |v|^2v-|u|^2u=\mathsf d_{u,v}e+\mathsf c_{u,v}\bar e,
\end{align}
where
\begin{align*}
 \mathsf d_{u,v}
 =2\int_0^1|(1-\theta)u+\theta v|^2\,\dd\theta\ge0,\qquad
 \mathsf c_{u,v}
 =\int_0^1\bigl((1-\theta)u+\theta v\bigr)^2\,\dd\theta
\end{align*}
implies 
\begin{align*}
 \partial_te=-\gamma e-\ii(-\Delta+\mathbf M_{\mathsf d_{u,v}})e
 -\ii \mathbf M_{\mathsf c_{u,v}}\bar e,\qquad e(0)=y-x.
\end{align*}
Moreover,
\begin{align*}
 \norm{\mathsf d_{u,v}(t)}_{L^\infty}
 +\norm{\mathsf c_{u,v}(t)}_{L^\infty}
 \le C\bigl(\norm{u(t)}_{L^\infty}^2+\norm{v(t)}_{L^\infty}^2\bigr),
\end{align*}
Since $u,v\in L^{7/2}(0,T;W^{7/8,7/2})$ and
$W^{7/8,7/2}\hookrightarrow L^\infty$, it follows that
\begin{align*}
 \mathsf d_{u,v}\in L^1(0,T;L^\infty(\T^3;\R)),\qquad
 \mathsf c_{u,v}\in L^1(0,T;L^\infty(\T^3;\C)).
\end{align*}
Set $e_{\gamma}(t)=\e^{\gamma t}e(t)$.  Then
\begin{align*}
 \partial_te_{\gamma}
 =-\ii(-\Delta+\mathbf M_{\mathsf d_{u,v}})e_{\gamma}
 -\ii \mathbf M_{\mathsf c_{u,v}}\overline{e_{\gamma}}.
\end{align*}
Let $\mathcal U_{u,v}(t,s)$ denote its real $L^2$ propagator.  By
\Cref{lem:time-dependent-bog},
\begin{align*}
 \mathcal U_{u,v}(T,0)
 =U_{\mathsf d_{u,v}}(T,0)+\mathcal K_{u,v}(T,0),
 \qquad \mathcal K_{u,v}(T,0)\in\cK(L^2),
\end{align*}
where $U_{\mathsf d_{u,v}}(T,0)$ is unitary on complex $L^2$.
Consequently,
\begin{align*}
 e(T)=\e^{-\gamma T}
 \bigl(U_{\mathsf d_{u,v}}(T,0)+\mathcal K_{u,v}(T,0)\bigr)(y-x).
\end{align*}
Define
\begin{align*}
 S_{T,x,y,\omega}
 =\e^{-\gamma T}U_{\mathsf d_{u,v}}(T,0),\qquad
 K_{T,x,y,\omega}
 =\e^{-\gamma T}\mathcal K_{u,v}(T,0).
\end{align*}
Then $K_{T,x,y,\omega}$ is compact and
\begin{align*}
 \norm{S_{T,x,y,\omega}}=\e^{-\gamma T},
\end{align*}
which proves \eqref{eq:snls-exact-stable-compact}.

On this regular pair domain, the pathwise realization in Appendix~\ref{app:analytic} through \eqref{eq:pathwise-shift} is Borel, and
the maps
\begin{align*}
 (u,v)\longmapsto
 \bigl(\mathsf d_{u,v},\mathsf c_{u,v}\bigr)
\end{align*}
are continuous from bounded energy Strichartz sets into
$L^1(0,T;L^\infty)^2$.  The operator-norm stability estimates in
\Cref{lem:time-dependent-bog} show that both
$U_{\mathsf d}(T,0)$ and $\mathcal U(T,0)$ depend continuously in
operator norm on the corresponding $L^1_tL^\infty_x$ coefficients.
Since
\begin{align*}
 \mathcal K(T,0)=\mathcal U(T,0)-U_{\mathsf d}(T,0),
\end{align*}
the fields $(x,y,\omega)\mapsto S_{T,x,y,\omega}$ and
$(x,y,\omega)\mapsto K_{T,x,y,\omega}$ are strongly measurable.  On
the complement define
\begin{align*}
 S_{T,x,y,\omega}=\e^{-\gamma T}\Id,
 \qquad K_{T,x,y,\omega}=0.
\end{align*}
Since $\mathfrak D_T$ is Borel, the extended fields are strongly
measurable.  By \eqref{eq:full-measure-regular-drivers}, the complement
is $\boldsymbol\gamma_T$-null for every fixed $x,y$, so
\eqref{eq:snls-exact-stable-compact} holds almost surely as asserted.
\end{proof}

\subsection{Dense range of the Malliavin derivative}\label{subsec:reduced-malliavin}

This subsection verifies \Cref{ass:dense-malliavin} for the SNLS.
Since the tangent flow is invertible on $L^2$, we factor it out and
work with the reduced Malliavin range.

All operators in this section act on the realification of $L^2$.  Fix
$T>0$, an initial state $x\in H^1$, and an energy solution $u=\Phi_t(x,\omega)$ on
$[0,T]$.  We suppress $u$ from the propagator notation. For $h\in H_T$, the Malliavin variation
$r_t=\cD\Phi_t(x,\omega)h$ solves
\begin{align*}
 \partial_t r=L_u(t)r+B\dot h,\qquad r(0)=0,
\end{align*}
where
\begin{align*}
 L_u(t)\xi=-\gamma\xi+\ii\Delta\xi
 -\ii\bigl(2|u(t)|^2\xi+u(t)^2\bar\xi\bigr).
\end{align*}
Let $J_{s,t}$ be the propagator of the homogeneous linearized
equation.  Since
\begin{align*}
 L_u(t)=(-\gamma+\ii\Delta)+\mathbf M_u(t),\qquad
 \norm{\mathbf M_u(t)}_{\cL(L^2)}
 \le3\norm{u(t)}_{L^\infty}^2,
\end{align*}
and $u\in L^2(0,T;L^\infty)$, the multiplication perturbation belongs
to $L^1(0,T;\cL(L^2))$.  As $-\gamma+\ii\Delta$ generates a strongly
continuous group on $L^2$, the non-autonomous propagator is invertible.
Put
\begin{align*}
 J_t=J_{0,t},\qquad Q_t=J_t^{-1}.
\end{align*}
Then $J_{s,t}=J_tQ_s$ 
and variation of constants gives
\begin{align*}
 \cA_T h:=\cD\Phi_T(x,\omega)h
 =\int_0^T J_{s,T}B\dot h(s)\,\dd s
 =J_T\widehat{\cA}_Th,
\end{align*}
where
\begin{align}\label{eq:reduced-factor}
 \widehat{\cA}_Th
 =\sum_{j=1}^m\int_0^TQ_tb_j\dot h_j(t)\,\dd t.
\end{align}
Moreover, the forward and inverse Volterra equations imply that, for
every $f\in L^2$,
\begin{align*}
 t\longmapsto J_tf,\qquad t\longmapsto Q_tf
\end{align*}
are continuous in $L^2$, and the same representations give Borel
dependence on the underlying path.  Finally, backward Gronwall gives
\begin{align}\label{eq:l2-inverse-propagator-bound}
 \norm{Q_s^{-1}Q_t}_{\cL(L^2)}
 \le \exp\left\{\gamma(t-s)
 +3\int_s^t\norm{u(r)}_{L^\infty}^2\,\dd r\right\},
 \qquad0\le s\le t\le T.
\end{align}

\subsubsection{The reduced endpoint space}
The next lemma identifies the closure of the reduced Malliavin range.

\begin{lemma}
\label{lem:reduced-factorization}
The space 
\begin{align*}
 \cR(u)=\overline{\spanop_\R\{Q_tb_j:0\le t\le T,\ 1\le j\le m\}}^{\,L^2}
\end{align*}
is a measurable random closed subspace and 
\begin{align*}
 \overline{\Ran\widehat{\cA}_T}^{\,L^2}=\cR(u).
\end{align*}
Consequently $\cD\Phi_T(x,\omega)$ has dense range if and only if $\cR(u)=L^2$.
\end{lemma}

\begin{proof}
Strong continuity of $t\mapsto Q_tb_j$ allows the span defining
$\cR(u)$ to be restricted to rational times.  Hence, for every fixed
$f\in L^2$,
\begin{align*}
 \operatorname{dist}(f,\cR(u))
 =\inf\left\{
 \norm{f-\sum_{\ell=1}^na_\ell Q_{t_\ell}b_{j_\ell}}_{L^2}:
 n\in\N,\ a_\ell\in\mathbb Q,\ 
 t_\ell\in\mathbb Q\cap[0,T]\right\}.
\end{align*}
The Borel dependence of $Q_tb_j$ therefore shows that $\cR(u)$ is a
measurable random closed subspace.

By \eqref{eq:reduced-factor}, every element of
$\Ran\widehat{\cA}_T$ is a Bochner integral of vectors in $\cR(u)$,
and hence
\begin{align*}
 \overline{\Ran\widehat{\cA}_T}^{\,L^2}\subset\cR(u).
\end{align*}
Conversely, fix $t\in[0,T]$ and $j$.  Using Cameron--Martin controls
whose derivatives are normalized indicators of intervals shrinking
to $t$, the strong continuity of $s\mapsto Q_sb_j$ gives
\begin{align*}
 Q_tb_j\in\overline{\Ran\widehat{\cA}_T}^{\,L^2}.
\end{align*}
Thus every generator of $\cR(u)$ belongs to the closed reduced range.
Since $\cA_T=J_T\widehat{\cA}_T$ and $J_T$ is invertible, the final
assertion follows.
\end{proof}

For a smooth constant direction $\phi$, define
\begin{align*}
 \mathscr F_\phi(v)=\gamma\phi-\ii\Delta\phi+\ii D\mathsf N(v)\phi .
\end{align*}
The map $\mathscr F_\phi:H^1\to L^2$ is a polynomial of degree two and
\begin{align*}
 D\mathscr F_\phi(v)a=\ii D^2\mathsf N(v)[a,\phi],
 \qquad
 D^2\mathscr F_\phi(v)[a,b]=\ii D^3\mathsf N[a,b,\phi].
\end{align*}

The next lemma provides absolutely continuous propagation mechanism. 
\begin{lemma}
\label{lem:ac-propagation}
For every smooth $\phi$,
\begin{align}\label{eq:ac-propagation}
 Q_t\phi-Q_s\phi
 =\int_s^tQ_r\mathscr F_\phi(u_r)\,\dd r
 \quad\text{in }L^2.
\end{align}
Hence, if $Q_t\phi\in\cR(u)$ for every $t$, then
\begin{align}\label{eq:fphi-in-range}
 Q_t\mathscr F_\phi(u_t)\in\cR(u),\qquad0\le t\le T.
\end{align}
\end{lemma}

\begin{proof}
Recall that 
\begin{align*}
 L_u(t)=A+\mathbf M_u(t),\qquad
 A=-\gamma+\ii\Delta,
\end{align*}
where
\begin{align*}
 \mathbf M_u(t)\xi
 =-\ii\bigl(2|u(t)|^2\xi+u(t)^2\bar\xi\bigr),
 \qquad
 \norm{\mathbf M_u(t)}_{\cL(L^2)}
 \le3\norm{u(t)}_{L^\infty}^2.
\end{align*}
Since $u\in L^2(0,T;L^\infty)$,
$\mathbf M_u\in L^1(0,T;\cL(L^2))$.  Let
\begin{align*}
 S(t)=\e^{tA},\qquad
 \widetilde J_t=S(-t)J_t,\qquad
 \widetilde{\mathbf M}_u(t)
 =S(-t)\mathbf M_u(t)S(t).
\end{align*}
For every $f\in L^2$, the map
$t\mapsto\widetilde{\mathbf M}_u(t)f$ is strongly measurable and
\begin{align*}
 \norm{\widetilde{\mathbf M}_u(t)f}_{L^2}
 \le 3\norm{u(t)}_{L^\infty}^2\norm f_{L^2}.
\end{align*}
Hence the interaction evolution and its inverse $\widetilde Q_t=\widetilde J_t^{-1}$ satisfy, in the strong sense,
\begin{align*}
 \widetilde J_tf
 =f+\int_0^t
 \widetilde{\mathbf M}_u(r)\widetilde J_rf\,\dd r, \quad \widetilde Q_tf
 =f-\int_0^t
 \widetilde Q_r\widetilde{\mathbf M}_u(r)f\,\dd r.
\end{align*}
Consequently, for every fixed $f\in L^2$,
$t\mapsto\widetilde Q_tf$ is absolutely continuous and
\begin{align*}
 \frac{\dd}{\dd t}\widetilde Q_tf
 =-\widetilde Q_t\widetilde{\mathbf M}_u(t)f
\end{align*}
for almost every $t$.
Since
\begin{align*}
 Q_t=J_t^{-1}=\widetilde Q_tS(-t),
\end{align*}
for every smooth $\phi\in\operatorname{Dom}(A)$ the strong product
rule gives
\begin{align*}
 \frac{\dd}{\dd t}(Q_t\phi)
 =-\widetilde Q_t\widetilde{\mathbf M}_u(t)S(-t)\phi
 -\widetilde Q_tS(-t)A\phi
 =-Q_t\bigl(\mathbf M_u(t)+A\bigr)\phi
 =Q_t\mathscr F_\phi(u_t)
\end{align*}
for almost every $t$.
Moreover,
\begin{align*}
 t\longmapsto\mathscr F_\phi(u_t)
 =\gamma\phi-\ii\Delta\phi+\ii D\mathsf N(u_t)\phi
\end{align*}
is continuous in $L^2$.  Indeed, $u\in C([0,T];H^1)$,
$H^1\hookrightarrow L^4$, and $\phi$ is smooth.  Together with the
strong continuity and local boundedness of $Q_t$, this shows that
$t\mapsto Q_t\mathscr F_\phi(u_t)$ is continuous in $L^2$.
Consequently \eqref{eq:ac-propagation} in fact implies that
$t\mapsto Q_t\phi$ is continuously differentiable as an $L^2$-valued
curve, with derivative $Q_t\mathscr F_\phi(u_t)$.

Now assume that $Q_t\phi\in\cR(u)$ for every $t$.  For
$0<t<T$, the difference quotients
\begin{align*}
 \frac{Q_{t+h}\phi-Q_t\phi}{h}
\end{align*}
belong to the linear space $\cR(u)$ and converge in $L^2$ to
$Q_t\mathscr F_\phi(u_t)$ as $h\to0$.  Since $\cR(u)$ is closed,
\begin{align*}
 Q_t\mathscr F_\phi(u_t)\in\cR(u).
\end{align*}
The same conclusion at $t=0$ and $t=T$ follows from one-sided
difference quotients.  This proves \eqref{eq:fphi-in-range}.
\end{proof}

\subsubsection{Cubic descendants and dense endpoint range}
We now prove a reduced cubic propagation mechanism. For $R\ge1$ let 
\begin{align}\label{eq:adapted-localization-Malliavin}
 \tau_R=\inf\left\{t\le T:\sup_{r\le t}\norm{u_r}_{H^1}
 +\left(\int_0^t\norm{u_r}_{W^{7/8,7/2}}^{7/2}\,\dd r\right)^{2/7}\ge R\right\}\wedge T 
\end{align}
be the stopping times as in \eqref{eq:adapted-localization}.
As in \Cref{subsec:malliavin-freezing}, let $u^0_{s,t}$ and $\widetilde Q^0_{s,t}$ denote respectively the zero noise  flow and its inverse linearized evolution from $u_s$, while $\widetilde Q_{s,t}=J_{s,t}^{-1}$ and $d_{s,t}=u_t-u^0_{s,t}$ and use the notation $\E_s[\cdot]=\E[\cdot\mid\mathcal F_s]$. Note in particular that $u^0_{s,t}$ and $\widetilde Q^0_{s,t}$ are $\mathcal F_s$ measurable for every $t\geq s$. 

\begin{proposition}
\label{prop:two-covariations}
Let $\phi\in C^\infty$ be deterministic.  If
\begin{align*}
 Q_t\phi\in\cR(u)\quad\text{for every }t\in[0,T],
\end{align*}
then, for every forced pair $b_j,b_k$,
\begin{align}
 Q_t\ii D^2\mathsf N(u_t)[b_j,\phi]&\in\cR(u),
 \label{eq:first-descendant}\\
 Q_t\ii D^3\mathsf N[b_k,b_j,\phi]&\in\cR(u)
 \label{eq:second-descendant}
\end{align}
for every $t\in[0,T]$, almost surely.
\end{proposition}

\begin{proof}
The proof is divided into three steps.  

\emph{Step 1: first covariation.}
By \Cref{lem:ac-propagation},
\begin{align*}
 Z_t:=Q_t\mathscr F_\phi(u_t)\in\cR(u),\qquad 0\le t\le T.
\end{align*}
On a deterministic interval
$[s,t]$, $h=t-s$, and on $\{t\le\tau_R\}$, write
\begin{align*}
 Z_t-Z_s=A_{s,t}+\sum_{\ell=1}^mC^\ell_{s,t}\Delta_sW^\ell+\rho_{s,t},
\end{align*}
where
\begin{align*}
 A_{s,t}=Q_s\widetilde Q^0_{s,t}\mathscr F_\phi(u^0_{s,t})
-Q_s\mathscr F_\phi(u_s), \text{ and } C^\ell_{s,t}=Q_s\widetilde Q^0_{s,t}
 D\mathscr F_\phi(u^0_{s,t})b_\ell.
\end{align*}  
Indeed, since $Q_t=Q_s\widetilde Q_{s,t}$ and
$d_{s,t}=u_t-u^0_{s,t}$, the exact quadratic Taylor formula for
$\mathscr F_\phi$ gives
\begin{align*}
 \rho_{s,t}
 =Q_s(\widetilde Q_{s,t}-\widetilde Q^0_{s,t})
       \mathscr F_\phi(u_t)+Q_s\widetilde Q^0_{s,t}
 \left\{D\mathscr F_\phi(u^0_{s,t})(d_{s,t}-B\Delta_sW)
 +\frac12D^2\mathscr F_\phi[d_{s,t},d_{s,t}]\right\}.
\end{align*}
By \eqref{eq:adapted-localization-Malliavin},
\eqref{eq:frozen-zero-noise-bound}, \eqref{eq:l2-tangent-bound}, and
\eqref{eq:l2-inverse-propagator-bound}, all state and linearized
factors in the preceding decomposition are uniformly bounded on
$\{t\le\tau_R\}$ by constants depending only on $R$.  Since
\begin{align*}
 \norm{D\mathscr F_\phi(v)a}_{L^2}
 \le C_{\phi,R}\norm a_{H^1},\qquad
 \norm{D^2\mathscr F_\phi[a,b]}_{L^2}
 \le C_\phi\norm a_{H^1}\norm b_{H^1},
\end{align*}
\Cref{lem:frozen-estimates} gives
\begin{align}\label{eq:covariation-remainder-rate}
 \bigl(\E_s[\one_{\{t\le\tau_R\}}\norm{\rho_{s,t}}_{L^2}^2]\bigr)^{1/2}
 \le C_R(h^{13/14}+h)\one_{\{s<\tau_R\}}.
\end{align}
Here the $h$ term is the quadratic Taylor remainder and follows from
the $p=4$ case of \eqref{eq:frozen-s}.

We next verify the remaining hypotheses of
\Cref{lem:stopped-hilbert-covariation}.  We first give the short-time
continuity used for conditions \eqref{eq:stopped-a-bound} and \eqref{eq:stopped-compensator}.  Let
$G:H^1\to L^2$ be continuous on bounded $H^1$ sets and suppose that,
for some $3/4<\sigma<1$,
\begin{align*}
 \norm{G(v)-G(w)}_{L^2}\le C_R\norm{v-w}_{H^\sigma}
\end{align*}
on such sets.  The free multiplier estimate \begin{align*} \sup_{\norm v_{H^1}\le R} \norm{(\e^{-(\gamma-\ii\Delta)h}-\Id)v}_{H^\sigma} \le C_Rh^{(1-\sigma)/2} \end{align*} and nonlinear estimate \begin{align*} \norm{\int_s^{s+h}\e^{-(\gamma-\ii\Delta)(s+h-a)} \mathsf N(u^0_{s,a})\,\dd a}_{H^\sigma} \le C_R\int_s^{s+h}\norm{u^0_{s,a}}_{L^\infty}^2\,\dd a\le C_Rh^{3/7} \end{align*} obtained by \Cref{lem:periodic-toolkit} and \eqref{eq:frozen-zero-noise-bound}, imply that 
uniformly for $s<\tau_R$,
\begin{align*}
 \norm{u^0_{s,s+h}-u_s}_{H^\sigma}
 \le C_R\bigl(h^{(1-\sigma)/2}+h^{3/7}\bigr)\longrightarrow0.
\end{align*}
For a fixed sample path,
$\mathcal C_u=\{u_s:0\le s\le T\}$ is compact in $H^1$, hence
$G(\mathcal C_u)$ is compact in $L^2$.  By
\eqref{eq:frozen-m-bound} and \eqref{eq:Rvst},
\begin{align*}
 \norm{(\widetilde Q^0_{s,s+h}-\Id)f}_{L^2}
 \le C_Rh^{3/7}\norm f_{L^2}
 +\norm{(\e^{(\gamma-\ii\Delta)h}-\Id)f}_{L^2}.
\end{align*}
Strong continuity of the free group is uniform on the compact set
$G(\mathcal C_u)$.  Consequently,
\begin{align}\label{eq:frozen-composition-continuity}
 \sup_{s<\tau_R}
 \norm{\widetilde Q^0_{s,s+h}G(u^0_{s,s+h})-G(u_s)}_{L^2}
 \longrightarrow0.
\end{align}
The maps $G=\mathscr F_\phi$ and
$G=D\mathscr F_\phi(\,\cdot\,)b_\ell$ satisfy the preceding
assumptions because $H^\sigma\hookrightarrow L^4$ and
$\mathscr F_\phi$ is quadratic.  Thus
\eqref{eq:frozen-composition-continuity},
\eqref{eq:l2-inverse-propagator-bound} and 
\eqref{eq:adapted-localization-Malliavin}  imply
\begin{align*}
 \sup_{s<\tau_R}\norm{A_{s,s+h}}_{L^2}\longrightarrow0,
 \qquad
 \sup_{s<\tau_R,\ 0<h\le h_0}\norm{A_{s,s+h}}_{L^2}\le C_R.
\end{align*}
Returning to the uniform deterministic partitions
$\Pi_n=\{t_i\}$ of $[a,b]\subset [0,T]$ in \Cref{lem:stopped-hilbert-covariation}, with
mesh $h_n\downarrow0$ and $\tau=\tau_R\wedge b$  we obtain
\begin{align*}
 \max_i\one_{\{t_i<\tau\}}
 \norm{A_{t_i,t_{i+1}}}_{L^2}\longrightarrow0
 \qquad\text{a.s.}
\end{align*}
and dominated convergence gives
\begin{align*}
 \E\max_i\one_{\{t_i<\tau\}}
 \norm{A_{t_i,t_{i+1}}}_{L^2}^2\longrightarrow0,
\end{align*}
which is \eqref{eq:stopped-a-bound}. Applying \eqref{eq:frozen-composition-continuity} with
$G=D\mathscr F_\phi(\,\cdot\,)b_\ell$ yields
\begin{align*}
 C^\ell_{s,s+h}\longrightarrow Q_sD\mathscr F_\phi(u_s)b_\ell
\end{align*}
uniformly along the stopped path.  The uniform bound in
\eqref{eq:stopped-c-bound} follows from
\eqref{eq:l2-inverse-propagator-bound},
\eqref{eq:adapted-localization-Malliavin}, and
\eqref{eq:frozen-zero-noise-bound}; continuity of $t\mapsto Q_tD\mathscr F_\phi(u_t)b_\ell$ and dominated
convergence then give the left-Riemann-sum limit
\eqref{eq:stopped-compensator} with
\begin{align*}
 C_t^\ell=Q_tD\mathscr F_\phi(u_t)b_\ell.
\end{align*}
Finally, 
\eqref{eq:covariation-remainder-rate} gives
\begin{align*}
 h_n^{1/2}\sum_i
 \left(\E\one_{\{t_{i+1}\le\tau\}}
 \norm{\rho_{t_i,t_{i+1}}}_{L^2}^2\right)^{1/2}
 \le C_{R,T}h_n^{-1/2}(h_n^{13/14}+h_n)
 =C_{R,T}(h_n^{3/7}+h_n^{1/2})
 \longrightarrow0,
\end{align*}
which verifies \eqref{eq:stopped-remainder}.

Thus \Cref{lem:stopped-hilbert-covariation} applies on every rational
interval $I=[a,b]\subset[0,T]$, with stopping time
$\tau=\tau_R\wedge b$.
Since every increment of $Z_{\cdot\wedge\tau}$ belongs to the
linear space $\cR(u)$, every covariation sum does as well.  Passing to
an almost surely convergent subsequence and using the closedness of
$\cR(u)$ therefore gives
\begin{align}\label{eq:first-covariation-integral}
 \int_I\one_{\{t<\tau_R\}}
 Q_tD\mathscr F_\phi(u_t)b_j\,\dd t
 =\int_I\one_{\{t<\tau_R\}}
 Q_t\ii D^2\mathsf N(u_t)[b_j,\phi]\,\dd t
 \in\cR(u).
\end{align}
The integrand is continuous in $L^2$.  For any fixed $t<\tau_R$,
take rational intervals $I_n$ shrinking to $t$.  The averages of
\eqref{eq:first-covariation-integral} belong to $\cR(u)$ and converge
in $L^2$ to
\begin{align*}
 Q_t\ii D^2\mathsf N(u_t)[b_j,\phi].
\end{align*}
Closedness of $\cR(u)$ proves \eqref{eq:first-descendant} for
$t<\tau_R$, and the endpoint follows by one-sided continuity.

\emph{Step 2: second covariation.}
Set
\begin{align*}
 \mathcal N_{j,\phi}(v)=\ii D^2\mathsf N(v)[b_j,\phi].
\end{align*}
Since the nonlinearity is cubic, $\mathcal N_{j,\phi}$ is linear in $v$ and
\begin{align*}
 D\mathcal N_{j,\phi}(v)b_k=\ii D^3\mathsf N[b_k,b_j,\phi],
 \qquad D^2\mathcal N_{j,\phi}=0.
\end{align*}
By Step~1,
\begin{align*}
 \widetilde Z_t:=Q_t\mathcal N_{j,\phi}(u_t)\in\cR(u).
\end{align*}
Replacing $\mathscr F_\phi$ in the frozen decomposition of Step~1 by $\mathcal N_{j,\phi}$ gives, on $\{t\le\tau_R\}$,
\begin{align*}
 \widetilde Z_t-\widetilde Z_s
 =\widetilde A_{s,t}
 +\sum_{\ell=1}^m\widetilde C^\ell_{s,t}\Delta_sW^\ell
 +\widetilde\rho_{s,t},
\end{align*}
where
\begin{align*}
 \widetilde A_{s,t}
 &=Q_s\widetilde Q^0_{s,t}\mathcal N_{j,\phi}(u^0_{s,t})
   -Q_s\mathcal N_{j,\phi}(u_s),\\
 \widetilde C^\ell_{s,t}
 &=Q_s\widetilde Q^0_{s,t}
   \ii D^3\mathsf N[b_\ell,b_j,\phi],\\
 \widetilde\rho_{s,t}
 &=Q_s(\widetilde Q_{s,t}-\widetilde Q^0_{s,t})\mathcal N_{j,\phi}(u_t)
   +Q_s\widetilde Q^0_{s,t}
   \ii D^3\mathsf N[d_{s,t}-B\Delta_sW,b_j,\phi].
\end{align*}
The $p=2$ cases of \eqref{eq:frozen-endpoint} and
\eqref{eq:frozen-inverse} therefore give
\begin{align*}
 \bigl(\E_s[\one_{\{t\le\tau_R\}}
 \norm{\widetilde\rho_{s,t}}_{L^2}^2]\bigr)^{1/2}
 \le C_Rh^{13/14}\one_{\{s<\tau_R\}}
\end{align*}
implying 
\eqref{eq:stopped-remainder}, since on a uniform partition of $[a,b]\subset [0,T]$ with $\tau=b\wedge\tau_R$ and mesh
$h_n$,
\begin{align*}
 h_n^{1/2}\sum_i
 \left(\E\one_{\{t_{i+1}\le\tau\}}
 \norm{\widetilde\rho_{t_i,t_{i+1}}}_{L^2}^2\right)^{1/2}
 \le C_{R,T}h_n^{3/7}\longrightarrow0.
\end{align*}
The condition \eqref{eq:stopped-a-bound} follows from
\eqref{eq:frozen-composition-continuity} applied to
$G=\mathcal N_{j,\phi}$.  Applying the same estimate to the constant
map $G=\ii D^3\mathsf N[b_\ell,b_j,\phi]$, together with
\eqref{eq:l2-inverse-propagator-bound}, gives the uniform
$\widetilde C^\ell$-bound \eqref{eq:stopped-c-bound} and its compensator convergence \eqref{eq:stopped-compensator}.  Hence
\Cref{lem:stopped-hilbert-covariation},
followed by an almost surely convergent subsequence and the closedness
of $\cR(u)$, gives, for every rational interval $I$,
\begin{align*}
 \int_I\one_{\{t<\tau_R\}}
 Q_t\ii D^3\mathsf N[b_k,b_j,\phi]\,\dd t\in\cR(u).
\end{align*}
Shrinking rational intervals to a point and using continuity and
closedness gives
\begin{align*}
 Q_t\ii D^3\mathsf N[b_k,b_j,\phi]\in\cR(u),
 \qquad t<\tau_R,
\end{align*}
and hence \eqref{eq:second-descendant} up to the stopping time.

\emph{Step 3: removal of the localization and a common null set.}
Take $R\in\N$, rational intervals equipped with dyadic meshes, the
finitely many Brownian coordinates, and a fixed real basis of each
finite-dimensional cubic saturation level $\mathscr W_n$ in
\eqref{eq:cubic-recursion}.  These choices form a countable
family of covariation limits.  For each limit choose a deterministic
subsequence converging almost surely and intersect the corresponding
full-measure events.  We thereby obtain a single full-measure event
on which the conclusions of Steps~1--2 hold for every basis element
at every cubic saturation level.  Real multilinearity then extends
them to every vector in each $\mathscr W_n$.

Finally, $\tau_R\uparrow T$ almost surely.  Hence, for every $t<T$,
one has $t<\tau_R$ for all sufficiently large $R$, and
\eqref{eq:first-descendant}--\eqref{eq:second-descendant} hold at $t$.
The conclusion at $t=T$ follows from one-sided continuity and the
closedness of $\cR(u)$.
\end{proof}

Combining the above cubic propagation mechanism with cubic saturation, we obtain the density of the Malliavin range. 
\begin{theorem}\label{thm:dense-endpoint}
Suppose that the smooth forcing is cubic saturating as in \Cref{def:cubic-saturation}.  For every
$T>0$, every $x\in H^1$, and almost every Wiener path,
\begin{align}\label{eq:dense-endpoint}
 \overline{\Ran \cD\Phi_T(x,\omega)}^{\,L^2}=L^2 .
\end{align}
\end{theorem}

\begin{proof}
By definition of $\cR(u)$, $Q_tb_\ell\in\cR(u)$ for every forced
direction.  By real trilinearity,
\Cref{prop:two-covariations} extends from the pairs $b_j,b_k$ to every
pair $h,k\in\mathscr W_0=\Ran B$.  If $Q_t\phi\in\cR(u)$ for every $t$, then 
\begin{align*}
 Q_t\mathsf JD^3\mathsf N[\phi,h,k]\in\cR(u)
 \qquad(h,k\in\mathscr W_0).
\end{align*}
Induction in \eqref{eq:cubic-recursion} gives
$Q_t\mathscr W_\infty\subset\cR(u)$ on the single full-measure event
constructed in \Cref{prop:two-covariations}.  Cubic saturation
makes $\mathscr W_\infty$ dense in $H^1$ and hence in $L^2$.  Since $Q_t$ is
invertible and $\cR(u)$ is closed, $\cR(u)=L^2$.
\Cref{lem:reduced-factorization} proves \eqref{eq:dense-endpoint}.
\end{proof}

\subsection{Uniform irreducibility}

This subsection derives uniform irreducibility and full support from
cubic saturation through an energy level approximate controllability
argument.  For \eqref{eq:intro-spde}, deterministic attraction to the
origin would already suffice for uniform irreducibility.  We instead
use approximate controllability, since the mechanism is robust under
the addition of a smooth deterministic forcing. Throughout this subsection, write
\begin{align*}
 \mathsf S(t)=\e^{-(\gamma-\ii\Delta)t}
\end{align*}
for the damped free propagator.  For $g\in H_\tau$, let
$\Phi_t^g(x)$, $0\le t\le\tau$, denote the controlled solution of
\begin{align*}
 \partial_tu=-\gamma u+\ii\Delta u-\ii\mathsf N(u)+B\dot g(t),
 \qquad u(0)=x.
\end{align*}
All controls below may be chosen with piecewise constant derivative.

The following lemma is the energy level cubic loop underlying the
saturation induction.

\begin{lemma}
\label{lem:energy-ideal-cubic-loop}
Let $\zeta\in C^\infty$, $M<\infty$, and $a\in\R\setminus\{0\}$.  Put
\begin{align*}
 \zeta_a=\operatorname{sgn}(a)\zeta,\qquad
 t_R=\abs aR^{-3},\qquad
 \mathcal T_\eta u=u+\eta.
\end{align*}
Then, as $R\to\infty$,
\begin{align}\label{eq:energy-ideal-loop}
 \sup_{\norm u_{H^1}\le M}
 \norm{\mathcal T_{-R\zeta_a}\Phi_{t_R}^0
 \mathcal T_{R\zeta_a}u
 -\bigl(\mathsf S(t_R)u-a\ii\mathsf N(\zeta)\bigr)}_{H^1}
 \longrightarrow0.
\end{align}
\end{lemma}

\begin{proof}
Write
\begin{align*}
 v_R(t;u)=\Phi_t^0(u+R\zeta_a)-R\zeta_a.
\end{align*}
The mild equation gives
\begin{align*}
 v_R(t;u)
 =\mathsf S(t)u+R(\mathsf S(t)-\Id)\zeta_a
 -\ii\int_0^t\mathsf S(t-r)
 \mathsf N(R\zeta_a+v_R(r;u))\,\dd r.
\end{align*}
Expand
\begin{align*}
 \mathsf N(R\zeta_a+v)
 =R^3\mathsf N(\zeta_a)
 +R^2P_1(\zeta_a,v)+RP_2(\zeta_a,v)+\mathsf N(v),
\end{align*}
where $P_1$ is linear and $P_2$ is quadratic in $v$, with smooth
coefficients depending on $\zeta_a$.

By \Cref{lem:periodic-toolkit}, the same local energy--Strichartz
bootstrap as in \eqref{eq:block-local-bootstrap} gives, for all
sufficiently large $R$,
\begin{align}\label{eq:energy-vr-bound}
 \sup_{\norm u_{H^1}\le M}\left\{
 \norm{v_R(\,\cdot\,;u)}_{C(0,t_R;H^1)}
 +\norm{v_R(\,\cdot\,;u)}_{L^{7/2}(0,t_R;W^{7/8,7/2})}
 \right\}\le C_{M,a,\zeta}.
\end{align}
Indeed, the $R^3\mathsf N(\zeta_a)$ term has
$L^1(0,t_R;H^1)$-norm $O(1)$, while on the bootstrap set
\begin{align*}
 R^2\norm{P_1(\zeta_a,v_R)}_{L^1H^1}
 &\le CR^2t_R\norm{v_R}_{CH^1}=O(R^{-1}),\\
 R\norm{P_2(\zeta_a,v_R)}_{L^1H^1}
 &\le CRt_R^{5/7}\norm{v_R}_{CH^1}
 \norm{v_R}_{L^{7/2}L^\infty}=O(R^{-8/7}),\\
 \norm{\mathsf N(v_R)}_{L^1H^1}
 &\le Ct_R^{3/7}\norm{v_R}_{CH^1}
 \norm{v_R}_{L^{7/2}L^\infty}^2=O(R^{-9/7}),
\end{align*}
where $t_R=O(R^{-3})$ and
$W^{7/8,7/2}\hookrightarrow L^\infty$. 

Since $R^3t_R\mathsf N(\zeta_a)=a\mathsf N(\zeta)$, we have 
\begin{align*}
 v_R(t_R;u)-\bigl(\mathsf S(t_R)u&-a\ii\mathsf N(\zeta)\bigr)
 =R(\mathsf S(t_R)-\Id)\zeta_a-\ii R^3\int_0^{t_R}
 (\mathsf S(t_R-r)-\Id)\mathsf N(\zeta_a)\,\dd r\\
 &-\ii R^2\int_0^{t_R}\mathsf S(t_R-r)
 P_1(\zeta_a,v_R(r;u))\,\dd r\\
 &-\ii R\int_0^{t_R}\mathsf S(t_R-r)
 P_2(\zeta_a,v_R(r;u))\,\dd r\\
 &-\ii\int_0^{t_R}\mathsf S(t_R-r)
 \mathsf N(v_R(r;u))\,\dd r. 
\end{align*}
As  $\zeta$ is smooth,
\begin{align*}
 R\norm{(\mathsf S(t_R)-\Id)\zeta_a}_{H^1}
 &\le CRt_R\norm{\zeta_a}_{H^3}=O(R^{-2}),\\
 R^3\int_0^{t_R}
 \norm{(\mathsf S(t_R-r)-\Id)\mathsf N(\zeta_a)}_{H^1}\,\dd r
 &\le CR^3t_R^2\norm{\mathsf N(\zeta_a)}_{H^3}
 =O(R^{-3}).
\end{align*}
The remaining three terms tend to zero by
\eqref{eq:energy-vr-bound} and the preceding estimates.  This proves
\eqref{eq:energy-ideal-loop}.
\end{proof}

The next lemma propagates these fast translations through the reduced
cubic saturation.

\begin{lemma}
\label{lem:reduced-fast-translation}
Assume cubic saturation as in
\Cref{def:cubic-saturation}.  For every
$\zeta\in\mathscr W_\infty$, $a\in\R$, $M<\infty$, and
$\delta,\epsilon>0$, there are $0<\theta<\delta$ and
$g\in H_\theta$ with piecewise constant derivative such that
\begin{align}\label{eq:energy-fast-translation}
 \sup_{\norm u_{H^1}\le M}
 \norm{\Phi_\theta^g(u)
 -\bigl(\mathsf S(\theta)u+a\zeta\bigr)}_{H^1}<\epsilon.
\end{align}
\end{lemma}

\begin{proof}
We argue by induction in the reduced saturation recursion.

\emph{Step 1: the forced level.}
Let $\zeta\in\mathscr W_0=\Ran B$ and choose $e\in\R^m$ with
$Be=\zeta$.  Set
\begin{align*}
 g_\theta(t)=\frac{at}{\theta}e,\qquad0\le t\le\theta,
\end{align*}
so that $B\dot g_\theta=a\zeta/\theta$.  The mild equation gives
\begin{align*}
 \Phi_\theta^{g_\theta}(u)
 =\mathsf S(\theta)u
 +a\int_0^1\mathsf S(\theta(1-\rho))\zeta\,\dd\rho
 -\ii\theta\int_0^1\mathsf S(\theta(1-\rho))
 \mathsf N(\Phi_{\theta\rho}^{g_\theta}(u))\,\dd\rho.
\end{align*}
\Cref{lem:periodic-toolkit} and the local energy--Strichartz
bootstrap leading to \eqref{eq:block-local-bootstrap} give
\begin{align*}
 \sup_{\norm u_{H^1}\le M}
 \left\{
 \norm{\Phi_\cdot^{g_\theta}(u)}_{C(0,\theta;H^1)}
 +\norm{\Phi_\cdot^{g_\theta}(u)}_{L^{7/2}(0,\theta;W^{7/8,7/2})}
 \right\}
 \le C_{M,a,\zeta}
\end{align*}
for all sufficiently small $\theta$.  Hence
\begin{align*}
 \sup_{\norm u_{H^1}\le M}
 \norm{\mathsf N(\Phi_\cdot^{g_\theta}(u))}_{L^1(0,\theta;H^1)}
 \le C_{M,a,\zeta}\theta^{3/7},
\end{align*}
while smoothness of $\zeta$ gives
\begin{align*}
 \norm{a\int_0^1
 (\mathsf S(\theta(1-\rho))-\Id)\zeta\,\dd\rho}_{H^1}
 \le C_{a,\zeta}\theta.
\end{align*}
Thus
\begin{align*}
 \sup_{\norm u_{H^1}\le M}
 \norm{\Phi_\theta^{g_\theta}(u)
 -\bigl(\mathsf S(\theta)u+a\zeta\bigr)}_{H^1}
 \le C_{M,a,\zeta}(\theta+\theta^{3/7})
 \longrightarrow0.
\end{align*}
This proves \eqref{eq:energy-fast-translation} on $\mathscr W_0$.

\emph{Step 2: the reduced cubic induction.}
Assume \eqref{eq:energy-fast-translation} holds for every direction
in $\mathscr W_n$.  Since $\mathscr W_0\subset\mathscr W_n$, it
suffices by \eqref{eq:cubic-recursion} to treat
\begin{align*}
 \mathsf JD^3\mathsf N[\phi,h,k],
 \qquad \phi\in\mathscr W_n,\quad h,k\in\mathscr W_0.
\end{align*}
For $\varepsilon=(\varepsilon_1,\varepsilon_2,\varepsilon_3)
\in\{-1,1\}^3$ set
\begin{align*}
 \xi_\varepsilon
 =\varepsilon_1\phi+\varepsilon_2h+\varepsilon_3k\in\mathscr W_n.
\end{align*}
Fixing $\xi=\xi_\varepsilon$, $a\ne0$, and applying
\Cref{lem:energy-ideal-cubic-loop} with 
\begin{align*}
 \xi_a=\operatorname{sgn}(a)\xi,\qquad t_R=\abs aR^{-3},
\end{align*}
one has 
\begin{align}\label{eq:kappaR}
 \kappa_R
 =\sup_{\norm w_{H^1}\le M}
 \norm{\mathcal T_{-R\xi_a}\Phi_{t_R}^0
 \mathcal T_{R\xi_a}w
 -\bigl(\mathsf S(t_R)w-a\ii\mathsf N(\xi)\bigr)}_{H^1}\to0.
\end{align}

Let $R, \eta>0$.  By the induction hypothesis, there are
arbitrarily small $\theta_R^+>0$ and a control $g_R^+$ such that
\begin{align*}
 \Phi_{\theta_R^+}^{g_R^+}(u)
 =\mathsf S(\theta_R^+)u+R\xi_a+e_R^+(u),
 \qquad
 \sup_{\norm u_{H^1}\le M}\norm{e_R^+(u)}_{H^1}<\eta.
\end{align*}
The standard bounded-tube $H^1$ stability of the zero control flow gives
\begin{align*}
 \Phi_{t_R}^0\bigl(\Phi_{\theta_R^+}^{g_R^+}(u)\bigr)
 =\Phi_{t_R}^0\bigl(\mathsf S(\theta_R^+)u+R\xi_a\bigr)
 +e_R^0(u),
 \qquad
 \sup_{\norm u_{H^1}\le M}\norm{e_R^0(u)}_{H^1}
 \le C_R\eta.
\end{align*}
Since $\norm{\mathsf S(\theta_R^+)u}_{H^1}\le\norm u_{H^1}$, it follows from \eqref{eq:kappaR} with $w=\mathsf S(\theta_R^+)u$ that 
\begin{align*}
 \Phi_{t_R}^0\bigl(\mathsf S(\theta_R^+)u+R\xi_a\bigr)
 -R\xi_a=\mathsf S(t_R+\theta_R^+)u
 -a\ii\mathsf N(\xi)+e_R^{\rm id}(u),
\end{align*}
where
\begin{align*}
 \sup_{\norm u_{H^1}\le M}
 \norm{e_R^{\rm id}(u)}_{H^1}\le\kappa_R.
\end{align*}
Thus
\begin{align}\label{eq:bounded-middle-word}
 \Phi_{t_R}^0\circ\Phi_{\theta_R^+}^{g_R^+}(u)
 =R\xi_a+\mathsf S(t_R+\theta_R^+)u
 -a\ii\mathsf N(\xi)+\widetilde e_R(u),
\end{align}
with
\begin{align*}
 \sup_{\norm u_{H^1}\le M}
 \norm{\widetilde e_R(u)}_{H^1}
 \le\kappa_R+C_R\eta.
\end{align*}

For fixed $R$, the middle states in \eqref{eq:bounded-middle-word}
remain in an $H^1$-ball of some radius $M_R<\infty$.  Apply the
induction hypothesis on this ball to choose an arbitrarily small
$\theta_R^->0$ and a control $g_R^-$ such that
\begin{align*}
 \Phi_{\theta_R^-}^{g_R^-}(v)
 =\mathsf S(\theta_R^-)v-R\xi_a+e_R^-(v),
 \qquad
 \sup_{\norm v_{H^1}\le M_R}\norm{e_R^-(v)}_{H^1}<\eta.
\end{align*}
Substituting \eqref{eq:bounded-middle-word} gives, with
$\theta_R=\theta_R^++t_R+\theta_R^-$,
\begin{align*}
 &\Phi_{\theta_R^-}^{g_R^-}\circ\Phi_{t_R}^0
 \circ\Phi_{\theta_R^+}^{g_R^+}(u)\\
 &\quad=\mathsf S(\theta_R)u-a\ii\mathsf N(\xi)
 +R\bigl(\mathsf S(\theta_R^-)\xi_a-\xi_a\bigr)
 -a\ii\bigl(\mathsf S(\theta_R^-)\mathsf N(\xi)-\mathsf N(\xi)\bigr)
 +E_R(u),
\end{align*}
where
\begin{align*}
 \sup_{\norm u_{H^1}\le M}\norm{E_R(u)}_{H^1}
 \le\kappa_R+C_R\eta+\eta.
\end{align*}
Since $\xi$ and $\mathsf N(\xi)$ are smooth, for fixed $R$ the
duration $\theta_R^-$ may be chosen so small that
\begin{align*}
 R\norm{(\mathsf S(\theta_R^-)-\Id)\xi_a}_{H^1}
 +|a|\norm{(\mathsf S(\theta_R^-)-\Id)\mathsf N(\xi)}_{H^1}
\end{align*}
is arbitrarily small.  Choosing first $R$ large, so that
$t_R$ and $\kappa_R$ are small, and then the two pulse durations and
$\eta$ sufficiently small, proves \eqref{eq:energy-fast-translation} for $\zeta=-\ii\mathsf N(\xi_\varepsilon)$. The real polarization identity
\begin{align*}
 D^3\mathsf N[\phi,h,k]
 =\frac18\sum_{\varepsilon\in\{-1,1\}^3}
 \varepsilon_1\varepsilon_2\varepsilon_3
 \mathsf N(\xi_\varepsilon)
\end{align*}
therefore yields the same property for
$\mathsf JD^3\mathsf N[\phi,h,k]$.  Finite real linear combinations
are obtained by concatenating the corresponding short controls; the
only additional errors are again of the form
$\mathsf S(s)\zeta-\zeta$, with $\zeta$ a smooth saturation direction.
Hence \eqref{eq:energy-fast-translation} holds on $\mathscr W_{n+1}$.
Induction proves the lemma.
\end{proof}

The support property is now immediate.

\begin{corollary}
\label{cor:full-support}
Assume cubic saturation.  Every invariant probability measure
of $(P_t)_{t\ge0}$ has full $H^1$-support.
\end{corollary}

\begin{proof}
Let $\nu$ be invariant and $O\subset H^1$ be nonempty and open.
Choose $x_0\in\supp\nu$ and $y\in O$.  By density of
$\mathscr W_\infty$, choose $p\in\mathscr W_\infty$ such that
$x_0+p$ lies sufficiently close to $y$.  Since
$\mathsf S(\theta)x_0\to x_0$ in $H^1$, \Cref{lem:reduced-fast-translation}
gives, for some sufficiently small $\theta>0$ and $g\in H_\theta$,
\begin{align*}
 \Phi_\theta^g(x_0)\in O.
\end{align*}
Since $(x_0,g)\in\mathfrak D_\theta$, continuity with respect to the initial data and the driver 
give neighborhoods $U$ of $x_0$ in $H^1$ and $\mathcal O_g$ of $g$
in $E_\theta$  such that all the corresponding solutions exist up to time
$\theta$ and
$\Phi_\theta(x,\omega)\in O$ for $x\in U$ and
$\omega\in\mathcal O_g$.  Since
$\nu(U)>0$ and $\boldsymbol{\gamma}_\theta(\mathcal O_g)>0$,
invariance gives
\begin{align*}
 \nu(O)\ge
 \boldsymbol{\gamma}_\theta(\mathcal O_g)\nu(U)>0.
\end{align*}
\end{proof}

The next proposition on uniform irreducibility verifies \Cref{ass:common-accessibility} for the SNLS  used in the
coupling argument.

\begin{proposition}
\label{prop:reduced-common-accessibility}
Assume cubic saturation as in \Cref{def:cubic-saturation}.  There exist
$u_*\in H^1$ and $R_*<\infty$ such that, for every
$R<\infty$ and $r>0$, there are $T=T(R,r)>0$ and
$p=p(R,r)>0$ satisfying
\begin{align*}
 \inf_{x\in\mathbb V_R}
 P_T\bigl(x,\mathbb V_{R_*}
 \cap B_{L^2}(u_*,r)\bigr)\ge p.
\end{align*}
\end{proposition}

\begin{proof}
We first reduce the initial Lyapunov sublevels to one fixed source
core.  Let $c_1,C_1>0$ be the constants in
\Cref{prop:lyapunov} for $r=1$ there.  Since $\cV=1+\cE$,
\eqref{eq:polynomial-foster} gives
\begin{align}\label{eq:access-lyapunov-return}
 P_t\cV(x)=\E_x\cV(u_t)
 \le \e^{-c_1t}\cV(x)+C_1.
\end{align}
Set $R_{\rm b}=4C_1$ and  for $R<\infty$, choose $s_R$ so large that
\begin{align*}
 \e^{-c_1s_R}R\le C_1.
\end{align*}
Then for $x\in\mathbb V_R$, Markov's inequality and
\eqref{eq:access-lyapunov-return} give
\begin{align*}
 P_{s_R}(x,\mathbb V_{R_{\rm b}}^c)
 =\Prob\{\cV(u_{s_R})>R_{\rm b}\}
 \le\frac12 
\end{align*}
implying 
\begin{align}\label{eq:access-fixed-source-core}
 \inf_{x\in\mathbb V_R}
 P_{s_R}(x,\mathbb V_{R_{\rm b}})\ge\frac12.
\end{align}
By energy coercivity \eqref{eq:energy-coercivity}, we may choose $M_{\rm b}<\infty$ such that
\begin{align*}
 \mathbb V_{R_{\rm b}}
 \subset\{u\in H^1:\norm u_{H^1}\le M_{\rm b}\}.
\end{align*}
Fix any $u_*\in H^1$ and set
\begin{align*}
 \varepsilon=\min\{1,r/32\},\qquad
 L=2M_{\rm b}+\norm{u_*}_{H^1}+2.
\end{align*}
The standard deterministic energy bound gives $L_+<\infty$ such that
\begin{align}\label{eq:short-padding-bound}
 \sup_{\norm v_{H^1}\le L}
 \norm{\Phi_t^0(v)}_{H^1}\le L_+, \quad 0\le t\le1. 
\end{align}
Moreover, the mild equation and
$\norm{(\mathsf S(t)-\Id)v}_{L^2}\le Ct^{1/2}\norm v_{H^1}$ give
\begin{align}\label{eq:short-padding-modulus}
 \sup_{\substack{\norm v_{H^1}\le L\\0\le t\le\delta}}
 \norm{\Phi_t^0(v)-v}_{L^2}\longrightarrow0
 \qquad\text{as }\delta\downarrow0,
\end{align}
since
\begin{align*}
 \int_0^t\norm{\mathsf N(\Phi_s^0(v))}_{L^2}\,\dd s
 \le Ct\sup_{s\le t}\norm{\Phi_s^0(v)}_{H^1}^3.
\end{align*}
Choose $0<\delta_r<1$ such that
\begin{align*}
 \sup_{\substack{\norm v_{H^1}\le L\\0\le t\le\delta_r}}
 \norm{\Phi_t^0(v)-v}_{L^2}<r/8,
 \qquad
 \sup_{\substack{\norm v_{H^1}\le M_{\rm b}\\0\le t\le\delta_r}}
 \norm{(\mathsf S(t)-\Id)v}_{L^2}<\varepsilon.
\end{align*}

Since $\mathbb V_{R_{\rm b}}$ is compact in $L^2$, choose an
$\varepsilon$-net $x_1,\ldots,x_N$.  Cubic saturation gives
$p_j\in\mathscr W_\infty$ such that
\begin{align*}
 \norm{x_j+p_j-u_*}_{H^1}<\varepsilon.
\end{align*}
In particular,
\begin{align*}
 \norm{p_j}_{H^1}
 \le M_{\rm b}+\norm{u_*}_{H^1}+1.
\end{align*}
Apply \Cref{lem:reduced-fast-translation} with
$\zeta=p_j$, $a=1$, $M=M_{\rm b}$, duration below $\delta_r$, and
error below $\varepsilon$.  We obtain $\theta_j<\delta_r$ and
$g_j\in H_{\theta_j}$ such that, for every
$\norm z_{H^1}\le M_{\rm b}$,
\begin{align*}
 \norm{\Phi_{\theta_j}^{g_j}(z)
 -\bigl(\mathsf S(\theta_j)z+p_j\bigr)}_{H^1}<\varepsilon,
\end{align*}
which particularly shows $\norm{\Phi_{\theta_j}^{g_j}(z)}_{H^1}\le L$. Therefore, for $\norm{z-x_j}_{L^2}<\varepsilon$, 
\begin{align*}
 \norm{\Phi_{\theta_j}^{g_j}(z)-u_*}_{L^2}
 &\le\norm{\mathsf S(\theta_j)(z-x_j)}_{L^2}
 +\norm{(\mathsf S(\theta_j)-\Id)x_j}_{L^2}\\
 &\quad+\norm{x_j+p_j-u_*}_{L^2}+\varepsilon
 <4\varepsilon\le r/8.
\end{align*}
Put $\tau=\max_j\theta_j$ and extend each $g_j$ constantly on
$[\theta_j,\tau]$.  By
\eqref{eq:short-padding-bound}--\eqref{eq:short-padding-modulus},
\begin{align*}
 \norm{\Phi_\tau^{g_j}(z)-u_*}_{L^2}<r/4,
 \qquad
 \norm{\Phi_\tau^{g_j}(z)}_{H^1}\le L_+
\end{align*}
whenever $z\in\mathbb V_{R_{\rm b}}$ and
$\norm{z-x_j}_{L^2}<\varepsilon$.

All these piecewise smooth controlled trajectories belong to
$\mathfrak D_\tau$.  The controlled energy--Strichartz estimates used in
\Cref{lem:reduced-fast-translation}, together with the zero control
padding, give
\begin{align*}
 \max_{1\le j\le N}\sup_{z\in\mathbb V_{R_{\rm b}}}
 \left\{
 \norm{\Phi_\cdot^{g_j}(z)}_{C(0,\tau;H^1)}
 +\norm{\Phi_\cdot^{g_j}(z)}_{L^{7/2}(0,\tau;W^{7/8,7/2})}
 \right\}<\infty.
\end{align*}
For each fixed $j$, apply \Cref{lem:regular-core-local-tube} at
$(z,g_j)$ for $z\in\mathbb V_{R_{\rm b}}$.  A finite $L^2$ subcover
of the compact
set $\mathbb V_{R_{\rm b}}$, and the intersection of the corresponding
finitely many driver neighborhoods, give a common
energy--Strichartz tube.  The bounded-tube estimate
\eqref{eq:Lipschitz-stability} then gives $\eta_j>0$ such that
\begin{align*}
 \norm{\Phi_\tau(z,\omega)-\Phi_\tau^{g_j}(z)}_{H^1}
 <\min\{1,r/2\}
\end{align*}
whenever $z\in\mathbb V_{R_{\rm b}}$, and 
\begin{align*}
 \norm{z-x_j}_{L^2}<\varepsilon,\qquad
 \norm{\omega-g_j}_{E_\tau}<\eta_j.
\end{align*}
Set
\begin{align*}
 R_*
 =1+\frac12(L_++1)^2+\frac14C_4^4(L_++1)^4,
\end{align*}
where the constant $C_4$ is from $\norm v_{L^4}\le C_4\norm v_{H^1}$.  Then
\begin{align*}
 \Phi_\tau(z,\omega)\in
 \mathbb V_{R_*}\cap B_{L^2}(u_*,r)
\end{align*}
on the preceding driver neighborhood.  Notice that
$R_*$ is independent of both $R$ and $r$.

Since Wiener measure has full support on $E_\tau$,
\begin{align*}
 p_0=\min_{1\le j\le N}
 \boldsymbol{\gamma}_\tau
 \{\omega:\norm{\omega-g_j}_{E_\tau}<\eta_j\}>0.
\end{align*}
Therefore
\begin{align*}
 \inf_{z\in\mathbb V_{R_{\rm b}}}
 P_\tau\bigl(z,\mathbb V_{R_*}
 \cap B_{L^2}(u_*,r)\bigr)\ge p_0.
\end{align*}
Combining this with \eqref{eq:access-fixed-source-core} and the Markov
property yields
\begin{align*}
 \inf_{x\in\mathbb V_R}
 P_{s_R+\tau}\bigl(x,\mathbb V_{R_*}
 \cap B_{L^2}(u_*,r)\bigr)\ge\frac{p_0}{2}.
\end{align*}
\end{proof}

\subsection{Geometric characterizations of saturation}
\label{subsec:saturation-geometry}

In this section we give two complementary geometric
descriptions of cubic saturation in \Cref{def:cubic-saturation}.  For general complex valued
smooth forcing profiles, the recursion is described by a
finitely generated algebra of real linear endomorphism fields, while for complex phase complete Fourier forcing, saturation is
equivalent to generation of the full difference lattice.

\subsubsection{General complex smooth forcing}

We first give a geometric characterization for general complex valued
smooth forcing profiles.  Assume throughout this part that
\begin{align*}
 b_1,\ldots,b_m\in C^\infty(\T^3;\C),\qquad
 b=(b_1,\ldots,b_m):\T^3\to\C^m.
\end{align*}
For $1\le j,k\le m$, define the real linear endomorphism field
$\mathsf T_{jk}(x)\in\operatorname{End}_\R(\C)$ by
\begin{align}\label{eq:complex-pair-operator}
 \mathsf T_{jk}(x)z
 =\mathsf JD^3\mathsf N[z,b_j(x),b_k(x)]
 =2\ii\bigl[(b_j\bar b_k+b_k\bar b_j)z+b_jb_k\bar z\bigr].
\end{align}
Let
\begin{align}\label{eq:complex-matrix-algebra}
 \mathbb A_b
 =\operatorname{Alg}_\R\{\operatorname{Id},\mathsf T_{jk}:1\le j,k\le m\}
 \subset C^\infty\bigl(\T^3;\operatorname{End}_\R(\C)\bigr)
\end{align}
be the unital real algebra generated by these fields under pointwise
composition.  Put
\begin{align*}
 \Gamma_b(x)=b(x)b(x)^*\in\operatorname{Herm}_m^+,
\end{align*}
where $b(x)^*$ is the conjugate transpose and
$\operatorname{Herm}_m^+$ is the cone of positive semidefinite
Hermitian $m\times m$ matrices.  Thus $\Gamma_b(x)$ records the
complex line spanned by $b(x)$ and its norm, while forgetting only
the common phase.

The cubic induction is encoded by matrix words in the $\mathsf T_{jk}$.
The following theorem identifies the resulting smooth geometry exactly.

\begin{theorem}
\label{thm:complex-smooth-geometric-saturation}
The following statements hold.
\begin{enumerate}[label=\textup{(\roman*)}]
 \item The cubic saturation closure admits the exact algebraic representation
 \begin{align}\label{eq:complex-exact-algebra}
  \mathscr W_\infty
  =\left\{\sum_{j=1}^m\mathsf A_jb_j:\mathsf A_j\in\mathbb A_b\right\},
 \end{align}
 where $\mathsf{A}_j b_j$ denotes pointwise action of the real-linear endomorphism $\mathsf{A}_j(x)$ on $b_j(x)$.
 \item For every integer $k\ge1$,
 \begin{align}\label{eq:complex-ck-characterization}
  \overline{\mathscr W_\infty}^{\,C^k}
  =C^k(\T^3;\C)
  \quad\Longleftrightarrow\quad
  \Gamma_b:\T^3\to\operatorname{Herm}_m^+
  \text{ is a smooth embedding}.
 \end{align}
 In particular, if $\Gamma_b$ is a smooth embedding, then the forcing
 is cubic saturating.
\end{enumerate}
\end{theorem}

\begin{proof}
For $h=\sum_j\alpha_jb_j$ and $g=\sum_\ell\beta_\ell b_\ell$, real
trilinearity gives
\begin{align*}
 \mathsf JD^3\mathsf N[\phi,h,g]
 =\sum_{j,\ell}\alpha_j\beta_\ell\mathsf T_{j\ell}\phi.
\end{align*}
Hence \eqref{eq:cubic-recursion} is precisely iteration of the
operators $\mathsf T_{jk}$ from
$\mathscr W_0=\spanop_\R\{b_1,\ldots,b_m\}$.  Inductively,
$\mathscr W_n$ is the real span of
\begin{align*}
 \mathsf T_{j_rk_r}\cdots\mathsf T_{j_1k_1}b_\ell,
 \qquad 0\le r\le n,
\end{align*}
and taking the union over $n$ proves
\eqref{eq:complex-exact-algebra}. Next we divide the proof of \eqref{eq:complex-ck-characterization} to two steps. 

\emph{Step 1: the implication $\Longleftarrow$ in
\eqref{eq:complex-ck-characterization}.}
Assume that $\Gamma_b$ is a
smooth embedding.  Then $b(x)\ne0$ for every $x$, since
$\dd\Gamma_b(x)=0$ whenever $b(x)=0$.  Thus
$U_j=\{x:b_j(x)\ne0\}$, $1\le j\le m$, form an open cover of $\T^3$.
A direct calculation from \eqref{eq:complex-pair-operator} gives,
for every $j,\ell$,
\begin{align}\label{eq:matrix-identity}
\begin{split}
\mathsf T_{jj}^2
 &=-12|b_j|^4\operatorname{Id},\\
 \mathsf T_{jj}\mathsf T_{j\ell}
 +\mathsf T_{j\ell}\mathsf T_{jj}
 &=-24|b_j|^2\Rea(b_j\bar b_\ell)\operatorname{Id},\\
 \mathsf T_{jj}[\mathsf T_{jj},\mathsf T_{j\ell}]b_j
 &=-48|b_j|^4\operatorname{Im}(b_j\bar b_\ell)b_j,\\
 -\mathsf T_{jj}[\mathsf T_{jj},\mathsf T_{j\ell}]
 (\mathsf T_{jj}b_j)
 &=-48|b_j|^4\operatorname{Im}(b_j\bar b_\ell)\mathsf T_{jj}b_j.
\end{split}
\end{align}
Moreover,
\begin{align*}
 \mathsf T_{jj}b_j=6\ii|b_j|^2b_j,
\end{align*}
so $b_j$ and $\mathsf T_{jj}b_j$ form a pointwise real basis of
$\C$ on $U_j$. Define
\begin{align*}
 \Psi_j(x)=\Bigl(
 |b_j|^4,\,
 (|b_j|^2\Rea(b_j\bar b_\ell))_{\ell=1}^m,\,
 (|b_j|^4\operatorname{Im}(b_j\bar b_\ell))_{\ell=1}^m
 \Bigr).
\end{align*}
As $|b_j|>0$ on $U_j$,  the maps $\Psi_j$ and $\Gamma_b$ smoothly
determine each other there.  Hence, as $\Gamma_b$ is a smooth
embedding, so is $\Psi_j$ on $U_j$.
The matrix identities \eqref{eq:matrix-identity}, after multiplication
by suitable real constants, show that each coordinate of $\Psi_j$
can be realized by a matrix word acting as scalar multiplication on
$b_j$ and on $\mathsf T_{jj}b_j$.  Since $\mathbb A_b$ in \eqref{eq:complex-matrix-algebra} is closed
under composition and real linear combinations, the same holds for
every real polynomial $P$ in these coordinates.  Hence by \eqref{eq:complex-exact-algebra}, 
\begin{align*}
 P(\Psi_j)b_j,\qquad
 P(\Psi_j)\mathsf T_{jj}b_j\in\mathscr W_\infty.
\end{align*}

Choose a smooth partition of unity $\chi_j$ subordinate to
$\{U_j\}$.  For $f\in C^\infty(\T^3;\C)$, the local real basis above
gives real functions $a_j,c_j\in C^\infty(\T^3)$, supported compactly
in $U_j$, such that
\begin{align*}
 \chi_jf=a_jb_j+c_j\mathsf T_{jj}b_j.
\end{align*}
Since $\Psi_j:U_j\to\Psi_j(U_j)$ is a smooth embedding and
$\operatorname{supp}a_j,\operatorname{supp}c_j\Subset U_j$, there are
smooth compactly supported functions $\widetilde a_j,\widetilde c_j$
on $\Psi_j(U_j)$ such that
\begin{align*}
 a_j=\widetilde a_j\circ\Psi_j,\qquad
 c_j=\widetilde c_j\circ\Psi_j.
\end{align*}
Since $a_j$ and $c_j$ vanish near $\partial U_j$ and $\Psi_j=0$ on
$\T^3\setminus U_j$, these functions may be extended smoothly to a
neighborhood of $\Psi_j(\T^3)$. After smooth extension to the ambient
Euclidean space and polynomial $C^k$ approximation on a compact box,
there are real polynomials $P_n,Q_n$ such that
\begin{align*}
 P_n(\Psi_j)\longrightarrow a_j,\qquad
 Q_n(\Psi_j)\longrightarrow c_j
 \qquad\text{in }C^k(\T^3).
\end{align*}
Therefore
\begin{align*}
 P_n(\Psi_j)b_j+Q_n(\Psi_j)\mathsf T_{jj}b_j
 \longrightarrow\chi_jf
 \qquad\text{in }C^k,
\end{align*}
and every term on the left belongs to $\mathscr W_\infty$.  Summing
in $j$ shows that every smooth complex valued function lies in the
$C^k$ closure of $\mathscr W_\infty$.  This finishes the proof of the implication $\Longleftarrow$. 

\emph{Step 2: The implication $\Longrightarrow$ in
\eqref{eq:complex-ck-characterization}.} Suppose that
$\overline{\mathscr W_\infty}^{\,C^k}=C^k(\T^3;\C)$ for some
$k\ge1$.  If $b(x_0)=0$ at some $x_0$, then every $\mathsf T_{jk}(x_0)$ vanishes,
so \eqref{eq:complex-exact-algebra} implies
$\phi(x_0)=0$ for every $\phi\in\mathscr W_\infty$, a contradiction. For nonzero $z,w\in\C^m$,
\begin{align*}
 zz^*=ww^*
 \quad\Longleftrightarrow\quad
 w=\e^{\ii\theta}z
\end{align*}
for some $\theta\in\R$.  For $h,g,z\in\C$, let
$\mathsf T_{h,g}z=2\ii[(h\bar g+g\bar h)z+hg\bar z]$.  Then
\begin{align*}
 \mathsf T_{\e^{\ii\theta}h,\e^{\ii\theta}g}
 (\e^{\ii\theta}z)
 =\e^{\ii\theta}\mathsf T_{h,g}z.
\end{align*}
Thus, if $\Gamma_b(x)=\Gamma_b(y)$, then
$b(y)=\e^{\ii\theta}b(x)$ for some $\theta$, and induction in
\eqref{eq:cubic-recursion} gives
\begin{align*}
 \phi(y)=\e^{\ii\theta}\phi(x),
 \qquad \phi\in\mathscr W_\infty.
\end{align*}
If $x\ne y$, this contradicts $C^k$ density: for
$\e^{\ii\theta}\ne1$ it excludes approximation of the constant
function, while for $\e^{\ii\theta}=1$ it prevents separation of the
two points.  Hence $\Gamma_b$ is injective.

It remains to prove immersion.  Suppose that
$0\ne\xi\in T_x\T^3$ and $\dd\Gamma_b(x)\xi=0$, and put
$w=\dd b(x)\xi$.  Then
\begin{align*}
 wb(x)^*+b(x)w^*=0.
\end{align*}
Applying this identity to $b(x)$ gives
$|b(x)|^2w+b(x)(w^*b(x))=0$, so $w=\lambda b(x)$ for some
$\lambda\in\C$.  Substituting back yields
$\lambda+\bar\lambda=0$, hence
\begin{align}\label{eq:db-alphab}
 \dd b(x)\xi=\ii\alpha b(x)
\end{align}
for some $\alpha\in\R$. For every $\phi\in\mathscr W_\infty$, the cubic induction gives a real
polynomial map $F_\phi:\C^m\simeq\R^{2m}\to\C$ such that
$\phi=F_\phi\circ b$ and
\begin{align}\label{eq:phase-identity}
 F_\phi(\e^{\ii\theta}z)=\e^{\ii\theta}F_\phi(z).
\end{align}
Indeed, for $\phi=b_j$ one takes $F_\phi(z)=z_j$.  If
$\phi=F_\phi\circ b$ and $h=H\circ b$, $g=G\circ b$ with real linear
$H,G$, then the new descendant is represented by
\begin{align*}
 2\ii\bigl[(H\bar G+G\bar H)F_\phi+HG\bar F_\phi\bigr],
\end{align*}
which again has phase weight one.  Differentiating the phase identity \eqref{eq:phase-identity} 
at $\theta=0$  and using \eqref{eq:db-alphab} give 
\begin{align*}
 \dd\phi(x)\xi
 =DF_\phi(b(x))[\dd b(x)\xi]
 =\alpha DF_\phi(b(x))[\ii b(x)]
 =\ii\alpha\phi(x).
\end{align*}
This first jet constraint is closed under $C^1$ convergence.  If
$\alpha\ne0$, it is violated by the constant function $1$; if
$\alpha=0$, it is violated by a smooth function with
$\dd f(x)\xi\ne0$.  This contradicts the assumed $C^k$ density.
Hence $\dd\Gamma_b(x)$ is injective for every $x$.  Thus $\Gamma_b$
is an injective immersion, and compactness of $\T^3$ makes it a
smooth embedding.
\end{proof}

The characterization yields a simple optimal complex family.

\begin{proposition}
Fix $0<r<R$ and set
\begin{align}\label{eq:three-complex-profiles}
 b_1(x)&=1,\notag\\
 b_2(x)&=(R+r\cos x_3)\e^{\ii x_1},\\
 b_3(x)&=(R+r\sin x_3)\e^{\ii x_2}.\notag
\end{align}
Then, for every integer $k\ge1$,
\begin{align*}
 \overline{\mathscr W_\infty}^{\,C^k}=C^k(\T^3;\C),
\end{align*}
and in particular these three profiles are cubic saturating.
Moreover, three profiles are minimal among complex profile families
with $C^1$ dense saturation.
\end{proposition}

\begin{proof}
Since $b_1=1$, the first column of $\Gamma_b$ contains $b_2$ and
$b_3$, and the map $x\mapsto(b_2(x),b_3(x))$ is  a smooth
embedding for $R>r$.  Hence the conclusion follows from
\Cref{thm:complex-smooth-geometric-saturation}.

For minimality, $C^1$ dense saturation forces $\Gamma_b$ to embed
$\T^3$ into the nonzero rank one positive semidefinite Hermitian cone, which has
real dimension $2m-1$.  Thus $m=1$ is impossible.  For $m=2$ this
cone is the connected noncompact three-manifold
$(0,\infty)\times\mathbb CP^1$.  By invariance of domain, an embedding
of the compact three-manifold $\T^3$ into this cone would have open
image, which is impossible.  Hence $m\ge3$, and
\eqref{eq:three-complex-profiles} attains the bound.
\end{proof}

\subsubsection{Phase-complete Fourier forcing}

We now specialize the general complex forcing to a natural
phase-complete Fourier class.  For each forced wave vector $k$, the
forcing contains the full complex line $\C e_k$.  In this class,
cubic saturation admits an exact affine lattice characterization.

Let $K=\{k_1,\ldots,k_N\}\subset\Z^3$, $N\ge1$, and write
$e_k(x)=\e^{\ii k\cdot x}$.  We call the forcing
\emph{complex phase-complete on $K$} if
\begin{align}\label{eq:complex-phase-complete-fourier}
 \Ran_\R B=\spanop_\R\{e_k,\ii e_k:k\in K\}
 =\spanop_\C\{e_k:k\in K\},
\end{align}
where the last space is regarded as a real vector space.   The phase-complete formulation is convenient for the exact affine
lattice characterization, but is not rank optimal in terms of real
forcing directions.  Indeed, the four complex Fourier profiles
\begin{align*}
 1,\quad \e^{\ii x_1},\quad
 \e^{\ii x_2},\quad \e^{\ii x_3}
\end{align*}
already form a cubic saturating family by
\Cref{thm:complex-smooth-geometric-saturation}, whereas their
phase-complete realization uses eight real Brownian directions.

Put $\Lambda_K=\spanop_\Z\{k-k':k,k'\in K\}$ and, for any $k^\circ\in K$, let $\mathcal C_K=k^\circ+\Lambda_K$. Since $k-k^\circ\in\Lambda_K$ for every $k\in K$, the coset $\mathcal C_K$ is independent of the choice of $k^\circ$.

\begin{theorem}
\label{thm:exact-fourier-saturation}
Assume that the forcing is complex phase-complete on $K$. Then
\begin{align}\label{eq:exact-saturation-union}
 \mathscr W_\infty=\spanop_\C\{e_\ell:\ell\in\mathcal C_K\}.
\end{align}
Consequently, for every finite $s$,
\begin{align}\label{eq:exact-saturation-closure}
 \overline{\mathscr W_\infty}^{\,H^s}=H^s_{\mathcal C_K}(\T^3;\C):=\{u\in H^s:\widehat u(\ell)=0\text{ for }\ell\notin\mathcal C_K\}.
\end{align}
In particular, the following are equivalent:
\begin{enumerate}[label=\textup{(\roman*)}]
 \item the forcing is cubic saturating;
 \item $\Lambda_K=\Z^3$;
 \item there are no nonzero $\vartheta\in\T^3$ and $c\in\mathbb S^1$ such that $\e^{\ii k\cdot\vartheta}=c$ for every $k\in K$.
\end{enumerate}
If $0\in K$, the common phase in \textup{(iii)} is necessarily $c=1$.
\end{theorem}

\begin{proof}
For $\phi\in\mathscr W_n$ and $h,g\in\mathscr W_0$, \eqref{eq:cubic-third-derivative} gives
\begin{align*}
 \ii\mathsf JD^3\mathsf N[\phi,h,g]
 &=\mathsf JD^3\mathsf N[\ii\phi,h,g]
 +\mathsf JD^3\mathsf N[\phi,\ii h,g]
 +\mathsf JD^3\mathsf N[\phi,h,\ii g],\\
 \phi h\bar g
 &=-\frac14\bigl\{\mathsf JD^3\mathsf N[\ii\phi,h,g]
 +\mathsf JD^3\mathsf N[\phi,\ii h,g]\bigr\}.
\end{align*}
Since $\mathscr W_0$ is complex linear by \eqref{eq:complex-phase-complete-fourier}, induction shows that every $\mathscr W_n$ is complex linear and
$\phi h\bar g\in\mathscr W_{n+1}$ whenever $\phi\in\mathscr W_n$ and $h,g\in\mathscr W_0$. For pure modes, $e_pe_q\overline{e_r}=e_{p+q-r}$, so the recursion permits the frequency operation $p\mapsto p+q-r$ with $q,r\in K$.  Starting from $k^\circ\in K$, repeated application produces
\begin{align*}
 k^\circ+\sum_{\nu=1}^M(q_\nu-r_\nu),\qquad q_\nu,r_\nu\in K.
\end{align*}
Every element of $\Lambda_K$ is such a finite sum, after interchanging $q_\nu$ and $r_\nu$ when necessary.  Hence
\begin{align*}
 \spanop_\C\{e_\ell:\ell\in\mathcal C_K\}\subset\mathscr W_\infty.
\end{align*}

For the reverse inclusion, if $p\in\mathcal C_K$ and $q,r\in K$, then $p+q-r$, $p+r-q$, and $q+r-p$ all belong to $\mathcal C_K$.  Hence
\begin{align*}
 \mathsf JD^3\mathsf N[e_p,e_q,e_r]
 =2\ii(e_{p+q-r}+e_{p+r-q}+e_{q+r-p})
\end{align*}
has Fourier support in $\mathcal C_K$.  Real trilinearity therefore shows that $\spanop_\C\{e_\ell:\ell\in\mathcal C_K\}$ is invariant under the cubic recursion and contains $\mathscr W_0$.  This proves \eqref{eq:exact-saturation-union}.  Fourier truncation gives \eqref{eq:exact-saturation-closure}, and hence \textup{(i)}$\Leftrightarrow$\textup{(ii)}.

If $\Lambda_K\ne\Z^3$, character duality gives a nonzero $\vartheta\in\T^3$ such that $\e^{\ii\lambda\cdot\vartheta}=1$ for every $\lambda\in\Lambda_K$.  Thus all $\e^{\ii k\cdot\vartheta}$, $k\in K$, have the same value.  Conversely, the condition in \textup{(iii)} implies that every $k-k'$ lies in the kernel of the nontrivial character $\ell\mapsto\e^{\ii\ell\cdot\vartheta}$, so $\Lambda_K\ne\Z^3$.  This proves \textup{(ii)}$\Leftrightarrow$\textup{(iii)}.
\end{proof}

The following arithmetic form is immediate.  Fix $k_1\in K$ and let
$\mathbf A_K$ be the $3\times(N-1)$ integer matrix with columns
$k_j-k_1$, $2\le j\le N$.

\begin{corollary}
\label{cor:fourier-saturation-criteria}
For complex phase-complete forcing on $K$, the following are
equivalent:
\begin{enumerate}[label=\textup{(\roman*)}]
 \item the forcing is cubic saturating;
 \item $\Lambda_K=\Z^3$;
 \item $\mathbf A_K:\Z^{N-1}\to\Z^3$ is onto.
\end{enumerate}
Equivalently, $\mathbf A_K$ has rank three and the greatest common
divisor of its $3\times3$ minors is one.
\end{corollary}

\begin{proof}
Since $\Lambda_K=\spanop_\Z\{k_j-k_1:2\le j\le N\}$,
$\operatorname{Range}\mathbf A_K=\Lambda_K$.  The conclusion follows
from \Cref{thm:exact-fourier-saturation} and the standard
determinantal divisor characterization of the Smith normal form;
see, for example, \cite[Chapter II]{Newman1972}.
\end{proof}

We now give the minimal rank in the complex phase-complete class.

\begin{corollary}\label{cor:fourier-minimal}
A cubic saturating complex phase-complete Fourier family requires at
least four wave vectors.  In particular,
$K=\{0,e_1,e_2,e_3\}$ is saturating.  If the constant mode is not
forced, one may take $K=\{e_1,e_2,e_3,e_1+e_2\}$.
\end{corollary}

\begin{proof}
Since $\Lambda_K$ is generated by at most $N-1$ differences, the
condition $\Lambda_K=\Z^3$ requires $N\ge4$.  The two displayed
families satisfy the criterion in
\Cref{cor:fourier-saturation-criteria}.
\end{proof}

The exact saturation formula also identifies the invariant Fourier sector when the difference lattice is proper.

\begin{proposition}

For complex phase-complete forcing, $H^1_{\mathcal C_K}(\T^3;\C)$ as in \eqref{eq:exact-saturation-closure} is invariant under the stochastic equation \eqref{eq:intro-spde}. 
\end{proposition}

\begin{proof}
Since $\mathcal C_K-\mathcal C_K+\mathcal C_K=\mathcal C_K$, the cubic
nonlinearity preserves $H^1_{\mathcal C_K}$.  The Laplacian preserves
each Fourier mode, and $K\subset\mathcal C_K$, so the forcing also
lies in this sector.  Hence $H^1_{\mathcal C_K}(\T^3;\C)$ is invariant
under \eqref{eq:intro-spde}.
\end{proof}

\begin{remark}
The main result remains valid on each invariant sector
$H^1_{\mathcal C_K}$, with saturation understood relative to that
sector. When $\operatorname{rank}\Lambda_K=r<3$, after removing the affine
carrier frequency the restricted dynamics are naturally identified
with an $r$-dimensional damped cubic NLS with a constant
positive-definite quadratic dispersion.  The same argument, using the corresponding lower-dimensional periodic
Strichartz estimates, yields the polynomial mixing result in dimensions
one and two.
\end{remark}

\subsection{Stationary regularity gain}

In this subsection we prove the stationary regularity gain used in the
$L^2$ to $H^1$ bootstrap in the proof of \Cref{thm:main}.  The argument
uses stationary four-wave identities for the three-dimensional SNLS
and does not require the saturation assumption.

Recall that $\mathsf S(t)=\e^{-(\gamma-\ii\Delta)t}$. The following main theorem of this section will be proved at the end. 
\begin{theorem}
\label{thm:stationary-asymptotic-smoothing}
Let $\nu$ be any invariant probability measure of
\eqref{eq:intro-spde}.  Then, for every $0<\varepsilon<1$,
\begin{align}\label{eq:stationary-h1eps}
 \int_{H^1}\norm u_{H^{1+\varepsilon}}^2\,\nu(\dd u)<\infty.
\end{align}
Consequently, $\nu(H^{2-})=1$. 
\end{theorem}

The gain is a property of stationary statistics and is not a
finite-time smoothing statement for the Markov flow.  The proof has
two main ingredients.  First, the subcritical periodic Strichartz family,
combined with stationarity, gives one time $L^\infty$ moments of every
order below four.  Second, a stationary four-wave identity and an
adaptive resolvent estimate control the nonlinear flux in truncated
$H^s$ balances.  Iterating the resulting regularity improvement yields
every Sobolev exponent below two.

The following lemma gives the stationary energy--Strichartz moments. 
\begin{lemma}
\label{lem:stationary-linfty-family}
For every $10/3<q<4$, every invariant probability measure $\nu$
satisfies
\begin{align}\label{eq:stationary-linfty-family}
 \int_{H^1}\norm u_{L^\infty}^q\,\nu(\dd u)<\infty.
\end{align}
Moreover, if $0\le r<2$ and $u\in H^r\cap L^\infty$, then
\begin{align}\label{eq:cubic-tame-hr}
 \norm{\mathsf N(u)}_{H^r}
 \le C_r\norm u_{L^\infty}^2\norm u_{H^r}.
\end{align}
\end{lemma}

\begin{proof}
Fix $10/3<q<4$.  Since $q<4$, the interval
\begin{align}\label{eq:stationary-strichartz-parameters}
 \frac3q<\sigma<\frac5q-\frac12
\end{align}
is nonempty.  Put $r_q=\sigma+3/2-5/q$.  Then $r_q<1$.
The scale-invariant periodic estimate
\cite[Theorem~1.1]{KillipVisan2016}, applied dyadically and summed by
the Littlewood--Paley square-function theorem, gives on every interval
$I=[a,b]$ of length at most one
\begin{align*}
 \norm{\mathsf S(t-a)f}_{L^q(I;W^{\sigma,q})}
 &\le C_q\norm f_{H^{r_q}},\\
 \norm{\int_a^t\mathsf S(t-r)F(r)\,\dd r}_{C(I;H^{r_q})\cap L^q(I;W^{\sigma,q})}
 &\le C_q\norm F_{L^1(I;H^{r_q})}.
\end{align*}
The damping changes only the constant.  The inhomogeneous estimate
follows from the homogeneous estimate and Minkowski's inequality.
By \eqref{eq:stationary-strichartz-parameters},
\begin{align*}
 H^1\hookrightarrow H^{r_q},\qquad
 W^{\sigma,q}\hookrightarrow L^\infty,\qquad
 H^2\hookrightarrow W^{\sigma,q}.
\end{align*}
The standard product inequality
\begin{align*}
 \norm{fg}_{H^r}
 \le C_r\bigl(
 \norm f_{L^\infty}\norm g_{H^r}
 +\norm g_{L^\infty}\norm f_{H^r}\bigr)
\end{align*}
then gives \eqref{eq:cubic-tame-hr} after two applications to $u\bar u u$. 

Let $I_0=[a,a+1]$ and set
\begin{align*}
 Z_a(t)=\int_a^t\mathsf S(t-r)B\,\dd W_r,\qquad
 H_a=\sup_{t\in I_0}(1+\cE(u_t)),\qquad
 S_a=\norm u_{L^{7/2}(I_0;W^{7/8,7/2})}.
\end{align*}
Since $r_q<1$, \eqref{eq:cubic-tame-hr} and
$W^{7/8,7/2}\hookrightarrow L^\infty$ imply
\begin{align*}
 \norm{\mathsf N(u)}_{L^1(I_0;H^{r_q})}
 \le C\norm u_{C(I_0;H^1)}
 \norm u_{L^2(I_0;L^\infty)}^2\le CH_a^{1/2}S_a^2.
\end{align*}
The Duhamel formula therefore yields
\begin{align*}
 \norm u_{L^q(I_0;W^{\sigma,q})}
 \le C_q\left(
 1+H_a^{1/2}+\norm{Z_a}_{C(I_0;H^2)}
 +H_a^{1/2}S_a^2\right).
\end{align*}
The conditional finite time energy estimate from
\Cref{prop:energy-strichartz} gives, for every $m>0$,
\begin{align*}
 \E\bigl[H_a^m\mid\mathcal F_a\bigr]
 \le C_m\bigl(1+\cE(u_a)\bigr)^{N_m}.
\end{align*}
Together with \eqref{eq:block-strichartz-moments}, the invariant
energy moments \eqref{eq:invariant-moments}, and the Gaussian moments
of $Z_a$, this shows that, under an invariant initial law,
\begin{align*}
 \E_\nu\norm u_{L^q(0,1;W^{\sigma,q})}^q<\infty.
\end{align*}

Taking a stationary solution $u_t$ issued from $\nu$, applying
Tonelli first to finite Littlewood--Paley truncations and then passing
to the full square function by monotone convergence gives
\begin{align*}
 \int_{H^1}\norm u_{W^{\sigma,q}}^q\,\nu(\dd u)
 =\E_\nu\int_0^1\norm{u_t}_{W^{\sigma,q}}^q\,\dd t<\infty.
\end{align*}
Since $W^{\sigma,q}\hookrightarrow L^\infty$, this proves
\eqref{eq:stationary-linfty-family}.
\end{proof}

We now deduce the deterministic four-wave estimate that will be used in the
stationary Sobolev balance. Write
\begin{align*}
 e_k(x)=\e^{\ii k\cdot x},\qquad
 \widehat f(k)=\int_{\T^3}f(x)\e^{-\ii k\cdot x}\,\dd x,\qquad
 \la k\ra=(1+|k|^2)^{1/2}.
\end{align*}
Fix $\rho\in C_c^\infty(\R^3;[0,1])$ with $\rho=1$ on
$\{|\xi|\le1\}$ and $\rho=0$ on $\{|\xi|\ge2\}$, and set
\begin{align}\label{eq:wsR}
 w_{s,R}(k)=\la k\ra^{2s}\rho(k/R),\qquad R\ge1.
\end{align}
On the momentum hyperplane
\begin{align}\label{eq:momentum-hyperplane}
 \Gamma=\{\boldsymbol k=(k_1,k_2,k_3,k_4)\in(\Z^3)^4:
 k_1-k_2+k_3-k_4=0\}
\end{align}
put
\begin{align*}
\begin{split}
 \Omega(\boldsymbol k)
 &=|k_1|^2-|k_2|^2+|k_3|^2-|k_4|^2,\\
 \delta_{s,R}(\boldsymbol k)
 &=w_{s,R}(k_2)+w_{s,R}(k_4)-w_{s,R}(k_1)-w_{s,R}(k_3).
\end{split}
\end{align*}

\begin{lemma}
\label{lem:four-wave-resolvent}
Let $0\le b<1$, $1<s<3/2+b$, and $R\ge1$.  For trigonometric
polynomials define
\begin{align*}
 \mathfrak B_{s,R}(f_1,f_2,f_3,f_4)
 =\sum_{\boldsymbol k\in\Gamma}
 \frac{\delta_{s,R}(\boldsymbol k)}
 {4\gamma+\ii\Omega(\boldsymbol k)}
 \widehat f_1(k_1)\overline{\widehat f_2(k_2)}
 \widehat f_3(k_3)\overline{\widehat f_4(k_4)}.
\end{align*}
Then $\mathfrak B_{s,R}$ extends uniquely to $(H^{1+b})^4$, and
\begin{align}\label{eq:four-wave-resolvent}
 \abs{\mathfrak B_{s,R}(f_1,f_2,f_3,f_4)}
 \le C_{s,b,\gamma}
 \sum_{1\le i<j\le4}
 \norm{f_i}_{H^{1+b}}\norm{f_j}_{H^{1+b}}
 \prod_{\ell\ne i,j}\norm{f_\ell}_{H^1},
\end{align}
where the constant is independent of $R$.
\end{lemma}
Note that the exact resonant rectangle
\begin{align*}
 (k_1,k_2,k_3,k_4)
 =\bigl(Ne_1,N(e_1+e_2),Ne_2,0\bigr)
\end{align*}
shows that the range $s<3/2+b$ is sharp up to the endpoint.

\begin{proof}
We divide the proof according to the possible configurations of the
largest interacting frequencies. The normalized dyadic symbols appearing below have uniformly bounded
rescaled derivatives.  By the standard Fourier series decomposition
of smooth dyadic multipliers, they may therefore be reduced to
factorized multipliers with absolutely summable coefficients and
uniform $L^r$ bounds.  We use this reduction implicitly below.

\emph{Step 1: dyadic setup and the resolvent representation.}
The cutoff weights satisfy, uniformly in $R\ge1$,
\begin{align}\label{eq:weight-symbol-bounds}
 \abs{\partial^\alpha w_{s,R}(\xi)}
 \le C_{\alpha,s}\la\xi\ra^{2s-|\alpha|}.
\end{align}
Indeed, when a derivative falls on $\rho(\xi/R)$ one has
$|\xi|\simeq R$, and the resulting factor $R^{-1}$ supplies the lost
power of $\la\xi\ra$. Let $P_N$ be a smooth dyadic decomposition, with $P_1$ containing the
origin, and put
\begin{align*}
 a_{j,N}^{(\beta)}
 =N^{1+\beta}\norm{P_Nf_j}_{L^2},\qquad \beta\in\{0,b\}.
\end{align*}
Then $\norm{a_j^{(\beta)}}_{\ell^2}
\simeq\norm{f_j}_{H^{1+\beta}}$. By multilinearity and the Littlewood--Paley decomposition,
\begin{align*}
 \mathfrak B_{s,R}(f_1,f_2,f_3,f_4)
 =\sum_{N_1,N_2,N_3,N_4}
 \mathfrak B_{s,R}(P_{N_1}f_1,P_{N_2}f_2,P_{N_3}f_3,P_{N_4}f_4).
\end{align*}
For $\boldsymbol N=(N_1,N_2,N_3,N_4)$, denote the corresponding
summand by $\mathfrak B_{s,R,\boldsymbol N}$.  Thus
\begin{align*}
 \mathfrak B_{s,R,\boldsymbol N}
 =\sum_{\boldsymbol k\in\Gamma}
 \frac{\delta_{s,R}(\boldsymbol k)}
 {4\gamma+\ii\Omega(\boldsymbol k)}
 \prod_{j=1}^4\varphi_{N_j}(k_j)\,
 \widehat f_1(k_1)\overline{\widehat f_2(k_2)}
 \widehat f_3(k_3)\overline{\widehat f_4(k_4)},
\end{align*}
where $\varphi_{N_j}$ is the Fourier multiplier of $P_{N_j}$.
Since $\Omega(\boldsymbol k)\in\Z$,
\begin{align}\label{eq:finite-laplace-resolvent}
 \frac1{4\gamma+\ii\Omega}
 =\frac1{1-\e^{-8\pi\gamma}}
 \int_0^{2\pi}\e^{-4\gamma t}\e^{-\ii t\Omega}\,\dd t.
\end{align}
Choose $10/3<p<4$ sufficiently close to $10/3$ that, with
\begin{align*}
 \kappa=\frac32-\frac5p>0,\qquad
 \frac1{p_*}=1-\frac3p,
\end{align*}
one has $p_*>10/3$ and
\begin{align}\label{eq:resolvent-exponent-choice}
 2s-3-2b+3\kappa<0.
\end{align}
The periodic scale-invariant estimate gives
\begin{align}\label{eq:periodic-p-resolvent}
 \norm{\e^{\ii t\Delta}P_Nf}_{L^p([0,2\pi]\times\T^3)}
 &\le C_pN^\kappa\norm{P_Nf}_{L^2},\notag\\
 \norm{\e^{\ii t\Delta}P_Nf}_{L^{p_*}([0,2\pi]\times\T^3)}
 &\le C_pN^{1-3\kappa}\norm{P_Nf}_{L^2}.
\end{align}

By the momentum constraint \eqref{eq:momentum-hyperplane}, the two
largest dyadic scales are comparable.  After fixing a sufficiently
large separation constant to distinguish whether the third largest
scale is comparable to the first two, we are reduced, up to finitely
many permutations of the frequency slots, to the following three
cases.

\emph{Step 2: at least three frequencies are comparable to the maximum.}
After a permutation of the slots, suppose
\begin{align*}
 N_1\simeq N_2\simeq N_3\simeq N,\qquad N_4=M\le N.
\end{align*}
The case in which all four scales are comparable is included by taking
$M\simeq N$.  By \eqref{eq:weight-symbol-bounds},
\begin{align*}
 |\delta_{s,R}(\boldsymbol k)|\le C_sN^{2s},
\end{align*}
and $N^{-2s}\delta_{s,R}$ has uniformly bounded rescaled derivatives
on the corresponding dyadic boxes.  Hence the multiplier
factorization, \eqref{eq:finite-laplace-resolvent},
H\"older's inequality with exponents $(p,p,p,p_*)$  and
\eqref{eq:periodic-p-resolvent} give
\begin{align*}
 |\mathfrak B_{s,R,\boldsymbol N}|
 &\le C_{s,\gamma}N^{2s}(N^\kappa)^3M^{1-3\kappa}
 \prod_{j=1}^4\norm{P_{N_j}f_j}_{L^2}\\
 &=C_{s,\gamma}N^{2s+3\kappa}M^{1-3\kappa}
 \prod_{j=1}^4\norm{P_{N_j}f_j}_{L^2}.
\end{align*}
Assign the two $H^{1+b}$ weights to the first two $N$-scale factors.
By the definition of $a_{j,N}^{(\beta)}$,
\begin{align*}
 \prod_{j=1}^4\norm{P_{N_j}f_j}_{L^2}
 =N^{-3-2b}M^{-1}
 a_{1,N}^{(b)}a_{2,N}^{(b)}a_{3,N}^{(0)}a_{4,M}^{(0)}.
\end{align*}
Consequently,
\begin{align}\label{eq:three-high-block}
 |\mathfrak B_{s,R,\boldsymbol N}|
 \le C_{s,\gamma}
 N^{2s-3-2b+3\kappa}M^{-3\kappa}
 a_{1,N}^{(b)}a_{2,N}^{(b)}a_{3,N}^{(0)}a_{4,M}^{(0)}.
\end{align}

\emph{Step 3: exactly two large frequencies in opposite-sign slots.}
After a permutation preserving the sign pattern in the momentum
relation, suppose
\begin{align*}
 N_1\simeq N_2\simeq N,\qquad N_3=M,\qquad N_4=K,\qquad N\gg M\ge K.
\end{align*}
Thus the two large frequencies occupy the opposite-sign slots
$k_1$ and $k_2$.  By the momentum constraint,
\begin{align*}
 k_1-k_2=k_4-k_3,
\end{align*}
and hence $|k_1-k_2|\lesssim M$.  Moreover,
\begin{align*}
 \delta_{s,R}(\boldsymbol k)
 =\{w_{s,R}(k_2)-w_{s,R}(k_1)\}
 +\{w_{s,R}(k_4)-w_{s,R}(k_3)\}.
\end{align*}
The mean value formula and \eqref{eq:weight-symbol-bounds} give
\begin{align*}
 |w_{s,R}(k_2)-w_{s,R}(k_1)|
 \le C_s|k_2-k_1|N^{2s-1}
 \le C_sMN^{2s-1}.
\end{align*}
For the low-frequency difference,
\begin{align*}
 |w_{s,R}(k_4)-w_{s,R}(k_3)|
 \le C_sM^{2s}
 \le C_sMN^{2s-1},
\end{align*}
where we used $M\le N$ and $s>1$.  Consequently,
\begin{align*}
 |\delta_{s,R}(\boldsymbol k)|
 \le C_sMN^{2s-1}.
\end{align*}

We next verify the rescaled derivative bounds needed for the dyadic
multiplier factorization.  Regard $k_2,k_3,k_4$ as the independent
variables on the momentum hyperplane and set
$k_1=k_2-k_3+k_4$.  On the corresponding dyadic boxes, the normalized symbol satisfies
\begin{align}\label{eq:opposite-sign-symbol}
 \left|
 \partial_{k_2}^\alpha\partial_{k_3}^\beta\partial_{k_4}^\chi
 \frac{\delta_{s,R}(k_2-k_3+k_4,k_2,k_3,k_4)}
 {MN^{2s-1}}
 \right|
 \le C_{\alpha,\beta,\chi,s}
 N^{-|\alpha|}M^{-|\beta|}K^{-|\chi|}.
\end{align}
Indeed, if no derivative falls on the low variables $k_3,k_4$, the
difference
\begin{align*}
 w_{s,R}(k_2)-w_{s,R}(k_2-k_3+k_4)
\end{align*}
retains the factor $|k_3-k_4|\lesssim M$ by the mean value formula,
and \eqref{eq:weight-symbol-bounds} gives
\begin{align*}
 \left|
 \partial_{k_2}^\alpha
 \{w_{s,R}(k_2)-w_{s,R}(k_2-k_3+k_4)\}
 \right|
 \le C_{\alpha,s}MN^{2s-1-|\alpha|}.
\end{align*}
If $|\beta|+|\chi|>0$, a derivative in $k_3$ or $k_4$ may remove
the factor $|k_3-k_4|\lesssim M$ coming from the mean value formula.
We therefore use the direct symbol bound
\begin{align*}
 \left|
 \partial_{k_2}^\alpha\partial_{k_3}^\beta\partial_{k_4}^\chi
 \{w_{s,R}(k_2)-w_{s,R}(k_2-k_3+k_4)\}
 \right|
 \le C_{\alpha,\beta,\chi,s}
 N^{2s-|\alpha|-|\beta|-|\chi|}.
\end{align*}
After division by $MN^{2s-1}$, this is bounded by
$C_{\alpha,\beta,\chi,s}N^{-|\alpha|}M^{-|\beta|}K^{-|\chi|}$,
because
\begin{align*}
 \frac{N^{1-|\beta|-|\chi|}}M
 \le M^{-|\beta|}K^{-|\chi|},
 \qquad N\ge M\ge K.
\end{align*}
For the remaining difference
$w_{s,R}(k_4)-w_{s,R}(k_3)$, mixed derivatives in $k_3,k_4$ vanish.
Pure $k_3$-derivatives are bounded by
$C_{\beta,s}M^{2s-|\beta|}$ and pure $k_4$-derivatives by
$C_{\chi,s}K^{2s-|\chi|}$.  Since
\begin{align*}
 M^{2s}\le MN^{2s-1},\qquad
 K^{2s}\le MN^{2s-1},
\end{align*}
these terms also satisfy \eqref{eq:opposite-sign-symbol}. Therefore, the standard multiplier
factorization applies.  Using
\eqref{eq:finite-laplace-resolvent}, H\"older's inequality with the
two $N$-scale factors and the $M$-scale factor in $L^p_{t,x}$ and the
$K$-scale factor in $L^{p_*}_{t,x}$, and
\eqref{eq:periodic-p-resolvent}, we obtain
\begin{align*}
 |\mathfrak B_{s,R,\boldsymbol N}|
 &\le C_{s,\gamma}
 (MN^{2s-1})(N^\kappa)^2M^\kappa K^{1-3\kappa}
 \prod_{j=1}^4\norm{P_{N_j}f_j}_{L^2}.
\end{align*}
Assigning the two $H^{1+b}$ weights to the two $N$-scale factors,
\begin{align*}
 \prod_{j=1}^4\norm{P_{N_j}f_j}_{L^2}
 =N^{-2-2b}M^{-1}K^{-1}
 a_{1,N}^{(b)}a_{2,N}^{(b)}
 a_{3,M}^{(0)}a_{4,K}^{(0)}.
\end{align*}
Therefore
\begin{align}\label{eq:opposite-sign-block}
 |\mathfrak B_{s,R,\boldsymbol N}|
 \le C_{s,\gamma}
 N^{2s-3-2b+3\kappa}
 \left(\frac MN\right)^\kappa K^{-3\kappa}
 a_{1,N}^{(b)}a_{2,N}^{(b)}
 a_{3,M}^{(0)}a_{4,K}^{(0)}.
\end{align}

\emph{Step 4: exactly two large frequencies in equal-sign slots.}
After a permutation preserving the sign pattern in the momentum
relation, suppose
\begin{align*}
 N_1\simeq N_3\simeq N,\qquad N_2=M,\qquad N_4=K,\qquad N\gg M\ge K.
\end{align*}
Thus the two large frequencies occupy the equal-sign slots
$k_1$ and $k_3$.  By the momentum constraint,
\begin{align*}
 k_3=-k_1+k_2+k_4.
\end{align*}
Substituting this identity into the resonance function gives
\begin{align*}
 \Omega(\boldsymbol k)
 =2|k_1|^2-2k_1\cdot(k_2+k_4)+2k_2\cdot k_4.
\end{align*}
Since $|k_1|\simeq N$ and $|k_2|+|k_4|\lesssim M\ll N$, choosing the
separation constant sufficiently large yields
\begin{align}\label{eq:equal-sign-nonresonance}
 |\Omega(\boldsymbol k)|\ge cN^2.
\end{align}
Thus this dyadic region is uniformly nonresonant.

We next verify the rescaled derivative bounds needed for the dyadic
multiplier factorization.  Regard $k_1,k_2,k_4$ as independent and set
$k_3=-k_1+k_2+k_4$.  Since
\begin{align*}
 \Omega=2|k_1|^2-2k_1\cdot(k_2+k_4)+2k_2\cdot k_4,
\end{align*}
\eqref{eq:equal-sign-nonresonance} and repeated differentiation give
\begin{align*}
 \left|\partial_{k_1}^\alpha\partial_{k_2}^\beta\partial_{k_4}^\chi
 (4\gamma+\ii\Omega)^{-1}\right|
 \le C_{\alpha,\beta,\chi,\gamma}N^{-2-|\alpha|-|\beta|-|\chi|}.
\end{align*}
Moreover, \eqref{eq:weight-symbol-bounds} gives, with
$m=|\alpha|+|\beta|+|\chi|$,
\begin{align*}
 \left|\partial_{k_1}^\alpha\partial_{k_2}^\beta\partial_{k_4}^\chi
 \delta_{s,R}\right|
 \le C_{\alpha,\beta,\chi,s}
 \left(N^{2s-m}
 +\one_{\{\alpha=\chi=0\}}M^{2s-|\beta|}
 +\one_{\{\alpha=\beta=0\}}K^{2s-|\chi|}\right).
\end{align*}
Leibniz's rule, together with $N\ge M\ge K$ and
$(M/N)^{2s},(K/N)^{2s}\le1$, therefore yields
\begin{align*}
 \left|
 \partial_{k_1}^\alpha\partial_{k_2}^\beta\partial_{k_4}^\chi
 \left\{
 N^{2-2s}
 \frac{\delta_{s,R}(k_1,k_2,-k_1+k_2+k_4,k_4)}
 {4\gamma+\ii\Omega(k_1,k_2,-k_1+k_2+k_4,k_4)}
 \right\}\right|
 \le C_{\alpha,\beta,\chi,s,\gamma}
 N^{-|\alpha|}M^{-|\beta|}K^{-|\chi|}.
\end{align*}

Hence the standard multiplier factorization applies.  Using $L^2$ for the two $N$-scale factors and
$L^\infty$ for the two low-frequency factors, Bernstein's inequality
gives
\begin{align*}
 |\mathfrak B_{s,R,\boldsymbol N}|
 &\le C_{s,\gamma}N^{2s-2}
 \norm{P_Nf_1}_{L^2}\norm{P_Nf_3}_{L^2}
 \norm{P_Mf_2}_{L^\infty}\norm{P_Kf_4}_{L^\infty}\\
 &\le C_{s,\gamma}N^{2s-2}M^{3/2}K^{3/2}
 \norm{P_Nf_1}_{L^2}\norm{P_Mf_2}_{L^2}
 \norm{P_Nf_3}_{L^2}\norm{P_Kf_4}_{L^2}.
\end{align*}
Assigning the two $H^{1+b}$ weights to the two $N$-scale factors,
\begin{align*}
 \norm{P_Nf_1}_{L^2}\norm{P_Mf_2}_{L^2}
 \norm{P_Nf_3}_{L^2}\norm{P_Kf_4}_{L^2}
 =N^{-2-2b}M^{-1}K^{-1}
 a_{1,N}^{(b)}a_{2,M}^{(0)}a_{3,N}^{(b)}a_{4,K}^{(0)}.
\end{align*}
Therefore
\begin{align}\label{eq:equal-sign-block}
 |\mathfrak B_{s,R,\boldsymbol N}|
 \le C_{s,\gamma}
 N^{2s-3-2b}
 \left(\frac MN\right)^{1/2}
 \left(\frac KN\right)^{1/2}
 a_{1,N}^{(b)}a_{3,N}^{(b)}
 a_{2,M}^{(0)}a_{4,K}^{(0)}.
\end{align}

\emph{Step 5: summation over the dyadic scales.}
Set
\begin{align*}
 \eta=3+2b-2s-3\kappa>0,\qquad
 \eta_0=3+2b-2s>0.
\end{align*}
The positivity of $\eta$ follows from
\eqref{eq:resolvent-exponent-choice}, while $\eta_0>0$ follows from
$s<3/2+b$.

For the three-high configuration, \eqref{eq:three-high-block} gives
the kernel $N^{-\eta}M^{-3\kappa}$.  Since the scales are dyadic,
\begin{align*}
 \sup_N\sum_{M\le N}M^{-3\kappa}|a_{4,M}^{(0)}|
 \le
 \left(\sum_M M^{-6\kappa}\right)^{1/2}
 \norm{a_4^{(0)}}_{\ell^2}
 \le C_\kappa\norm{a_4^{(0)}}_{\ell^2}.
\end{align*}
Moreover,
\begin{align*}
 \sum_N
 |a_{1,N}^{(b)}a_{2,N}^{(b)}a_{3,N}^{(0)}|
 \le
 \norm{a_1^{(b)}}_{\ell^2}
 \norm{a_2^{(b)}}_{\ell^2}
 \norm{a_3^{(0)}}_{\ell^2},
\end{align*}
where we used Cauchy--Schwarz on the first two factors and
$\ell^2\hookrightarrow\ell^\infty$ on the third.  Since
$N^{-\eta}\le1$, the full sum of the three-high blocks is therefore
bounded by
\begin{align*}
 C
 \norm{a_1^{(b)}}_{\ell^2}
 \norm{a_2^{(b)}}_{\ell^2}
 \norm{a_3^{(0)}}_{\ell^2}
 \norm{a_4^{(0)}}_{\ell^2}.
\end{align*}

For the opposite-sign configuration, \eqref{eq:opposite-sign-block}
gives the kernel
\begin{align*}
 N^{-\eta}\left(\frac MN\right)^\kappa K^{-3\kappa},
 \qquad K\le M\le N.
\end{align*}
For every fixed $N$,
\begin{align*}
 \sum_{M\le N}\sum_{K\le M}
 \left(\frac MN\right)^\kappa K^{-3\kappa}
 |a_{3,M}^{(0)}a_{4,K}^{(0)}|
 &\le
 \left(\sum_{M\le N}\left(\frac MN\right)^{2\kappa}\right)^{1/2}
 \norm{a_3^{(0)}}_{\ell^2}\times
 \left(\sum_KK^{-6\kappa}\right)^{1/2}
 \norm{a_4^{(0)}}_{\ell^2}\\
 &\le C_\kappa
 \norm{a_3^{(0)}}_{\ell^2}
 \norm{a_4^{(0)}}_{\ell^2}.
\end{align*}
The remaining sum satisfies
\begin{align*}
 \sum_NN^{-\eta}
 |a_{1,N}^{(b)}a_{2,N}^{(b)}|
 \le
 \norm{a_1^{(b)}}_{\ell^2}
 \norm{a_2^{(b)}}_{\ell^2}.
\end{align*}
Hence the opposite-sign blocks satisfy the same four-factor bound.

Finally, for the equal-sign configuration,
\eqref{eq:equal-sign-block} gives the kernel
\begin{align*}
 N^{-\eta_0}
 \left(\frac MN\right)^{1/2}
 \left(\frac KN\right)^{1/2}.
\end{align*}
For every $N$,
\begin{align*}
 \sum_{M\le N}\left(\frac MN\right)^{1/2}|a_{2,M}^{(0)}|
 \le C\norm{a_2^{(0)}}_{\ell^2}, \quad 
 \sum_{K\le N}\left(\frac KN\right)^{1/2}|a_{4,K}^{(0)}|
 \le C\norm{a_4^{(0)}}_{\ell^2},
\end{align*}
by Cauchy--Schwarz and the summability of the dyadic geometric
kernels.  Enlarging the original region $K\le M\le N$ to
$K,M\le N$ therefore gives
\begin{align*}
 \sum_{\boldsymbol N\text{ in the equal-sign case}}
 |\mathfrak B_{s,R,\boldsymbol N}|
 \le C
 \norm{a_1^{(b)}}_{\ell^2}
 \norm{a_3^{(b)}}_{\ell^2}
 \norm{a_2^{(0)}}_{\ell^2}
 \norm{a_4^{(0)}}_{\ell^2}.
\end{align*}

The same estimates hold for the finitely many permutations of the
frequency slots.  Since
$\norm{a_j^{(\beta)}}_{\ell^2}\simeq
\norm{f_j}_{H^{1+\beta}}$, summing over the six possible choices of
the two slots carrying the $H^{1+b}$ norms yields
\eqref{eq:four-wave-resolvent}.  Density of trigonometric polynomials
then gives the extension of $\mathfrak B_{s,R}$ to $(H^{1+b})^4$.
\end{proof}

The next lemma provides one stationary bootstrap step.
\begin{lemma}
\label{lem:one-stationary-bootstrap}
Let $\nu$ be an invariant probability measure.  Fix $10/3<q<4$ and
$0\le b<1$.  Suppose either $b=0$, or there exists $A\in(0,1)$ such
that
\begin{align*}
 b<A\left(1-\frac2q\right),
 \qquad
 \int_{H^1}\norm u_{H^{1+A}}^2\,\nu(\dd u)<\infty.
\end{align*}
Then
\begin{align}\label{eq:bootstrap-output}
 \int_{H^1}\norm u_{H^s}^2\,\nu(\dd u)<\infty
 \qquad\text{for every }1<s<\frac32+b.
\end{align}
\end{lemma}

\begin{proof}
Fix $1<s<3/2+b$.  The proof is divided into three steps.

\emph{Step 1: the truncated stationary $H^s$ balance.}
When $b=0$, put $\theta=0$.  When $b>0$, put $\theta=b/A$.  By
assumption, $\frac2q+\theta<1$. 
Sobolev interpolation gives
\begin{align*}
 \norm u_{H^{1+b}}^2
 \le
 \norm u_{H^{1+A}}^{2\theta}
 \norm u_{H^1}^{2(1-\theta)}
\end{align*}
when $b>0$, while for $b=0$ the same estimate is trivial with
$\theta=0$.  Hence H\"older's inequality,
\Cref{lem:stationary-linfty-family}, and the invariant energy moments
give
\begin{align}\label{eq:stationary-cubic-integrability}
 \int_{H^1}
 \norm u_{L^\infty}^2
 \norm u_{H^{1+b}}^2
 \norm u_{H^1}^2\,\nu(\dd u)<\infty.
\end{align}
In particular, $u\in H^{1+b}\cap L^\infty$ for $\nu$-almost every
$u$, and \eqref{eq:cubic-tame-hr} implies
\begin{align*}
 \int\norm{\mathsf N(u)}_{H^{1+b}}
 \norm u_{H^{1+b}}\norm u_{H^1}^2\,\nu(\dd u)<\infty.
\end{align*}
Define
\begin{align*}
 Q_{s,R}(u)=\sum_{k\in\Z^3}w_{s,R}(k)|u_k|^2,\qquad
 q_{s,R}=\sum_{j=1}^m\sum_{k\in\Z^3}
 w_{s,R}(k)|b_{j,k}|^2,
\end{align*}
where $w_{s,R}$ is given as in \eqref{eq:wsR} and $b_{j,k}=\widehat b_j(k)$ is the $k$-th Fourier coefficient of $b_j$.  
The sums are finite because $w_{s,R}$ has compact frequency support,
and
\begin{align}\label{eq:stationary-noise-trace}
 q_{s,R}\le\sum_{j=1}^m\norm{b_j}_{H^s}^2.
\end{align}
Choose a stationary solution, so that $u_0\sim\nu$ and $u_t\sim\nu$
for every $t\ge0$.  For each Fourier mode,
\begin{align*}
 \dd u_k
 =-(\gamma+\ii|k|^2)u_k\,\dd t
 -\ii\mathsf N_k(u)\,\dd t
 +\sum_{j=1}^m b_{j,k}\,\dd W_t^j,
\end{align*}
where $\mathsf N_k(u)=\widehat{\mathsf N(u)}(k)$.  Applying real
It\^o calculus to the finite sum $Q_{s,R}(u_t)$ gives
\begin{align*}
 \dd Q_{s,R}(u_t)
 =-2\gamma Q_{s,R}(u_t)\,\dd t
 +2\sum_k w_{s,R}(k)
 \Im\bigl(\bar u_k\mathsf N_k(u)\bigr)\,\dd t+q_{s,R}\,\dd t+\dd M_t^{s,R}.
\end{align*}
For fixed $R$,
\begin{align*}
 \left|
 \sum_k w_{s,R}(k)
 \bar u_k\mathsf N_k(u)
 \right|
 \le C_R\norm u_{L^2}\norm{\mathsf N(u)}_{L^2}
 \le C_R\norm u_{H^1}^4,
\end{align*}
and the quadratic variation density of $M^{s,R}$ is bounded by a
polynomial in $\norm u_{H^1}$.  Thus one may first stop at the exit of
an $H^1$ energy ball, apply It\^o's formula, and then remove the stop
using \eqref{eq:invariant-moments} and BDG.  Taking expectations and using stationarity therefore gives
\begin{align}\label{eq:stationary-weighted-prebalance}
 2\gamma\int Q_{s,R}\,\dd\nu
 =q_{s,R}
 +2\int\sum_k w_{s,R}(k)
 \Im\bigl(\bar u_k\mathsf N_k(u)\bigr)\,\nu(\dd u).
\end{align}

We now justify the symmetrization rigorously.  Fix a real even
$\chi\in C_c^\infty(\R^3;[0,1])$, equal to one on the unit ball and
zero outside the ball of radius two, and let $S_K$ be the Fourier
multiplier defined by $\widehat{S_Ku}(k)=\chi(k/K)\widehat u(k)$. 
Define the finite sum
\begin{align*}
 \mathscr F_{s,R}^{(K)}(u)
 =\sum_{\boldsymbol k\in\Gamma}
 \delta_{s,R}(\boldsymbol k)
 \prod_{a=1}^4\chi(k_a/K)
 \Im G_{\boldsymbol k}(u),
\end{align*}
where
\begin{align}\label{eq:Gku}
 G_{\boldsymbol k}(u)
 =u_{k_1}\bar u_{k_2}u_{k_3}\bar u_{k_4}.
\end{align}
For this finite sum the slot symmetries are exact.  By changes of
variables preserving $\Gamma$ and the symmetric cutoff,
\begin{align*}
 \sum_{\Gamma}w_{s,R}(k_2)\prod_{a=1}^4\chi(k_a/K)\Im G_{\boldsymbol k}
 &=\sum_{\Gamma}w_{s,R}(k_4)\prod_{a=1}^4\chi(k_a/K)\Im G_{\boldsymbol k},\\
 \sum_{\Gamma}w_{s,R}(k_1)\prod_{a=1}^4\chi(k_a/K)\Im G_{\boldsymbol k}
 &=-\sum_{\Gamma}w_{s,R}(k_4)\prod_{a=1}^4\chi(k_a/K)\Im G_{\boldsymbol k},\\
 \sum_{\Gamma}w_{s,R}(k_3)\prod_{a=1}^4\chi(k_a/K)\Im G_{\boldsymbol k}
 &=-\sum_{\Gamma}w_{s,R}(k_4)\prod_{a=1}^4\chi(k_a/K)\Im G_{\boldsymbol k}.
\end{align*}
Hence
\begin{align*}
 \mathscr F_{s,R}^{(K)}(u)
 =4\sum_k w_{s,R}(k)
 \Im\left(
 \overline{(S_Ku)_k}\,
 \widehat{\mathsf N(S_Ku)}(k)
 \right).
\end{align*}
Since $S_Ku\to u$ in $H^1$ and
$\mathsf N(S_Ku)\to\mathsf N(u)$ in $L^2$, for fixed $R$,
\begin{align*}
 \mathscr F_{s,R}^{(K)}(u)
 \longrightarrow
 4\sum_k w_{s,R}(k)
 \Im\bigl(\bar u_k\mathsf N_k(u)\bigr).
\end{align*}
The limit is bounded by $C_R\norm u_{H^1}^4$.  We therefore define
$\mathscr F_{s,R}$ by this limit, so that
\eqref{eq:stationary-weighted-prebalance} takes the
symmetrized form
\begin{align}\label{eq:stationary-weighted-balance}
 2\gamma\int Q_{s,R}\,\dd\nu
 =q_{s,R}+\frac12\int\mathscr F_{s,R}(u)\,\nu(\dd u).
\end{align}

Applying It\^o formula to 
$G_{\boldsymbol k}$ in \eqref{eq:Gku} gives, after localization and removal by \eqref{eq:invariant-moments},
\begin{align*}
 \begin{split}
 G_{\boldsymbol k}(u_t)-G_{\boldsymbol k}(u_0)
 ={}&\int_0^t
 \Bigl[
 -(4\gamma+\ii\Omega(\boldsymbol k))G_{\boldsymbol k}(u_r)
 +\mathscr N_{\boldsymbol k}(u_r)\\
 &\qquad
 +\frac12\sum_{j=1}^m
 D^2G_{\boldsymbol k}(u_r)[b_j,b_j]
 \Bigr]\,\dd r
 +M_t^{\boldsymbol k},
 \end{split}
\end{align*}
where
\begin{align*}
 \mathscr N_{\boldsymbol k}(u)
 ={}&-\ii\mathsf N_{k_1}(u)\bar u_{k_2}u_{k_3}\bar u_{k_4}
 +\ii u_{k_1}\overline{\mathsf N_{k_2}(u)}u_{k_3}\bar u_{k_4}\\
 &-\ii u_{k_1}\bar u_{k_2}\mathsf N_{k_3}(u)\bar u_{k_4}
 +\ii u_{k_1}\bar u_{k_2}u_{k_3}\overline{\mathsf N_{k_4}(u)}.
\end{align*}
Taking expectations,
dividing by $t>0$, and using stationarity gives the stationary
four-wave identity
\begin{align}\label{eq:stationary-quartet-hierarchy}
 (4\gamma+\ii\Omega(\boldsymbol k))
 \int G_{\boldsymbol k}(u)\,\nu(\dd u)
 =\int\mathscr N_{\boldsymbol k}(u)\,\nu(\dd u)
 +\frac12\sum_{j=1}^m
 \int D^2G_{\boldsymbol k}(u)[b_j,b_j]\,\nu(\dd u).
\end{align}

\emph{Step 2: finite Fourier summation and the uniform flux bound.}
For fixed $K$, multiply \eqref{eq:stationary-quartet-hierarchy} by
\begin{align*}
 \frac{\delta_{s,R}(\boldsymbol k)}
 {4\gamma+\ii\Omega(\boldsymbol k)}
 \prod_{a=1}^4\chi(k_a/K),
\end{align*}
take the imaginary part, and sum over $\boldsymbol k\in\Gamma$.
Since the sum is finite, summation and expectation may be
interchanged directly.  By the definition of
$\mathscr F_{s,R}^{(K)}$ and the triangle inequality,
\begin{align*}
 \abs{\int\mathscr F_{s,R}^{(K)}(u)\,\nu(\dd u)}
 \le \mathscr R+\mathscr G,
\end{align*}
where
\begin{align*}
 \mathscr R
 =\int\Bigl\{
 &\abs{\mathfrak B_{s,R}(S_K\mathsf N(u),S_Ku,S_Ku,S_Ku)}
 +\abs{\mathfrak B_{s,R}(S_Ku,S_K\mathsf N(u),S_Ku,S_Ku)}\\
 &+\abs{\mathfrak B_{s,R}(S_Ku,S_Ku,S_K\mathsf N(u),S_Ku)}
 +\abs{\mathfrak B_{s,R}(S_Ku,S_Ku,S_Ku,S_K\mathsf N(u))}
 \Bigr\}\,\nu(\dd u)
\end{align*}
and
\begin{align*}
 \mathscr G
 =\frac12\sum_{j=1}^m
 \int\abs{
 \sum_{\boldsymbol k\in\Gamma}
 \frac{\delta_{s,R}(\boldsymbol k)}
 {4\gamma+\ii\Omega(\boldsymbol k)}
 \prod_{a=1}^4\chi(k_a/K)
 D^2G_{\boldsymbol k}(u)[b_j,b_j]
 }\nu(\dd u).
\end{align*}
Here the factors involving the nonlinearity are
$S_K\mathsf N(u)$ rather than $\mathsf N(S_Ku)$, since they arise
directly from \eqref{eq:stationary-quartet-hierarchy}. Since $|\chi|\le1$, $S_K$ is a contraction on every Sobolev space
under consideration.  By \Cref{lem:four-wave-resolvent} and
\eqref{eq:cubic-tame-hr},
\begin{align*}
 \mathscr R
 \le C_{s,b,\gamma}
 \int
 \norm u_{L^\infty}^2
 \norm u_{H^{1+b}}^2
 \norm u_{H^1}^2\,\nu(\dd u).
\end{align*}
Indeed, if $\mathsf N(u)$ occupies a strong slot, we use
\begin{align*}
 \norm{\mathsf N(u)}_{H^{1+b}}
 \le C\norm u_{L^\infty}^2\norm u_{H^{1+b}},
\end{align*}
while if it occupies a weak slot, we use the corresponding $H^1$
estimate and assign the two strong norms to two of the remaining
$u$ factors.  The right-hand side is finite by
\eqref{eq:stationary-cubic-integrability}. Moreover,
\begin{align*}
 S_Ku\to u\quad\text{in }H^{1+b},\qquad
 S_K\mathsf N(u)\to\mathsf N(u)\quad\text{in }H^{1+b}
\end{align*}
for $\nu$-almost every $u$.  Hence dominated convergence permits
$K\to\infty$ in the nonlinear terms.

For $\mathscr G$, the real Hessian
$D^2G_{\boldsymbol k}[b_j,b_j]$ is the sum of the six terms obtained
by placing the two noise directions in two distinct slots of the
quartic monomial.  Each term therefore becomes a form
$\mathfrak B_{s,R}$ with two inputs $S_Kb_j$ and two inputs $S_Ku$.
Another application of \Cref{lem:four-wave-resolvent} gives
\begin{align*}
 \mathscr G
 \le C_{s,b,\gamma}
 \sum_{j=1}^m\norm{b_j}_{H^{1+b}}^2
 \int\norm u_{H^{1+b}}^2\,\nu(\dd u),
\end{align*}
uniformly in $K$ and $R$.  The right-hand side is finite by Step~1.
Since $S_Kb_j\to b_j$ in every Sobolev space, dominated convergence
also permits $K\to\infty$ in these terms. Combining the two estimates and using
$\mathscr F_{s,R}^{(K)}(u)\to\mathscr F_{s,R}(u)$ gives
\begin{align}\label{eq:stationary-flux-uniform}
 \sup_{R\ge1}
 \abs{\int\mathscr F_{s,R}(u)\,\nu(\dd u)}<\infty.
\end{align}

\emph{Step 3: removal of the Sobolev cutoff.}
Insert \eqref{eq:stationary-flux-uniform} and
\eqref{eq:stationary-noise-trace} into
\eqref{eq:stationary-weighted-balance}.  Since $\gamma>0$,
\begin{align*}
 \sup_{R\ge1}
 \abs{\int Q_{s,R}(u)\,\nu(\dd u)}<\infty.
\end{align*}
For every $u$,
\begin{align*}
 Q_{s,R}(u)
 =\sum_k\la k\ra^{2s}\rho(k/R)|u_k|^2
 \longrightarrow\norm u_{H^s}^2
\end{align*}
as $R\to\infty$.  Since $Q_{s,R}\ge0$, Fatou's lemma gives
\begin{align*}
 \int\norm u_{H^s}^2\,\nu(\dd u)<\infty.
\end{align*}
This proves \eqref{eq:bootstrap-output}.
\end{proof}

\begin{proof}[Proof of \Cref{thm:stationary-asymptotic-smoothing}]
Fix $0<\varepsilon<1$.  Choose $q\in(10/3,4)$ so close to $4$ that
$\varepsilon<q/4$, and put $c=1-2/q$.  At $\eta=0$ the affine
recursion below has fixed point
\begin{align*}
 \frac{1/2}{1-c}=\frac q4>\varepsilon.
\end{align*}
Hence one may choose $\eta>0$ sufficiently small that
$c-\eta>0$ and
\begin{align*}
 \frac{1/2-\eta}{1-c+\eta}>\varepsilon.
\end{align*}
Starting from $g_0=0$, define
\begin{align*}
 g_{n+1}=\frac12-\eta+(c-\eta)g_n.
\end{align*}
Since $0<c-\eta<1$, the sequence is increasing and converges to
\begin{align*}
 \frac{1/2-\eta}{1-c+\eta}<\frac q4<1.
\end{align*}
By the choice of $\eta$, there is a finite $n$ for which
$g_n>\varepsilon$.

We prove inductively that
\begin{align}\label{eq:bootstrap-induction}
 \int\norm u_{H^{1+g_n}}^2\,\nu(\dd u)<\infty.
\end{align}
For $n=0$ this is the invariant energy moment.  Apply
\Cref{lem:one-stationary-bootstrap} with $b=0$.  It gives every
Sobolev order below $3/2$, and in particular
\eqref{eq:bootstrap-induction} for $g_1=\frac12-\eta$. 
Suppose now that \eqref{eq:bootstrap-induction} holds for some
$n\ge1$.  Put $b_n=(c-\eta)g_n$. 
Since $c=1-2/q$,
\begin{align*}
 0<b_n<g_n\left(1-\frac2q\right),
\end{align*}
and $g_n<1$.  Thus
\Cref{lem:one-stationary-bootstrap} applies with $A=g_n$ and
$b=b_n$.  It gives every Sobolev exponent strictly below
$3/2+b_n$.  Since
\begin{align*}
 1+g_{n+1}
 =\frac32-\eta+b_n
 <\frac32+b_n,
\end{align*}
we obtain \eqref{eq:bootstrap-induction} at level $n+1$.

Choose an index with $g_n>\varepsilon$.  The embedding
$H^{1+g_n}\hookrightarrow H^{1+\varepsilon}$ then yields
\eqref{eq:stationary-h1eps}.  Applying the result with
$\varepsilon=1-1/n$, $n\ge2$, and intersecting the resulting
full-measure sets gives
\begin{align*}
 \nu\left(\bigcap_{n\ge2}H^{2-1/n}\right)=1,
\end{align*}
which is the asserted concentration on $H^{2-}$.
\end{proof}

\subsection{Proof of the main theorem}

The passage from the coupling topology to the energy topology below is a quantitative form of the energy equation method; see \cite{Ball1997,EkrenKukavicaZiane2017}. We combine convergence of norm expectations with stationary Sobolev tails to retain a polynomial rate in the untruncated $H^1$ Wasserstein distance.

\begin{proof}[Proof of \Cref{thm:main}]
Recall that $\cV=1+\cE$.  By \eqref{eq:energy-coercivity}, the sublevels of
$\cV$ are bounded in $H^1$ and compact in $L^2$.  Hence
\Cref{prop:lyapunov} verifies \Cref{ass:poly-lyapunov}.  The exact
difference formula in \Cref{prop:exact-stable-compact-difference},
together with \Cref{prop:snls-regular-dependence}, verifies
\Cref{ass:stable-compact-regularity}, with $H=L^2$, $\mathcal H_T=H_T$,
$\rho(T)=\e^{-\gamma T}$, and the path functional
\eqref{eq:snls-growth-functional}.  The dense endpoint theorem
\Cref{thm:dense-endpoint} verifies \Cref{ass:dense-malliavin}.  Finally,
\Cref{prop:reduced-common-accessibility} verifies
\Cref{ass:common-accessibility}.
For $0<\delta\le1$, put
\begin{align*}
 d_\delta(u,v)=1\wedge\norm{u-v}_{L^2}^\delta.
\end{align*}
The preceding results verify
\Cref{ass:poly-lyapunov,ass:stable-compact-regularity,ass:dense-malliavin,ass:common-accessibility}.  Hence
\Cref{thm:stable-compact-polynomial-mixing}, applied with $X=H^1$,
$H=L^2$, and $\cV=1+\cE$, gives, for every $q>0$ and
$0<\delta\le1$,
\begin{align}\label{eq:snls-pairwise-polynomial-mixing}
 \mathcal W_{d_\delta}
 \bigl(P_t(x,\cdot),P_t(y,\cdot)\bigr)
 \le C\{1+\cE(x)^M+\cE(y)^M\}(1+t)^{-q},
\end{align}
where $M=M(q)$. By \Cref{prop:existence}, there is an invariant measure
$\mu$, and every invariant measure has all polynomial energy moments.
By \Cref{cor:full-support}, $\supp\mu=H^1$.
Moreover, \Cref{thm:stationary-asymptotic-smoothing} gives
$\mu(H^{2-})=1$.
Invariance, the coupling inequality, and \eqref{eq:snls-pairwise-polynomial-mixing} yield
\begin{align}\label{eq:snls-invariant-l2-polynomial-mixing}
\begin{split}
 \mathcal W_{d_\delta}\bigl(P_t(x,\cdot),\mu\bigr)
 &\le\int_{H^1}\mathcal W_{d_\delta}
 \bigl(P_t(x,\cdot),P_t(y,\cdot)\bigr)\,\mu(\dd y)\le C\left(1+\cE(x)^M
 \right)(1+t)^{-q},
\end{split}
\end{align}
and the uniqueness follows. 

We next upgrade the convergence from $L^2$ to $H^1$.  We first deduce
a consequence of \eqref{eq:snls-invariant-l2-polynomial-mixing}.
Suppose that $F:H^1\to\R$ satisfies
\begin{align*}
 |F(u)-F(v)|
 \le C_F\{1+\cE(u)^a+\cE(v)^a\}d_\theta(u,v)
\end{align*}
for some $a<\infty$ and $0<\theta\le1$.  Then, for every $L>0$,
\begin{align}\label{eq:weighted-observable-transfer}
 |P_tF(x)-\mu(F)|
 \le C_{F,L}\{1+\cE(x)^{M_L}\}(1+t)^{-L}.
\end{align}
Indeed, for conjugate exponents $r,r'>1$ and a nearly optimal
$d_\theta$-coupling $\pi_t$ of $P_t(x,\cdot)$ and $\mu$,
H\"older's inequality and $d_\theta^{r'}\le d_\theta$ give
\begin{align*}
 |P_tF(x)-\mu(F)|
 \le C_F
 \left(\int\{1+\cE(u)^a+\cE(v)^a\}^r\,\dd\pi_t\right)^{1/r}
 \left(\int d_\theta(u,v)\,\dd\pi_t\right)^{1/r'}.
\end{align*}
The first factor is bounded polynomially in $1+\cE(x)$ by
\Cref{prop:lyapunov} and \eqref{eq:invariant-moments}, while the
second is controlled by
\eqref{eq:snls-invariant-l2-polynomial-mixing} with order $Lr'$.

Set
\begin{align*}
 \mathsf G(u)
 =\frac12\sum_{j=1}^m\left[
 \norm{\nabla b_j}_{L^2}^2+
 \int_{\T^3}\left(|u|^2|b_j|^2+
 2(\Rea(u\bar b_j))^2\right)\dd x\right]
 -\frac{\gamma}{2}\norm u_{L^4}^4.
\end{align*}
The smoothness of the noise profiles, energy coercivity, and
\begin{align*}
 \norm{u-v}_{L^4}
 \le C\norm{u-v}_{L^2}^{1/4}
 \bigl(\norm u_{H^1}+\norm v_{H^1}\bigr)^{3/4}
\end{align*}
give
\begin{align*}
 |\mathsf G(u)-\mathsf G(v)|
 &\le C\{1+\cE(u)^2+\cE(v)^2\}d_{1/4}(u,v),\\
 \left|\norm u_{L^4}^4-\norm v_{L^4}^4\right|
 &\le C\{1+\cE(u)^2+\cE(v)^2\}d_{1/4}(u,v).
\end{align*}
Hence \eqref{eq:weighted-observable-transfer} applies to both
observables.

Since
\begin{align*}
 \mathsf D(u)=2\cE(u)+\frac12\norm u_{L^4}^4,
\end{align*}
the energy identity \eqref{eq:ito-identity} gives
\begin{align*}
 \frac{\dd}{\dd t}P_t\cE(x)+2\gamma P_t\cE(x)=P_t\mathsf G(x),
 \qquad
 \mu(\mathsf G)=2\gamma\mu(\cE).
\end{align*}
Variation of constants and \eqref{eq:weighted-observable-transfer}
therefore imply, for every $L>0$,
\begin{align*}
 |P_t\cE(x)-\mu(\cE)|
 \le C_L\{1+\cE(x)^{M_L}\}(1+t)^{-L}.
\end{align*}
Together with the exact mass balance
\begin{align*}
 P_t\bigl(\norm{\cdot}_{L^2}^2\bigr)(x)
 -\mu\bigl(\norm{\cdot}_{L^2}^2\bigr)
 =\e^{-2\gamma t}\left\{
 \norm x_{L^2}^2-\mu\bigl(\norm{\cdot}_{L^2}^2\bigr)\right\}
\end{align*}
and
\begin{align*}
 \norm u_{H^1}^2
 =\norm u_{L^2}^2+2\cE(u)-\frac12\norm u_{L^4}^4,
\end{align*}
this yields
\begin{align}\label{eq:h1-second-moment-convergence}
 \left|
 P_t\bigl(\norm{\cdot}_{H^1}^2\bigr)(x)
 -\mu\bigl(\norm{\cdot}_{H^1}^2\bigr)
 \right|
 \le C_L\{1+\cE(x)^{M_L}\}(1+t)^{-L}.
\end{align}

Let $\Pi_N$ be the Fourier projection onto
$\{k\in\Z^3:|k|\le N\}$ and $Q_N=\Id-\Pi_N$.  Since
\begin{align*}
 \left|
 \norm{\Pi_Nu}_{H^1}^2-\norm{\Pi_Nv}_{H^1}^2
 \right|
 \le CN^2\{1+\cE(u)+\cE(v)\}d_1(u,v),
\end{align*}
\eqref{eq:weighted-observable-transfer} and
\eqref{eq:h1-second-moment-convergence} imply
\begin{align*}
 &P_t\bigl(\norm{Q_N\cdot}_{H^1}^2\bigr)(x)
 +\int\norm{Q_Nu}_{H^1}^2\,\mu(\dd u)\\
 &\qquad\le
 2\int\norm{Q_Nu}_{H^1}^2\,\mu(\dd u)
 +C_LN^2\{1+\cE(x)^{M_L}\}(1+t)^{-L}.
\end{align*}
Taking $\varepsilon=1/2$ in
\eqref{eq:stationary-h1eps} gives
\begin{align}\label{eq:high-frequency-tail-transfer}
 P_t\bigl(\norm{Q_N\cdot}_{H^1}^2\bigr)(x)
 +\int\norm{Q_Nu}_{H^1}^2\,\mu(\dd u)
 \le CN^{-1}
 +C_LN^2\{1+\cE(x)^{M_L}\}(1+t)^{-L}.
\end{align}

Let now $\pi_t$ be an optimal $d_1$-coupling of
$P_t(x,\cdot)$ and $\mu$, and put $z=\norm{u-v}_{L^2}$.
Splitting into $\{z\le1\}$ and $\{z>1\}$ and using
Cauchy--Schwarz gives
\begin{align*}
 \int z^2\,\dd\pi_t
 \le \int d_1(u,v)\,\dd\pi_t
 +\left(\int(\norm u_{L^2}+\norm v_{L^2})^4\,\dd\pi_t\right)^{1/2}
 \left(\int d_1(u,v)\,\dd\pi_t\right)^{1/2}.
\end{align*}
The fourth moments are controlled by
\Cref{prop:lyapunov} and \eqref{eq:invariant-moments}.
Since the $d_1$-mixing estimate is available at arbitrary polynomial
order, for every $L>0$,
\begin{align}\label{eq:l2-second-moment-coupling}
 \int\norm{u-v}_{L^2}^2\,\pi_t(\dd u,\dd v)
 \le C_L\{1+\cE(x)^{M_L}\}(1+t)^{-L}.
\end{align}

Finally,
\begin{align*}
 \norm{u-v}_{H^1}^2
 \le CN^2\norm{u-v}_{L^2}^2
 +2\norm{Q_Nu}_{H^1}^2+2\norm{Q_Nv}_{H^1}^2.
\end{align*}
Integrating against $\pi_t$ and using
\eqref{eq:high-frequency-tail-transfer} and
\eqref{eq:l2-second-moment-coupling} gives
\begin{align*}
 \mathcal W_{2,H^1}\bigl(P_t(x,\cdot),\mu\bigr)^2
 \le CN^{-1}
 +C_LN^2\{1+\cE(x)^{M_L}\}(1+t)^{-L}.
\end{align*}
Given $q>0$, choose $a>2q$, set
$N=\lceil(1+t)^a\rceil$, and then choose $L>2a+2q$.  Hence
\begin{align}\label{eq:h1-w2-mixing}
 \mathcal W_{2,H^1}\bigl(P_t(x,\cdot),\mu\bigr)
 \le C_q\{1+\cE(x)^{M_q}\}(1+t)^{-q}.
\end{align}

It remains only to pass to arbitrary finite Wasserstein order.  If
$p>2$ and $\widehat\pi_t$ is an optimal
$\mathcal W_{2,H^1}$-coupling, interpolation between
$L^2(\widehat\pi_t)$ and $L^{2p}(\widehat\pi_t)$ gives
\begin{align*}
 \mathcal W_{p,H^1}\bigl(P_t(x,\cdot),\mu\bigr)
 &\le
 \mathcal W_{2,H^1}\bigl(P_t(x,\cdot),\mu\bigr)^{1/(p-1)}\\
 &\quad\times
 \left\{
 \left(P_t\norm{\cdot}_{H^1}^{2p}(x)\right)^{1/(2p)}
 +\left(\int\norm u_{H^1}^{2p}\,\mu(\dd u)\right)^{1/(2p)}
 \right\}^{(p-2)/(p-1)}.
\end{align*}
By \eqref{eq:energy-coercivity}, \Cref{prop:lyapunov}, and
\eqref{eq:invariant-moments}, the term in braces is bounded by
$C_p\{1+\cE(x)\}^{1/2}$ uniformly in $t$.  Applying
\eqref{eq:h1-w2-mixing} with polynomial order $(p-1)q$ therefore
proves
\begin{align*}
 \mathcal W_{p,H^1}\bigl(P_t(x,\cdot),\mu\bigr)
 \le C_{p,q}\{1+\cE(x)^{M_{p,q}}\}(1+t)^{-q}.
\end{align*}
For $1\le p\le2$, the same conclusion follows from
$\mathcal W_{p,H^1}\le\mathcal W_{2,H^1}$.  This proves
\eqref{eq:main-mixing}.
\end{proof}

\appendix
\section{Analytic estimates for the cubic stochastic NLS}\label{app:analytic}

Global well-posedness in the energy space is standard.  By \cite[Theorem~1.5(ii)]{CheungMosincat2019}, equation \eqref{eq:intro-spde} is globally well posed in $H^1$ with trajectories in $C([0,T];H^1)\cap L^{7/2}(0,T;W^{7/8,7/2})$ almost surely; the additional term $-\gamma u$ is a bounded dissipative perturbation and the smooth finite-rank operator $B$ satisfies the noise regularity hypotheses.  We therefore take the global $H^1$ Markov flow for granted and provide only the estimates needed in the mixing argument.

For an interval $I$, put
\begin{align}\label{eq:energy-strichartz-space}
 \norm v_{\mathcal S(I)}
 =\norm v_{C(I;H^1)}
 +\norm v_{L^{7/2}(I;W^{7/8,7/2})}.
\end{align}

For the pathwise formulation, let $\omega\in E_T$ and set $z=u-B\omega$.  Then $z$ solves
\begin{align}\label{eq:pathwise-shift}
 \partial_tz=-\gamma(z+B\omega)+\ii\Delta(z+B\omega)
 -\ii\abs{z+B\omega}^2(z+B\omega),\qquad z(0)=u_0.
\end{align}
Since $B\omega\in C([0,T];H^k)$ for every $k$, the standard
deterministic energy-subcritical theory gives a unique maximal $H^1$
solution of \eqref{eq:pathwise-shift}.  Put
\begin{align}\label{eq:regular-driver-domain}
 \mathfrak D_T
 =\{(x,\omega)\in H^1\times E_T:
 \text{\eqref{eq:pathwise-shift} has an $H^1$ solution on $[0,T]$}\}.
\end{align}
The local theory and the blow-up alternative show that
$\mathfrak D_T$ is open, hence Borel.  The almost sure global theory mentioned above gives
\begin{align}\label{eq:full-measure-regular-drivers}
 \boldsymbol\gamma_T\{\omega:(x,\omega)\in\mathfrak D_T\}=1,
 \qquad x\in H^1.
\end{align}
On $\mathfrak D_T$, set
$\boldsymbol\Phi_T(x,\omega)=z+B\omega$, which is the genuine solution
path and belongs to the energy--Strichartz space.  Extend this map
Borel measurably to $H^1\times E_T$, for instance by setting
\begin{align*}
 \boldsymbol\Phi_T(x,\omega)(s)
 =\e^{-(\gamma-\ii\Delta)s}x
\end{align*}
on $\mathfrak D_T^c$.  We thereby obtain a Borel representation
\begin{align*}
 \boldsymbol\Phi_T:H^1\times E_T\longrightarrow C([0,T];H^1)
\end{align*}
which, for each fixed initial state, agrees with the stochastic flow
for almost every Wiener driver.  Restriction and concatenation hold
almost surely by pathwise uniqueness.  Values of the extension on
$\mathfrak D_T^c$ will never be used as solutions.  For piecewise
constant Cameron--Martin controls, the standard persistence of
regularity in $H^s$, $s>3/2$, will also be used; see
\cite[Section~4]{Sarychev2012}.
The tangent and Malliavin operator fields used below are defined by
their variational equations on $\mathfrak D_T$ and are given arbitrary
Borel extensions on its complement.

\subsection{Lyapunov structure and existence of invariant measures}
We first provide necessary periodic estimates. Recall that $\mathsf N (u)=|u|^2u$ is the nonlinear term. 

\begin{lemma}\label{lem:periodic-toolkit}
If $I=[a,b]$ has length at most one, then 
\begin{align} 
 \norm{\e^{-(\gamma-\ii\Delta)(t-a)}f}_
 {L^{7/2}(I;W^{7/8,7/2})}
 &\le C\norm f_{H^{53/56}},\label{eq:periodic-homogeneous}\\
 \norm{\int_a^t\e^{-(\gamma-\ii\Delta)(t-s)}F(s)\,\dd s}_
 {C(I;H^{53/56})\cap L^{7/2}(I;W^{7/8,7/2})}
 &\le C\norm F_{L^1(I;H^{53/56})}. 
 \label{eq:periodic-inhomogeneous}
\end{align}
For $0\le s\le1$ and $u,v\in H^s\cap L^\infty$,
\begin{align}
 \norm{\mathsf N(u)}_{H^s}
 &\le C\norm u_{L^\infty}^2\norm u_{H^s},\notag\\
 \norm{\mathsf N(v)-\mathsf N(u)}_{H^s}
 &\le C(\norm u_{L^\infty}+\norm v_{L^\infty})^2
       \norm{v-u}_{H^s}\notag\\
 &\quad+C(\norm u_{L^\infty}+\norm v_{L^\infty})
       (\norm u_{H^s}+\norm v_{H^s})
       \norm{v-u}_{L^\infty} .
 \label{eq:cubic-difference-hs}
\end{align}
\end{lemma}

\begin{proof}
These estimates are standard; we include the short argument for completeness.
Let $P_N$ be a smooth dyadic projector, with $N\ge1$.  The scale-invariant estimate of \cite[Theorem~1.1]{KillipVisan2016} in dimension three gives, for $p=7/2$,
\begin{align*}
 \norm{\e^{\ii(t-a)\Delta}P_Nf}_{L^{7/2}(I\times\T^3)}
 \le C N^{\frac32-\frac5{7/2}}\norm{P_Nf}_{L^2}
 =C N^{1/14}\norm{P_Nf}_{L^2}.
\end{align*}
Applying $\Lambda^{7/8}$ on a frequency block contributes $N^{7/8}$, and hence
\begin{align*}
 \norm{\e^{-(\gamma-\ii\Delta)(t-a)}P_Nf}_{L^{7/2}(I;W^{7/8,7/2})}
 \le C N^{7/8+1/14}\norm{P_Nf}_{L^2}
 =C N^{53/56}\norm{P_Nf}_{L^2}.
\end{align*}
Then the spatial Littlewood--Paley square function theorem followed by Minkowski in time yields
\begin{align*}
 \norm{\e^{-(\gamma-\ii\Delta)(t-a)}f}_{L^{7/2}(I;W^{7/8,7/2})}\le C\left(\sum_N N^{2\cdot53/56}\norm{P_Nf}_{L^2}^2\right)^{1/2}\le C\norm f_{H^{53/56}},
\end{align*}
which is \eqref{eq:periodic-homogeneous}.

For \eqref{eq:periodic-inhomogeneous}, Minkowski's inequality and the preceding bound on every interval $[s,b]\subset I$ give
\begin{align*}
 &\norm{\int_a^t\e^{-(\gamma-\ii\Delta)(t-s)}F(s)\,\dd s}_{L^{7/2}(I;W^{7/8,7/2})}\\
 &\qquad\le\int_a^b
 \norm{\one_{\{t\ge s\}}\e^{-(\gamma-\ii\Delta)(t-s)}F(s)}_{L^{7/2}_t(I;W^{7/8,7/2})}\,\dd s\le C\int_a^b\norm{F(s)}_{H^{53/56}}\,\dd s.
\end{align*}
The $C(I;H^{53/56})$ bound follows from the semigroup bound and
Minkowski's inequality.

For $0\le s\le1$, the fractional Leibniz rule gives
\begin{align*}
 \norm{fgh}_{H^s}
 \le C\bigl(\norm f_{H^s}\norm g_{L^\infty}\norm h_{L^\infty}
 +\norm g_{H^s}\norm f_{L^\infty}\norm h_{L^\infty}
 +\norm h_{H^s}\norm f_{L^\infty}\norm g_{L^\infty}\bigr).
\end{align*}
Taking $(f,g,h)=(u,\bar u,u)$ gives the first estimate.  For the difference, write
\begin{align*}
 \mathsf N(v)-\mathsf N(u)=(v-u)|v|^2+u(\bar v-\bar u)v+|u|^2(v-u)
\end{align*}
and apply the same product rule term by term.  Grouping the terms in which the $H^s$ derivative falls on $v-u$ and those in which it falls on $u$ or $v$ gives exactly \eqref{eq:cubic-difference-hs}.
\end{proof}
Recall that the Hamiltonian and its damping dissipation are
\begin{align*}
 \cE(u)&=\frac12\norm{\nabla u}_{L^2}^2+\frac14\norm u_{L^4}^4,\\
 \mathsf D(u)&=\norm{\nabla u}_{L^2}^2+\norm u_{L^4}^4.
\end{align*}

\begin{proposition}\label{prop:energy-strichartz}
For the global $H^1$ solution $u$, one has
\begin{align}
 \cE(u_t)+\frac{\gamma}{2}\int_0^t\mathsf D(u_s)\,\dd s
 \leq \cE(u_0)+C_{\gamma,B}t+M_t,\label{eq:energy-ito}
\end{align}
where 
\begin{align}
 d\langle M\rangle_t\le C_B\bigl(1+\cE(u_t)^{3/2}\bigr)dt.\label{eq:energy-qv}
\end{align}
Moreover, for every $T,R>0$ and $0<r<\infty$,
\begin{align}\label{eq:energy-supremum-moments}
 \sup_{\cE(u_0)\le R}\E\sup_{t\le T}(1+\cE(u_t))^r<\infty,
\end{align}
and there are $C,N<\infty$ such that
\begin{align}\label{eq:block-strichartz-moments}
 \E\norm u_{L^{7/2}(0,T;W^{7/8,7/2})}^r
 \le C\bigl(1+\cE(u_0)\bigr)^N.
\end{align}
If $u_0$ has an invariant law, then \eqref{eq:block-strichartz-moments} holds on every block $[a,a+T]$, uniformly in $a$.
\end{proposition}
\begin{proof}
It follows from It\^o formula that 
\begin{align}\label{eq:ito-identity}
\begin{split}
\cE(u_t)+\gamma\int_0^t\mathsf D(u_s)\,\dd s
 &=\cE(u_0)
 +M_t\\
 &+\frac12\sum_{j=1}^m\int_0^t\left[\norm{\nabla b_j}_{L^2}^2
 +\int_{\T^3}\left(|u_s|^2|b_j|^2+2(\Rea(u_s\bar b_j))^2\right)\dd x\right]\dd s,
\end{split} 
\end{align}
where
\begin{align*}
 M_t=\sum_{j=1}^m\int_0^t(\nabla\cE(u_s),b_j)_\R\,\dd W_s^j.
\end{align*}
Indeed,
\begin{align*}
 \nabla\cE(u)=-\Delta u+\mathsf N(u),\qquad
 (\nabla\cE(u),u)_\R=\norm{\nabla u}_{L^2}^2+\norm u_{L^4}^4=\mathsf D(u),
\end{align*}
and the Hamiltonian vector field is $-\mathsf J\nabla\cE(u)$, so
\begin{align*}
 (\nabla\cE(u),-\mathsf J\nabla\cE(u))_\R=0.
\end{align*}
Moreover,
\begin{align*}
 D\mathsf N(u)h=2|u|^2h+u^2\bar h,\qquad
 (D\mathsf N(u)h,h)_\R
 =\int_{\T^3}\left(|u|^2|h|^2+2(\Rea(u\bar h))^2\right)\dd x,
\end{align*}
which gives the displayed It\^o correction. Note that 
\begin{align}\label{eq:ito-correction}
 \frac12\sum_{j=1}^m\left[\norm{\nabla b_j}_{L^2}^2
 +\int_{\T^3}\left(|u|^2|b_j|^2+2(\Rea(u\bar b_j))^2\right)\dd x\right]
 \le C_B(1+\norm u_{L^2}^2)
 \le\frac{\gamma}{2}\mathsf D(u)+C_{\gamma,B},
\end{align}
which implies \eqref{eq:energy-ito}.  Furthermore,
\begin{align*}
 |(\nabla\cE(u),b_j)_\R|
 \le\norm{\nabla u}_{L^2}\norm{\nabla b_j}_{L^2}
 +\norm{b_j}_{L^\infty}\norm u_{L^3}^3
 \le C_{b_j}\bigl(1+\cE(u)^{3/4}\bigr),
\end{align*}
and therefore
\begin{align*}
 \dd\langle M\rangle_t
 =\sum_{j=1}^m|(\nabla\cE(u_t),b_j)_\R|^2\,\dd t
 \le C_B\bigl(1+\cE(u_t)^{3/2}\bigr)\,\dd t,
\end{align*}
which proves \eqref{eq:energy-qv}.

For an integer $m\ge1$, It\^o's formula applied to $(1+\cE)^m$, together with the preceding drift bound and \eqref{eq:energy-qv}, gives
\begin{align*}
 (1+\cE(u_t))^m
 \le(1+\cE(u_0))^m+C_m\int_0^t(1+\cE(u_s))^m\,\dd s
 +m\int_0^t(1+\cE(u_s))^{m-1}\,\dd M_s.
\end{align*}
By the Burkholder--Davis--Gundy and Young inequalities,
\begin{align*}
 \E\sup_{s\le t}\left|\int_0^s(1+\cE(u_\tau))^{m-1}\,\dd M_\tau\right|
 \le\frac12\E\sup_{s\le t}(1+\cE(u_s))^m
 +C_m\int_0^t\E\sup_{\tau\le s}(1+\cE(u_\tau))^m\,\dd s.
\end{align*}
Gronwall's inequality therefore yields
\begin{align}\label{eq:finite-time-energy-bound}
 \E\sup_{t\le T}(1+\cE(u_t))^m\le C_{m,T}(1+\cE(u_0))^m.
\end{align}
Increasing an arbitrary $r>0$ to an integer power and using Jensen's inequality proves \eqref{eq:energy-supremum-moments}.

It remains to establish the finite block dispersive estimate.  We first work on $I_0=[a,a+1]$ and set
\begin{align*}
 Z_a(t)=\int_a^t\e^{-(\gamma-\ii\Delta)(t-s)}B\,\dd W_s,\quad
 H_a=\sup_{t\in I_0}(1+\cE(u_t)),\quad
 R_a=1+H_a^{1/2}+\norm{Z_a}_{C(I_0;H^2)}.
\end{align*}
Conditionally on $\mathcal F_a$, the fresh convolution $Z_a$ has Gaussian moments of every finite order with deterministic bounds.  Let $I=[s,s+\ell]\subset I_0$.  Restarting the mild equation at $s$, using \Cref{lem:periodic-toolkit}, and noting that
\begin{align*}
 \int_s^t\e^{-(\gamma-\ii\Delta)(t-r)}B\,\dd W_r
 =Z_a(t)-\e^{-(\gamma-\ii\Delta)(t-s)}Z_a(s),
\end{align*}
we obtain
\begin{align*}
 \norm u_{C(I;H^1)}+\norm u_{L^{7/2}(I;W^{7/8,7/2})}
 \le CR_a+C\norm{\mathsf N(u)}_{L^1(I;H^1)}.
\end{align*}
Since $W^{7/8,7/2}\hookrightarrow L^\infty$, the cubic estimate in \Cref{lem:periodic-toolkit} and H\"older's inequality give
\begin{align*}
 \norm{\mathsf N(u)}_{L^1(I;H^1)}
 \le C\ell^{3/7}\norm u_{C(I;H^1)}
 \norm u_{L^{7/2}(I;L^\infty)}^2.
\end{align*}
Hence
\begin{align}\label{eq:block-local-bootstrap}
 \norm u_{C(I;H^1)}+\norm u_{L^{7/2}(I;W^{7/8,7/2})}
 \le CR_a+C\ell^{3/7}\norm u_{C(I;H^1)}
 \norm u_{L^{7/2}(I;W^{7/8,7/2})}^2.
\end{align}
Choosing $\ell\le cR_a^{-14/3}$ and applying the continuity bootstrap bounds the left-hand side by $2CR_a$.  The interval $I_0$ can be covered by at most $C(1+R_a^{14/3})$ such subintervals. 
The conditional version of \eqref{eq:finite-time-energy-bound}, applied from time $a$, together with the Gaussian moments of $Z_a$, shows that every finite moment of the right-hand side is bounded by a polynomial in $1+\cE(u_a)$.  Iterating over the finitely many unit intervals intersecting $[0,T]$ and using \eqref{eq:finite-time-energy-bound} proves \eqref{eq:block-strichartz-moments}.  If the initial law is invariant, the same estimate on $[a,a+T]$ is uniform in $a$ by stationarity.
\end{proof}

The next proposition verifies \Cref{ass:poly-lyapunov}  for SNLS \eqref{eq:intro-spde} with $\cV=1+\cE$ and $V_r=1+\cV^r$. 
\begin{proposition}\label{prop:lyapunov}
For every $r\ge1$ there are $c_r,C_r>0$ such that
\begin{align}
 \E(1+\cE(u_t))^r
 &\le \e^{-c_rt}(1+\cE(u_0))^r+C_r,\label{eq:polynomial-foster}\\
 \frac1T\int_0^T\E(1+\cE(u_s))^r\,\dd s
 &\le C_r+\frac{(1+\cE(u_0))^r}{c_rT}.\label{eq:polynomial-time-average}
\end{align}
\end{proposition}

\begin{proof}
Put $Y=1+\cE$.  By \eqref{eq:ito-identity} and \eqref{eq:ito-correction}, for every $\varepsilon>0$,
\begin{align*}
 \dd Y_t\le\bigl\{-(\gamma-\varepsilon)\mathsf D(u_t)+C_{\varepsilon,B}\bigr\}\,\dd t+\dd M_t.
\end{align*}
Applying It\^o's formula to $Y^r$ and using \eqref{eq:energy-qv} gives
\begin{align*}
 \dd Y_t^r\le-r(\gamma-\varepsilon)Y_t^{r-1}\mathsf D(u_t)\,\dd t
 +C_r\bigl(Y_t^{r-1}+Y_t^{r-1/2}\bigr)\,\dd t+rY_t^{r-1}\,\dd M_t.
\end{align*}
Choose $\varepsilon<\gamma/2$.  Since $\mathsf D\ge2\cE=2(Y-1)$, Young's inequality yields $c_r,C_r>0$ such that
\begin{align*}
 -r(\gamma-\varepsilon)Y^{r-1}\mathsf D+C_r\bigl(Y^{r-1}+Y^{r-1/2}\bigr)\le-c_rY^r+C_r.
\end{align*}
Hence
\begin{align*}
 \dd Y_t^r\le(-c_rY_t^r+C_r)\,\dd t+rY_t^{r-1}\,\dd M_t.
\end{align*}
Taking expectations and applying Gronwall's inequality prove \eqref{eq:polynomial-foster}.  Integrating the same differential inequality gives
\begin{align*}
 c_r\int_0^T\E Y_s^r\,\dd s\le Y_0^r+C_rT,
\end{align*}
which proves \eqref{eq:polynomial-time-average}.

\end{proof}

Existence of invariant measures for damped stochastic NLS is standard in closely related settings; see \cite{EkrenKukavicaZiane2017} for additive noise equations on $\R^d$ and \cite{BrzezniakFerrarioZanella2023} for subsequent ergodic results in dimensions $d\le3$.  On compact domains, the weak topology Krylov--Bogoliubov method is developed for damped 2D SNLS in \cite{BrzezniakFerrarioZanella2024}, following the sequential weak-Feller principle of \cite{MaslowskiSeidler1999}. Since the available compact domain result is two-dimensional and does not
directly cover the present three-dimensional energy space setting, we include
a short equation specific verification.

\begin{proposition}
\label{prop:existence}
The semigroup $P_t$ is Feller in the norm topology and sequentially weakly Feller on $H^1$.  It admits at least one invariant probability measure.  Every invariant probability $\nu$ satisfies
\begin{align}\label{eq:invariant-moments}
 \int_{H^1}(1+\cE(u))^r\,\nu(\dd u)<\infty,
 \qquad 0<r<\infty.
\end{align}
\end{proposition}

\begin{proof}
Norm Feller continuity follows from the $H^1$ well-posedness and
continuous dependence.  To prove the sequential weak Feller property,
let $u_0^n\rightharpoonup u_0$ in $H^1$ and denote the corresponding
solutions by $u^n,u$. Rellich compactness gives 
\begin{align*}
 \norm{u_0^n-u_0}_{L^2}\longrightarrow0.
\end{align*}
For $R\ge1$, set
\begin{align*}
 \Omega_R^n=\left\{
 \norm{u^n}_{\mathcal S(0,t)}+\norm u_{\mathcal S(0,t)}\le R
 \right\},
\end{align*}
where $\mathcal S$ is the energy--Strichartz norm defined in
\eqref{eq:energy-strichartz-space}.
The estimates \eqref{eq:energy-supremum-moments} and
\eqref{eq:block-strichartz-moments} give
\begin{align*}
 \lim_{R\to\infty}\sup_n\Prob((\Omega_R^n)^c)=0,
\end{align*}
while an $L^2$ difference estimate 
gives on $\Omega_R^n$
\begin{align*}
 \norm{u_t^n-u_t}_{L^2}
 \le C_{t,R}\norm{u_0^n-u_0}_{L^2}\longrightarrow0.
\end{align*}
On every bounded $H^1$-ball, the strong $L^2$ topology coincides with
the weak $H^1$ topology.  Hence, if $f$ is bounded and sequentially
weakly continuous,
\begin{align*}
 \abs{P_tf(u_0^n)-P_tf(u_0)}
 \le \omega_{f,R}\!\left(
 C_{t,R}\norm{u_0^n-u_0}_{L^2}\right)
 +2\norm f_{L^\infty}\Prob((\Omega_R^n)^c),
\end{align*}
where $\omega_{f,R}(s)\to0$ as $s\downarrow0$.  Letting first
$n\to\infty$ and then $R\to\infty$ proves
$P_tf(u_0^n)\to P_tf(u_0)$.

For fixed $u_0\in H^1$, set
\begin{align*}
 \nu_T=\frac1T\int_0^T\delta_{u_0}P_t\,\dd t.
\end{align*}
By \eqref{eq:polynomial-time-average} with $r=1$ and \eqref{eq:energy-coercivity}, the family
$(\nu_T)_{T\ge1}$ is tight in the weak $H^1$ topology.  The weak topology Krylov--Bogoliubov theorem of \cite{MaslowskiSeidler1999}, in the form used for damped SNLS in \cite{BrzezniakFerrarioZanella2024}, therefore yields an invariant subsequential limit.

Finally, \eqref{eq:invariant-moments} follows directly
from \eqref{eq:polynomial-foster} and invariance by standard 
truncation argument, see \cite[Proposition~A.3]{LiuLuWave2026} for example.
\end{proof}

\subsection{Regular estimates of variations}
Next we derive  the weak tangent bounds and the regular dependence requirements of \Cref{ass:stable-compact-regularity}.  

Along a fixed energy path $u$, define the real-linear operator
\begin{align}\label{eq:linearized-operator}
 L_u(t)\xi=-\gamma\xi+\ii\Delta\xi
 -\ii\bigl(2|u(t)|^2\xi+u(t)^2\bar\xi\bigr).
\end{align}
Because $u\in L^2(0,T;L^\infty)$, the non-autonomous equation
\begin{align}\label{eq:J-homogeneous}
\partial_t\xi=L_u(t)\xi
\end{align}
has a unique propagator
$J^u_{s,t}$ on real $L^2$.

\begin{lemma}
For $0\le s\le t\le T$, one has 
\begin{align}\label{eq:l2-tangent-bound}
 \norm{J^u_{s,t}}_{\cL(L^2)}
 \le\exp\left\{-\gamma(t-s)
 +\int_s^t\norm{u(r)}_{L^\infty}^2\,\dd r\right\}.
\end{align}
If $f\in L^1(s,T;L^2)$ and $\xi$ solves
$\partial_t\xi=L_u(t)\xi+f$, then
\begin{align}\label{eq:l2-inhomogeneous}
 \sup_{t\in[s,T]}\norm{\xi(t)}_{L^2}
 \le C_u\left(\norm{\xi(s)}_{L^2}+
 \norm f_{L^1(s,T;L^2)}\right),
\end{align}
where
\begin{align*}
 \log C_u\le\norm u_{L^2(s,T;L^\infty)}^2.
\end{align*}
\end{lemma}

\begin{proof}
Taking the real $L^2$ inner product of the homogeneous equation \eqref{eq:J-homogeneous} with $\xi$, the terms $\ii\Delta\xi$ and $-2\ii|u|^2\xi$ vanish, while
\begin{align*}
 \abs{(-\ii u^2\bar\xi,\xi)_\R}
 \le\norm u_{L^\infty}^2\norm\xi_{L^2}^2.
\end{align*}
Hence
\begin{align*}
 \frac12\frac{\dd}{\dd t}\norm\xi_{L^2}^2
 \le\bigl(-\gamma+\norm{u(t)}_{L^\infty}^2\bigr)\norm\xi_{L^2}^2.
\end{align*}
Gronwall's inequality gives
\begin{align*}
 \norm{J^u_{s,t}}_{\cL(L^2)}
 \le\exp\left\{-\gamma(t-s)+\int_s^t\norm{u(r)}_{L^\infty}^2\,\dd r\right\},
\end{align*}
which is \eqref{eq:l2-tangent-bound}.

For the inhomogeneous equation, variation of constants gives
\begin{align*}
 \xi(t)=J^u_{s,t}\xi(s)+\int_s^tJ^u_{r,t}f(r)\,\dd r.
\end{align*}
Using \eqref{eq:l2-tangent-bound} and discarding the damping factor,
\begin{align*}
 \sup_{t\in[s,T]}\norm{\xi(t)}_{L^2}
 \le \exp\left\{\int_s^T\norm{u(r)}_{L^\infty}^2\,\dd r\right\}
 \left(\norm{\xi(s)}_{L^2}+\norm f_{L^1(s,T;L^2)}\right).
\end{align*}
Thus \eqref{eq:l2-inhomogeneous} holds with
\begin{align*}
 C_u=\exp\left\{\int_s^T\norm{u(r)}_{L^\infty}^2\,\dd r\right\},
 \qquad
 \log C_u\le\norm u_{L^2(s,T;L^\infty)}^2.
\end{align*}
\end{proof}

For $(\xi,h)\in H^1\times H_T$, the joint first variation is
\begin{align}\label{eq:joint-first-variation}
 \partial_t r=L_u(t)r+B\dot h,\qquad r(0)=\xi.
\end{align}
Its endpoint is
\begin{align*}
 D_x\Phi_T(x,\omega)\xi+\cD\Phi_T(x,\omega)h.
\end{align*}
The state derivative, initially defined on $H^1$ directions, extends
to a bounded operator on $L^2$, although the nonlinear endpoint map
itself is defined only on the phase space $H^1$.

Fix $T>0$ and put $I=[0,T]$.
By the pathwise formulation \eqref{eq:pathwise-shift} and the standard deterministic stability argument, if $(x,\omega),(y,\widetilde\omega)\in\mathfrak D_T$,
\begin{align*}
 u=\boldsymbol\Phi_T(x,\omega),\quad
 v=\boldsymbol\Phi_T(y,\widetilde\omega),\quad \text{ and }
 \norm u_{\mathcal S(I)}+\norm v_{\mathcal S(I)}\le R,
\end{align*}
then one has 
\begin{align}\label{eq:Lipschitz-stability}
 \norm{u-v}_{\mathcal S(I)}
 \le C_{T,R,B}\left(
 \norm{x-y}_{H^1}
 +\norm{B(\omega-\widetilde\omega)}_{C(I;H^2)}
 \right)
\end{align}
implying that the solution map is continuous from $\mathfrak D_T$ into $\mathcal S(0,T)$.
The corresponding tangent maps are Borel and continuous on such tubes by their mild equations.

The next lemma provides a standard local stability consequence of the energy-subcritical theory.  On a bounded Lyapunov core, $L^2$ closeness of the initial data upgrades by interpolation to the subcritical Sobolev topology, yielding a common energy--Strichartz tube around every regular driver.
\begin{lemma}
\label{lem:regular-core-local-tube}
Fix $R<\infty$, $(x_0,\omega_0)\in\mathfrak D_T$ with $x_0\in\mathbb V_R$, and put
\begin{align*}
 \sigma=\frac{53}{56},\qquad \vartheta=1-\sigma=\frac3{56}.
\end{align*}
There are a relative $L^2$-neighborhood $U$ of $x_0$ in $\mathbb V_R$, a neighborhood $O$ of $\omega_0$ in $E_T$, and $R_1<\infty$ such that
\begin{align}\label{eq:regular-core-common-tube}
 U\times O\subset\mathfrak D_T,
 \qquad \sup_{(x,\omega)\in U\times O}
 \norm{\boldsymbol\Phi_T(x,\omega)}_{\mathcal S(0,T)}\le R_1.
\end{align}
After shrinking $U$ and $O$ if necessary, for $(x,\omega),(a',\eta')\in U\times O$ one has
\begin{align}\label{eq:local-weak-path-stability}
 &\norm{\boldsymbol\Phi_T(x,\omega)-\boldsymbol\Phi_T(a',\eta')}
 _{C(0,T;H^\sigma)\cap L^{7/2}(0,T;W^{7/8,7/2})}\notag\\
 &\qquad\le C_{T,R_1,B}
 \bigl(\norm{x-a'}_{H^\sigma}+\norm{\omega-\eta'}_{E_T}\bigr).
\end{align}
\end{lemma}

\begin{proof}
It follows from energy coercivity \eqref{eq:energy-coercivity}
and interpolation that 
\begin{align}\label{eq:regular-core-interpolation}
 \norm{x-x_0}_{H^\sigma}
 \le C_R\norm{x-x_0}_{L^2}^{\vartheta},
 \qquad x\in\mathbb V_R.
\end{align}
By
\Cref{lem:periodic-toolkit}, the shifted equation
\eqref{eq:pathwise-shift} is locally wellposed in $H^\sigma$. Denote for $I\subset [0,T]$
\begin{align*}
\|v\|_{\mathcal S^{\sigma}(I)} = \|v\|_{C(I;H^\sigma)}+ \norm{v}_{L^{7/2}(I;W^{7/8,7/2})}. 
\end{align*}
Choose $\ell_0\in(0,1]$ so small that
\begin{align*}
 C\ell_0^{3/7}
 \left(1+
 \norm{\boldsymbol\Phi_T(x_0,\omega_0)}_{\mathcal S^{\sigma}(0,T)}^2 \right)
 \le\frac14,
\end{align*}
and partition $[0,T]$ into finitely many intervals of length at most
$\ell_0$.
Let $\boldsymbol\Phi(x,\omega)$ denote the maximal $H^\sigma$ solution issued from $(x,\omega)$. For a subdivision interval $I=[s,s+\ell_0]$,  on  the bootstrap set $\norm{\boldsymbol\Phi(x,\omega)-\boldsymbol\Phi_T(x_0,\omega_0)}_{\mathcal S^{\sigma}(I)} \le1$, the cubic difference estimate \eqref{eq:cubic-difference-hs},
H\"older's inequality, and the choice of $\ell_0$ give
\begin{align*}
\begin{split}
\norm{\boldsymbol\Phi(x,\omega)-\boldsymbol\Phi_T(x_0,\omega_0)}_{\mathcal S^{\sigma}(I)}
 &\le C\Bigl(
 \norm{\boldsymbol\Phi(x,\omega)(s)
 -\boldsymbol\Phi_T(x_0,\omega_0)(s)}_{H^\sigma}
 +\norm{\omega-\omega_0}_{E_T}\Bigr)\\
 &+\frac12 \norm{\boldsymbol\Phi(x,\omega)-\boldsymbol\Phi_T(x_0,\omega_0)}_{\mathcal S^{\sigma}(I)}.
\end{split}
\end{align*}
After absorption and finite iteration over the subdivision,
\begin{align}\label{eq:reference-global-hsigma-stability}
\norm{\boldsymbol\Phi(x,\omega)-\boldsymbol\Phi_T(x_0,\omega_0)}_{\mathcal S^{\sigma}(0,T)}\le C_{T,x_0,\omega_0,B}
 \bigl(\norm{x-x_0}_{H^\sigma}
 +\norm{\omega-\omega_0}_{E_T}\bigr)
\end{align}
whenever the right-hand side is sufficiently small.  By
\eqref{eq:regular-core-interpolation}, after shrinking a relative
$L^2$-neighborhood $U$ of $x_0$ in $\mathbb V_R$ and a neighborhood
$O$ of $\omega_0$ in $E_T$, all the corresponding maximal
$H^\sigma$ solutions reach time $T$ and remain uniformly bounded in
\begin{align}\label{eq:common-hsigma-tube}
 C(0,T;H^\sigma)\cap
 L^{7/2}(0,T;W^{7/8,7/2}).
\end{align}
In particular, by denoting the solutions as $\boldsymbol\Phi_T(x,\omega)$, one has 
\begin{align}\label{eq:common-l2linfty}
 \sup_{(x,\omega)\in U\times O}
 \norm{\boldsymbol\Phi_T(x,\omega)}_{L^2(0,T;L^\infty)}
 <\infty.
\end{align}
For $(x,\omega)\in U\times O$, persistence of
regularity for \eqref{eq:pathwise-shift} gives, on the maximal
$H^1$ lifespan,
\begin{align*}
 \norm{\boldsymbol\Phi_T(x,\omega)}_{C(0,t;H^1)}
 \le C_{R,O,B,T}+
 C_{B,T}\int_0^t
 \norm{\boldsymbol\Phi_T(x,\omega)(s)}_{L^\infty}^2
 \norm{\boldsymbol\Phi_T(x,\omega)(s)}_{H^1}\,\dd s.
\end{align*}
By \eqref{eq:common-l2linfty} and Gronwall's inequality,
\begin{align*}
 \sup_{(x,\omega)\in U\times O}
 \norm{\boldsymbol\Phi_T(x,\omega)}_{C(0,T;H^1)}<\infty.
\end{align*}
The $H^1$ blow-up alternative therefore shows that
$U\times O\subset\mathfrak D_T$.  Combining this estimate with
\eqref{eq:common-hsigma-tube} gives
\eqref{eq:regular-core-common-tube} for some $R_1<\infty$. Finally, \eqref{eq:local-weak-path-stability} follows by repeating the
same subdivision-and-absorption stability argument on the common tube
\eqref{eq:regular-core-common-tube}, now for the pair
$(x,\omega)$ and $(a',\eta')$.
\end{proof}

Set
\begin{align}\label{eq:snls-growth-functional}
 \mathcal G_T(x,\omega)=
 \begin{cases}
  1+\norm{\boldsymbol\Phi_T(x,\omega)}_{\mathcal S(0,T)},
  &(x,\omega)\in\mathfrak D_T,\\
  +\infty,&(x,\omega)\notin\mathfrak D_T.
 \end{cases}
\end{align}

\begin{proposition}
\label{prop:snls-regular-dependence}
For every $p>0$, there are
$C,N<\infty$ such that
\begin{align}\label{eq:snls-growth-moments}
 \E\mathcal G_T(x,W)^p\le C\bigl(1+\cE(x)\bigr)^N.
\end{align}
In addition, with $\mathbf r_T=\norm{x-y}_{L^2}+\norm h_{H_T}$, 
one has
\begin{align}\label{eq:snls-path-growth}
 \mathbf d_T\bigl((x,\omega),(y,\omega+\iota_Th)\bigr)
 \le \mathbf r_T
 \exp\!\left\{C_{B,T}
 \bigl(1+\mathcal G_T(x,\omega)
 +\mathcal G_T(y,\omega+\iota_Th)\bigr)^2\right\}
\end{align}
whenever the two growth functionals are finite.

Furthermore, for every $R<\infty$, $T>0$, finite-dimensional
$F\subset H_T^0$, $(x_0,y_0)\in\mathbb V_R^2$, and
$\boldsymbol{\gamma}_T$-almost every $\omega_0\in E_T$, the map
\begin{align*}
 (x,y,\omega,h)\longmapsto
 \bigl(K_{T,x,y,\omega},\Phi_T(y,\omega+\iota_Th)\bigr)
\end{align*}
is defined and continuous on a product neighborhood $\mathcal O$ of
$(x_0,y_0,\omega_0,0)$ in $\mathbb V_R^2\times E_T\times F$, is
continuously Fr\'echet differentiable in $h$, with derivative  
\begin{align*}
 \bigl(0,\cD\Phi_T(y,\omega+\iota_Th)|_F\bigr),
\end{align*}
jointly continuous on $\mathcal O$. 
\end{proposition}

\begin{proof}
The moment bound \eqref{eq:snls-growth-moments} follows from
\eqref{eq:energy-coercivity},
\eqref{eq:energy-supremum-moments}, and
\eqref{eq:block-strichartz-moments}.  We prove the remaining assertions
in three steps.

\emph{Step 1: weak path growth.}
Let
\begin{align*}
 u=\boldsymbol\Phi_T(x,\omega),\qquad
 v=\boldsymbol\Phi_T(y,\omega+\iota_Th),\qquad e=v-u.
\end{align*}
By \eqref{eq:nonlinear-decomposition},
\begin{align*}
 \partial_te=-\gamma e-\ii(-\Delta+\mathbf M_{\mathsf d_{u,v}})e
 -\ii \mathbf M_{\mathsf c_{u,v}}\bar e+B\dot h.
\end{align*}
Since $\mathsf d_{u,v}$ is real valued and
\begin{align*}
 \norm{\mathsf c_{u,v}(t)}_{L^\infty}
 \le C\bigl(\norm{u(t)}_{L^\infty}^2+\norm{v(t)}_{L^\infty}^2\bigr),
\end{align*}
the real $L^2$ energy estimate and Gronwall's inequality give
\begin{align*}
 \norm e_{C(0,T;L^2)}
 \le C_{B,T}\mathbf r_T
 \exp\left\{C\int_0^T
 \bigl(\norm{u(t)}_{L^\infty}^2+\norm{v(t)}_{L^\infty}^2\bigr)\,\dd t\right\}.
\end{align*}
Since $W^{7/8,7/2}\hookrightarrow L^\infty$,
\begin{align*}
 \int_0^T\norm{u(t)}_{L^\infty}^2\,\dd t
 \le T^{3/7}\norm u_{L^{7/2}(0,T;L^\infty)}^2
 \le C_T\mathcal G_T(x,\omega)^2,
\end{align*}
and similarly for $v$.  Absorbing the fixed prefactor into the
exponential proves \eqref{eq:snls-path-growth}.

\emph{Step 2: a weak one-jet estimate.}
Suppose $(x,\omega),(y,\omega+\iota_Th)\in\mathfrak D_T$ and, for some $R_1<\infty$, that
\begin{align*}
 \norm x_{H^1}+\norm y_{H^1}+\norm h_{H_T}
 +\norm u_{\mathcal S(0,T)}+\norm v_{\mathcal S(0,T)}\le R_1,
\end{align*}
and put
\begin{align*}
 \delta=\norm{y-x}_{L^2}+\norm h_{H_T}\le1.
\end{align*}
Step~1 gives
\begin{align}\label{eq:snls-weak-difference}
 \norm e_{C(0,T;L^2)}\le C_{T,R_1}\delta.
\end{align}
Since $\norm e_{C(0,T;H^1)}\le CR_1$, interpolation with
\begin{align*}
 \sigma=\frac{53}{56},\qquad \vartheta=1-\sigma=\frac3{56},
\end{align*}
gives
\begin{align*}
 \norm e_{C(0,T;H^\sigma)}
 \le C_{T,R_1}\delta^\vartheta.
\end{align*}

Let $I=[s,s+\ell]\subset[0,T]$.  Restarting the difference equation
at $s$ and using \Cref{lem:periodic-toolkit} gives
\begin{align*}
 \norm e_{L^{7/2}(I;W^{7/8,7/2})}
 \le C\norm{e(s)}_{H^\sigma}
 +C\norm{B\dot h}_{L^1(I;H^\sigma)}
 +C\norm{\mathsf N(v)-\mathsf N(u)}_{L^1(I;H^\sigma)}.
\end{align*}
By \eqref{eq:cubic-difference-hs}, the preceding $H^\sigma$ bound,
and H\"older's inequality,
\begin{align*}
 \norm{\mathsf N(v)-\mathsf N(u)}_{L^1(I;H^\sigma)}
 \le C_{T,R_1}\delta^\vartheta
 +C_{R_1}\ell^{3/7}\norm e_{L^{7/2}(I;L^\infty)}.
\end{align*}
Moreover,
\begin{align*}
 \norm{B\dot h}_{L^1(I;H^\sigma)}
 \le C_{B,T}\delta\le C_{B,T}\delta^\vartheta.
\end{align*}
Using $W^{7/8,7/2}\hookrightarrow L^\infty$ and subdividing
$[0,T]$ into finitely many intervals for which
$C_{R_1}\ell^{3/7}\le1/2$, we obtain
\begin{align}\label{eq:snls-weak-strichartz}
 \norm e_{L^{7/2}(0,T;L^\infty)}
 \le C_{T,R_1}\delta^\vartheta.
\end{align}

Let $r$ solve \eqref{eq:joint-first-variation} with $\xi=y-x$ and
set $\rho=e-r$.  Then
\begin{align*}
 \partial_t\rho=L_u(t)\rho-\ii\mathscr N(u,e),\qquad \rho(0)=0,
\end{align*}
where
\begin{align*}
 \mathscr N(u,e)
 =\mathsf N(u+e)-\mathsf N(u)-D\mathsf N(u)e,\qquad
 |\mathscr N(u,e)|
 \le C\bigl(|u||e|^2+|e|^3\bigr).
\end{align*}
Thus \eqref{eq:snls-weak-difference},
\eqref{eq:snls-weak-strichartz}, and H\"older's inequality give
\begin{align*}
 \norm{\mathscr N(u,e)}_{L^1(0,T;L^2)}
 &\le CT^{3/7}\norm e_{C L^2}
 \left(\norm u_{L^{7/2}L^\infty}\norm e_{L^{7/2}L^\infty}
 +\norm e_{L^{7/2}L^\infty}^2\right)\\
 &\le C_{T,R_1}\delta^{1+\vartheta}.
\end{align*}
Hence \eqref{eq:l2-inhomogeneous} yields
\begin{align}\label{eq:snls-onejet-remainder}
 &\norm{\Phi_T(y,\omega+\iota_Th)-\Phi_T(x,\omega)
 -D_x\Phi_T(x,\omega)(y-x)-\cD\Phi_T(x,\omega)h}_{L^2}
 \le C_{T,R_1}\delta^{1+\vartheta}.
\end{align}

\emph{Step 3: local Cameron--Martin regularity.}
Fix $R<\infty$, a finite-dimensional $F\subset H_T^0$,
$(x_0,y_0)\in\mathbb V_R^2$, and a driver $\omega_0$ such that
$(x_0,\omega_0),(y_0,\omega_0)\in\mathfrak D_T$.  By
\Cref{lem:regular-core-local-tube}, applied at both reference
trajectories, there are a relative $L^2\times L^2$ neighborhood $D$
of $(x_0,y_0)$, a neighborhood $O$ of $\omega_0$ in $E_T$,
$r_*>0$, and $R_1<\infty$ such that
\begin{align*}
 \boldsymbol\Phi_T(x,\omega),\qquad
 \boldsymbol\Phi_T(y,\omega),\qquad
 \boldsymbol\Phi_T(y,\omega+\iota_Th)
\end{align*}
are genuine solutions in a common energy--Strichartz tube whenever
$(x,y)\in D$, $\omega\in O$, $h\in F$, and
$\norm h_{H_T}<r_*$.  Here we used
\begin{align*}
 \norm{\omega+\iota_Th-\omega_0}_{E_T}
 \le\norm{\omega-\omega_0}_{E_T}+C_T\norm h_{H_T}.
\end{align*}
The same lemma shows that these three trajectory maps depend
continuously on $(x,y,\omega,h)$ in
\begin{align*}
 C([0,T];H^{53/56})\cap L^{7/2}(0,T;L^\infty).
\end{align*}

Consequently
$(u,v)\mapsto(\mathsf d_{u,v},\mathsf c_{u,v})$ is continuous into
$L^1(0,T;L^\infty)^2$.  The operator-norm stability in
\Cref{lem:time-dependent-bog}, together with
\Cref{prop:exact-stable-compact-difference}, therefore gives
operator-norm continuity of $K_{T,x,y,\omega}$.

We next prove differentiability in the Cameron--Martin variable.
For $k\in F$ sufficiently small, apply
\eqref{eq:snls-onejet-remainder} with identical initial states,
base driver $\omega+\iota_Th$, and Cameron--Martin increment $k$.
The common energy--Strichartz bound gives
\begin{align*}
 &\norm{\Phi_T(y,\omega+\iota_T(h+k))
 -\Phi_T(y,\omega+\iota_Th)
 -\cD\Phi_T(y,\omega+\iota_Th)k}_{L^2}
 \le C\norm k_{H_T}^{1+\vartheta}.
\end{align*}
Hence
\begin{align*}
 D_h\Phi_T(y,\omega+\iota_Th)
 =\cD\Phi_T(y,\omega+\iota_Th)|_F,
\end{align*}
so $h\mapsto\Phi_T(y,\omega+\iota_Th)$ is Fr\'echet differentiable
on $F$.

It remains to prove continuity of this derivative.  If
$w=\boldsymbol\Phi_T(y,\eta)$, then
\begin{align*}
 \cD\Phi_T(y,\eta)k
 =\int_0^T J^w_{s,T}B\dot k(s)\,\dd s.
\end{align*}
For two paths $w,w'$ in the common tube,
\begin{align*}
 \int_0^T\norm{L_w(t)-L_{w'}(t)}_{\cL(L^2)}\,\dd t
 \le C_{T,R_1}\norm{w-w'}_{L^{7/2}(0,T;L^\infty)}.
\end{align*}
Duhamel's formula and \eqref{eq:l2-tangent-bound} therefore imply
\begin{align*}
 \sup_{0\le s\le t\le T}
 \norm{J^w_{s,t}-J^{w'}_{s,t}}_{\cL(L^2)}
 \longrightarrow0
\end{align*}
whenever $w'\to w$ in $L^{7/2}(0,T;L^\infty)$.  Hence
\begin{align*}
 (y,\omega,h)\longmapsto
 \cD\Phi_T(y,\omega+\iota_Th)|_F
\end{align*}
is jointly continuous in operator norm.  This proves the required
continuous Fr\'echet differentiability in $h$ and completes the proof.
\end{proof}

\section{A covariation lemma and short time expansion}

This appendix collects the probabilistic and short time analytic estimates used in the reduced Malliavin propagation.  

\subsection{A covariation lemma}

The following lemma extracts the predictable Brownian coefficient from a Hilbert valued process whose short increments admit an adapted expansion. We first give some notations. Let $I=[a,b]$, $\tau\le b$ be a stopping time, and $Z$ be an adapted continuous $H_0$-valued process, where $H_0$ is a separable Hilbert space.  Let
\begin{align*}
 \Pi_n=\{a=t_0<t_1<\cdots<t_{N_n}=b\}
\end{align*}
be uniform deterministic partitions of $I$ with mesh
$h_n=(b-a)/N_n\downarrow0$.  For each $n$, we suppress the partition
index $n$ for convenience and put
\begin{align*}
 \chi_i=\one_{\{t_i<\tau\}},\qquad
 \Delta_iW^\ell=W^\ell_{t_{i+1}}-W^\ell_{t_i}
\end{align*}
where $W^{\ell}, 1\leq \ell\leq m$ are independent Brownian motions. 

\begin{lemma}
\label{lem:stopped-hilbert-covariation}
Suppose that there are $\mathcal F_{t_i}$-measurable
$A_i,C_i^\ell\in H_0$, $1\le\ell\le m$, and remainders
$\rho_i\in H_0$ such that, on $\{t_{i+1}\le\tau\}$,
\begin{align*}
 Z_{t_{i+1}}-Z_{t_i}
 =A_i+\sum_{\ell=1}^mC_i^\ell\Delta_iW^\ell+\rho_i.
\end{align*}
Assume that there are measurable $H_0$-valued processes
$C^\ell=(C_t^\ell)_{t\in I}$ such that
\begin{align}
 &\sup_{n,i,\ell}\norm{\chi_iC_i^\ell}_{L^\infty(\Omega;H_0)}<\infty,
 \label{eq:stopped-c-bound}\\
 &\E\max_i\chi_i\norm{A_i}_{H_0}^2\longrightarrow0,
 \label{eq:stopped-a-bound}\\
 &\sum_i\chi_iC_i^\ell h_n
 \longrightarrow\int_I\one_{\{t<\tau\}}C_t^\ell\,\dd t
 \quad\text{in }L^2(\Omega;H_0),
 \label{eq:stopped-compensator}\\
 &h_n^{1/2}\sum_i
 \left(\E\one_{\{t_{i+1}\le\tau\}}\norm{\rho_i}_{H_0}^2\right)^{1/2}
 \longrightarrow0.
 \label{eq:stopped-remainder}
\end{align}
Then, for every Brownian coordinate $k$,
\begin{align}\label{eq:stopped-covariation-limit}
 \sum_i\bigl(Z_{t_{i+1}\wedge\tau}-Z_{t_i\wedge\tau}\bigr)\Delta_iW^k
 \longrightarrow\int_I\one_{\{t<\tau\}}C_t^k\,\dd t
\end{align}
in probability in $H_0$.
\end{lemma}

\begin{proof}
Set
\begin{align*}
 S_n=\sum_i\chi_i\left(A_i+\sum_{\ell=1}^mC_i^\ell\Delta_iW^\ell\right)\Delta_iW^k.
\end{align*}
Since $\chi_iA_i$ and $\chi_iC_i^\ell$ are
$\mathcal F_{t_i}$-measurable, conditional centering and
Hilbert-space martingale orthogonality give
\begin{align*}
 \E\norm{\sum_i\chi_iA_i\Delta_iW^k}_{H_0}^2
 &=h_n\sum_i\E[\chi_i\norm{A_i}_{H_0}^2]
 \le |I|\,\E\max_i\chi_i\norm{A_i}_{H_0}^2\longrightarrow0,
\end{align*}
and
\begin{align*}
 \E\norm{\sum_{i,\ell}\chi_iC_i^\ell
 (\Delta_iW^\ell\Delta_iW^k-\delta_{\ell k}h_n)}_{H_0}^2
 \le C_m\sup_{n,i,\ell}\norm{\chi_iC_i^\ell}_{L^\infty(\Omega;H_0)}^2N_nh_n^2
 \le C_{m,I}h_n\longrightarrow0.
\end{align*}
Together with \eqref{eq:stopped-compensator}, this yields
\begin{align}\label{eq:predictable-covariation-limit}
 S_n\longrightarrow\int_I\one_{\{t<\tau\}}C_t^k\,\dd t
\end{align}
in $L^2(\Omega;H_0)$.

Denote the stopped covariation sum by
\begin{align*}
 T_n=\sum_i\bigl(Z_{t_{i+1}\wedge\tau}-Z_{t_i\wedge\tau}\bigr)\Delta_iW^k.
\end{align*}
We show that $T_n-S_n\to0$ in probability.  On cells with $t_{i+1}\le\tau$, one has $\chi_i=1$ and 
\begin{align*}
 Z_{t_{i+1}\wedge\tau}-Z_{t_i\wedge\tau}
 =A_i+\sum_{\ell=1}^mC_i^\ell\Delta_iW^\ell+\rho_i,
\end{align*}
so their contribution to $T_n-S_n$ is
$\rho_i\Delta_iW^k$, and by \eqref{eq:stopped-remainder}
\begin{align*}
 \E\sum_i\one_{\{t_{i+1}\le\tau\}}
 \norm{\rho_i}_{H_0}\abs{\Delta_iW^k}
 \le h_n^{1/2}\sum_i
 \left(\E\one_{\{t_{i+1}\le\tau\}}\norm{\rho_i}_{H_0}^2\right)^{1/2}
 \longrightarrow0.
\end{align*}

On cells with $t_i\ge\tau$, both terms vanish: the stopped increment
is zero and $\chi_i=0$.  Thus only the unique possible crossing cell
$t_{i_*}<\tau<t_{i_*+1}$ remains.  Its contribution to $T_n$ is bounded by
\begin{align*}
 \max_i\norm{Z_{t_{i+1}\wedge\tau}-Z_{t_i\wedge\tau}}_{H_0}
 \max_i\abs{\Delta_iW^k}\longrightarrow0
\end{align*}
almost surely by uniform continuity of $Z_{\cdot\wedge\tau}$ and $W$.
The corresponding contribution to $S_n$ consists of the $A$-term,
which tends to zero in probability by \eqref{eq:stopped-a-bound}, and
the $C$-term, which is bounded by
\begin{align*}
 C\max_i\abs{\Delta_iW}^2\longrightarrow0
\end{align*}
almost surely by \eqref{eq:stopped-c-bound}.  Hence $T_n-S_n\to0$ in
probability, proving \eqref{eq:stopped-covariation-limit} together with \eqref{eq:predictable-covariation-limit}.
\end{proof}

\subsection{Short time expansion for Malliavin propagation}\label{subsec:malliavin-freezing}

We next justify the frozen noise expansions used in \Cref{subsec:reduced-malliavin}.  The following stochastic convolution estimate is standard. Throughout this subsection, write $\E_s[\cdot]=\E[\cdot\mid\mathcal F_s]$.

\begin{lemma}
Define
\begin{align*}
 Z^s_r=\sum_{\ell=1}^m\int_s^r
 \e^{-(\gamma-\ii\Delta)(r-a)}b_\ell\,\dd W^\ell_a,
 \qquad s\le r\le s+h.
\end{align*}
For every finite $p\ge2$ and $0<h\le1$,
\begin{align}
 \bigl(\E_s\norm{Z^s}_{\cS(s,s+h)}^p\bigr)^{1/p}
 &\le C_ph^{1/2},\label{eq:short-convolution-s}\\
 \bigl(\E_s\norm{Z^s_{s+h}-B\Delta_sW}_{H^1}^p\bigr)^{1/p}
 &\le C_ph^{3/2}.\label{eq:short-convolution-endpoint}
\end{align}
The constants are deterministic and independent of $s$.
\end{lemma}

\begin{proof}
It is enough to consider one profile.  Stochastic integration by parts gives
\begin{align}\label{eq:short-convolution-ibp}
 \int_s^r\e^{-(\gamma-\ii\Delta)(r-a)}b_\ell\,\dd W^\ell_a
 =b_\ell(W^\ell_r-W^\ell_s)
 -\int_s^r(\gamma-\ii\Delta)\e^{-(\gamma-\ii\Delta)(r-a)}
 b_\ell(W^\ell_a-W^\ell_s)\,\dd a .
\end{align}
Since the profiles are smooth, uniformly for $0\le r-a\le1$ the deterministic factors on the right are bounded in both $H^1$ and $W^{7/8,7/2}$.  The maximal inequality therefore gives
\begin{align*}
 \left(\E_s\sup_{s\le r\le s+h}
 \norm{Z^s_r}_{H^1\cap W^{7/8,7/2}}^p\right)^{1/p}
 \le C_ph^{1/2},
\end{align*}
which proves \eqref{eq:short-convolution-s}.  At $r=s+h$, the leading term in \eqref{eq:short-convolution-ibp} is $b_\ell(W^\ell_{s+h}-W^\ell_s)$, while the remaining integral is bounded in $H^1$ by
\begin{align*}
 Ch\sup_{s\le a\le s+h}|W_a-W_s|.
\end{align*}
Its conditional $L^p$ norm is $O(h^{3/2})$, proving
\eqref{eq:short-convolution-endpoint}.
\end{proof}

We now derive the frozen noise estimates through a short time mild stochastic Taylor expansion
tailored to the energy level cubic SNLS.  The underlying principle of
extracting short time stochastic leading terms directly from the mild
formulation is classical; see, for example,
\cite{JentzenKloeden2010,DaPratoJentzenRockner2019}.
The exponents and the inverse linearization estimate used here are specific
to the present energy--Strichartz setting.

Let $\cS(I)$ be defined by \eqref{eq:energy-strichartz-space}.  We use the adapted stopping times
\begin{align}\label{eq:adapted-localization}
 \tau_R=\inf\left\{t\le T:\sup_{r\le t}\norm{u_r}_{H^1}
 +\left(\int_0^t\norm{u_r}_{W^{7/8,7/2}}^{7/2}\,\dd r\right)^{2/7}\ge R\right\}\wedge T .
\end{align}
Since the solution $u\in C([0,T];H^1)$ and $\norm{u}_{\cS(0,T)}<\infty$ almost surely, $\tau_R\uparrow T$ almost surely.  
For deterministic $s<t$, let $u^0_{s,r}$, $s\le r\le t$, denote the zero noise solution of \eqref{eq:intro-spde} started from $u_s$ at time $s$ so that $u^0_{s,r}$ is $\mathcal F_s$-measurable for every $r\ge s$, and let $\widetilde Q^0_{s,r}$ be its inverse linearized evolution.  For the stochastic path put
\begin{align*}
 \widetilde Q_{s,r}=J_{s,r}^{-1}=Q_s^{-1}Q_r,\qquad Q_r=Q_s\widetilde Q_{s,r},
\end{align*}
and write
\begin{align*}
 d_{s,r}=u_r-u^0_{s,r},\qquad \Delta_sW=W_t-W_s,\qquad h=t-s.
\end{align*}

\begin{lemma}\label{lem:frozen-estimates}
For every $R<\infty$ there is $h_0=h_0(R)\in(0,1]$ such that, for every finite $p\ge2$ and deterministic $s<t$ with $h=t-s\le h_0$,
\begin{align}
 \bigl(\E_s[\one_{\{t\le\tau_R\}}\norm{d_{s,\cdot}}_{\cS(s,t)}^p]\bigr)^{1/p}
 &\le C_{p,R}h^{1/2}\one_{\{s<\tau_R\}},\label{eq:frozen-s}\\
 \bigl(\E_s[\one_{\{t\le\tau_R\}}\norm{d_{s,t}-B\Delta_sW}_{H^1}^p]\bigr)^{1/p}
 &\le C_{p,R}h^{13/14}\one_{\{s<\tau_R\}},\label{eq:frozen-endpoint}\\
 \bigl(\E_s[\one_{\{t\le\tau_R\}}\norm{\widetilde Q_{s,t}-\widetilde Q^0_{s,t}}_{\cL(L^2)}^p]\bigr)^{1/p}
 &\le C_{p,R}h^{13/14}\one_{\{s<\tau_R\}}.\label{eq:frozen-inverse}
\end{align}
Here $C_{p,R}$ may also depend on the fixed parameters $T,\gamma$, and $B$.
\end{lemma}

\begin{proof}
The proof is divided into three steps. 

\emph{Step 1: short path difference.}
By \eqref{eq:adapted-localization}, on $\{s<\tau_R\}$ one has
$\norm{u_s}_{H^1}\le R$.  The deterministic version of the local
energy--Strichartz bootstrap leading to
\eqref{eq:block-local-bootstrap}, with the stochastic convolution set
to zero and using \Cref{lem:periodic-toolkit}, therefore gives
$h_0=h_0(R)\in(0,1]$ and $C_R<\infty$ such that
\begin{align}\label{eq:frozen-zero-noise-bound}
 \norm{u^0_{s,\cdot}}_{\cS(s,s+h_0)}\le C_R.
\end{align}
On $\{s+h\le\tau_R\}$, the definition
\eqref{eq:adapted-localization} also gives a deterministic bound for
$\norm u_{\cS(s,s+h)}$.  The mild difference equation is
\begin{align}\label{eq:frozen-difference-mild}
 d_{s,r}=Z^s_r-\ii\int_s^r
 \e^{-(\gamma-\ii\Delta)(r-a)}
 \{\mathsf N(u_a)-\mathsf N(u^0_{s,a})\}\,\dd a .
\end{align}
The $CH^1$ Duhamel estimate, \eqref{eq:periodic-inhomogeneous}, and
the cubic difference estimate \eqref{eq:cubic-difference-hs}, exactly
as in the derivation of \eqref{eq:block-local-bootstrap}, give
\begin{align}\label{eq:frozen-absorption}
 \one_{\{s+h\le\tau_R\}}\norm{d_{s,\cdot}}_{\cS(s,s+h)}
 \le C\one_{\{s+h\le\tau_R\}}\norm{Z^s}_{\cS(s,s+h)}
 +C_Rh^{3/7}\one_{\{s+h\le\tau_R\}}
 \norm{d_{s,\cdot}}_{\cS(s,s+h)} .
\end{align}
Indeed, the case $s=1$ of \eqref{eq:cubic-difference-hs} gives
\begin{align*}
 \norm{\mathsf N(u)-\mathsf N(u^0)}_{H^1}
 &\le C(\norm u_{L^\infty}^2+\norm{u^0}_{L^\infty}^2)
 \norm d_{H^1}\\
 &\quad+C(\norm u_{H^1}+\norm{u^0}_{H^1})
 (\norm u_{L^\infty}+\norm{u^0}_{L^\infty})
 \norm d_{L^\infty}.
\end{align*}
Using \eqref{eq:adapted-localization},
\eqref{eq:frozen-zero-noise-bound}, the embedding
$W^{7/8,7/2}\hookrightarrow L^\infty$, and H\"older's inequality,
\begin{align*}
 \int_s^{s+h}
 (\norm{u_a}_{L^\infty}^2+\norm{u^0_{s,a}}_{L^\infty}^2)\,\dd a
 \le C_Rh^{3/7},
\end{align*}
and
\begin{align*}
 \int_s^{s+h}
 (\norm{u_a}_{L^\infty}+\norm{u^0_{s,a}}_{L^\infty})
 \norm{d_{s,a}}_{L^\infty}\,\dd a
 \le C_Rh^{3/7}
 \norm{d_{s,\cdot}}_{L^{7/2}(s,s+h;L^\infty)}.
\end{align*}
This proves \eqref{eq:frozen-absorption}.  Decrease $h_0$ so that
$C_Rh_0^{3/7}\le1/2$.  Absorption, conditional $L^p$ norms, and
\eqref{eq:short-convolution-s} then give \eqref{eq:frozen-s}.

\emph{Step 2: endpoint refinement.}
Evaluating \eqref{eq:frozen-difference-mild} at $s+h$ and using the
same $H^1$ Duhamel and cubic difference estimates as in
\eqref{eq:frozen-absorption} gives
\begin{align*}
 \one_{\{s+h\le\tau_R\}}
 \norm{d_{s,s+h}-Z^s_{s+h}}_{H^1}
 \le C_Rh^{3/7}\one_{\{s+h\le\tau_R\}}
 \norm{d_{s,\cdot}}_{\cS(s,s+h)}.
\end{align*}
Hence \eqref{eq:short-convolution-endpoint} and
\eqref{eq:frozen-s} yield
\begin{align*}
 \bigl(\E_s[\one_{\{s+h\le\tau_R\}}
 \norm{d_{s,s+h}-B\Delta_sW}_{H^1}^p]\bigr)^{1/p}
 \le C_{p,R}(h^{3/7+1/2}+h^{3/2})
 \le C_{p,R}h^{13/14},
\end{align*}
which proves \eqref{eq:frozen-endpoint}.

\emph{Step 3: inverse linearization stability.}
By \eqref{eq:linearized-operator}, the tangent equation along a path
$v$ can be written as
\begin{align*}
 \partial_r\xi=A\xi+\mathbf M(v_r)\xi,\qquad
 A=-\gamma+\ii\Delta,\qquad
 \mathbf M(v)\xi=-\ii\bigl(2|v|^2\xi+v^2\bar\xi\bigr).
\end{align*}
Define the interaction propagator
\begin{align*}
 \mathcal P^v_{s,r}=\e^{-(r-s)A}J^v_{s,r}.
\end{align*}
Then, in the strong sense
after application to a fixed vector in $L^2$,
\begin{align*}
 \partial_r\mathcal P^v_{s,r}
 =\mathsf V^v_s(r)\mathcal P^v_{s,r},\qquad
 \mathsf V^v_s(r)
 =\e^{-(r-s)A}\mathbf M(v_r)\e^{(r-s)A}.
\end{align*}
As the damping factors cancel under conjugation and the Schr\"odinger
group is unitary, one has 
\begin{align*}
 \norm{\mathsf V^v_s(r)}_{\cL(L^2)}
 &\le C\norm{v_r}_{L^\infty}^2,\\
 \norm{\mathsf V^v_s(r)-\mathsf V^w_s(r)}_{\cL(L^2)}
 &\le C(\norm{v_r}_{L^\infty}+\norm{w_r}_{L^\infty})
 \norm{v_r-w_r}_{L^\infty}.
\end{align*}
Taking $v=u$ and $w=u^0_{s,\cdot}$, and using
\eqref{eq:adapted-localization},
\eqref{eq:frozen-zero-noise-bound}, and H\"older's inequality, gives
\begin{align}
 \int_s^{s+h}\bigl(\norm{\mathsf V^u_s(r)}_{\cL(L^2)}
 +\norm{\mathsf V^{u^0}_s(r)}_{\cL(L^2)}\bigr)\,\dd r
 &\le C_Rh^{3/7},\label{eq:frozen-m-bound}\\
 \int_s^{s+h}
 \norm{\mathsf V^u_s(r)-\mathsf V^{u^0}_s(r)}_{\cL(L^2)}\,\dd r
 &\le C_Rh^{3/7}
 \norm{d_{s,\cdot}}_{\cS(s,s+h)}.\label{eq:frozen-m-difference}
\end{align}

Let $\mathcal R^v_{s,r}=(\mathcal P^v_{s,r})^{-1}$. Then, again in
the strong sense
\begin{align}\label{eq:Rvst}
 \partial_r\mathcal R^v_{s,r}
 =-\mathcal R^v_{s,r}\mathsf V^v_s(r),\qquad
 \mathcal R^v_{s,s}=\Id.
\end{align}
Applying the strong Volterra equations to fixed vectors, taking the
supremum over the unit ball of $L^2$, and using
\eqref{eq:frozen-m-bound} and Gronwall's inequality give
\begin{align*}
 \norm{\mathcal R^u_{s,s+h}-\mathcal R^{u^0}_{s,s+h}}_{\cL(L^2)}
 \le C_R\int_s^{s+h}
 \norm{\mathsf V^u_s(r)-\mathsf V^{u^0}_s(r)}_{\cL(L^2)}\,\dd r.
\end{align*}
Since
\begin{align*}
 \widetilde Q_{s,s+h}
 =\mathcal R^u_{s,s+h}\e^{-hA},\qquad
 \widetilde Q^0_{s,s+h}
 =\mathcal R^{u^0}_{s,s+h}\e^{-hA},
\end{align*}
and $\norm{\e^{-hA}}_{\cL(L^2)}=\e^{\gamma h}$, it follows from
\eqref{eq:frozen-m-difference} that
\begin{align*}
 \one_{\{s+h\le\tau_R\}}
 \norm{\widetilde Q_{s,s+h}-\widetilde Q^0_{s,s+h}}_{\cL(L^2)}
 \le C_Rh^{3/7}\one_{\{s+h\le\tau_R\}}
 \norm{d_{s,\cdot}}_{\cS(s,s+h)}.
\end{align*}
Taking conditional $L^p$ norms and using \eqref{eq:frozen-s} gives
\begin{align*}
 \bigl(\E_s[\one_{\{s+h\le\tau_R\}}
 \norm{\widetilde Q_{s,s+h}-\widetilde Q^0_{s,s+h}}_{\cL(L^2)}^p]\bigr)^{1/p}
 \le C_{p,R}h^{3/7+1/2}
 =C_{p,R}h^{13/14},
\end{align*}
which proves \eqref{eq:frozen-inverse}.
\end{proof}

\section*{Statements and declarations}

\textbf{Competing interests}.
The authors declare that they have no competing interests.

\medskip

\textbf{Use of generative artificial intelligence}.
OpenAI ChatGPT  was used to assist with proof auditing and manuscript preparation. The authors take full responsibility for the mathematical content.

\end{document}